\documentclass[reqno]{amsart}

\usepackage[english]{babel}
\usepackage[utf8]{inputenc} 

\usepackage{amsthm,amsfonts,amssymb,amsmath}
\numberwithin{equation}{section}
\usepackage{dsfont}
\usepackage{multirow}
\usepackage{url} 
\usepackage{xcolor} 
\usepackage{graphics}
\usepackage{booktabs}
\usepackage{epsfig}
\usepackage{epstopdf}
\usepackage{psfrag}                
\usepackage{mathrsfs}
\usepackage{mathtools}             
\usepackage{subfigure}
\usepackage{nicefrac}
\usepackage{paralist}
\usepackage[shortlabels,inline]{enumitem}

\usepackage[pdftex,
	pdftitle={Regularity theory for a new class of fractional parabolic stochastic evolution equations},
	pdfauthor={K.~Kirchner and J.~Willems},
	bookmarksopen,
	colorlinks,
	linkcolor=black,
	urlcolor=black,
	citecolor=black
	]{hyperref}

\newcommand{\bbN}{\mathbb{N}}
\newcommand{\bbP}{\mathbb{P}}
\newcommand{\bbR}{\mathbb{R}}

\newcommand{\R}{\mathbb{R}}
\newcommand{\C}{\mathbb{C}}
\newcommand{\N}{\mathbb{N}}

\newcommand{\K}{\mathbb{K}}
\newcommand{\E}{\mathbb{E}}

\newcommand{\cA}{\mathcal{A}}
\newcommand{\cB}{\mathcal{B}}

\newcommand{\cD}{\mathcal{D}}

\newcommand{\cF}{\mathcal{F}}

\newcommand{\cI}{\mathcal{I}}

\newcommand{\cL}{\mathcal{L}}
\newcommand{\cM}{\mathcal{M}}

\newcommand{\cS}{\mathcal{S}}
\newcommand{\cT}{\mathcal{T}}

\newcommand{\cW}{\mathcal{W}}
\newcommand{\cX}{\mathcal{X}}

\newcommand{\norm}[2]{     \| #1       \|_{ #2 }}
\newcommand{\seminorm}[2]{  | #1        |_{ #2 }}

\newcommand{\scalar}[2]{     ( #1       )_{ #2 }}

\renewcommand{\Re}{\operatorname{Re}}
\renewcommand{\Im}{\operatorname{Im}}

\newcommand{\rd}{\mathop{}\!\mathrm{d}}
\newcommand{\from}{\colon} 
\newcommand{\clos}[1]{\overline{ #1 }} 

\newcommand{\dual}[1]{#1^*} 
\newcommand{\Hdot}[1]{\dot{H}^{#1}_A}

\newtheorem{lemma}{Lemma}[section]
\newtheorem{proposition}[lemma]{Proposition}
\newtheorem{theorem}[lemma]{Theorem}
\newtheorem{corollary}[lemma]{Corollary}

\theoremstyle{remark}
\newtheorem{remark}[lemma]{Remark}

\theoremstyle{definition}
\newtheorem{definition}[lemma]{Definition}
\newtheorem{assumption}[lemma]{Assumption}
\newtheorem{example}[lemma]{Example}

\begin{document}

\author{Kristin Kirchner \and Joshua Willems}

\address[Kristin Kirchner]{Delft Institute of Applied Mathematics\\
	Delft University of Technology\\
	P.O.~Box 5031 \\ 
	2600 GA Delft \\
	The Netherlands.}

\email{k.kirchner@tudelft.nl}

\address[Joshua Willems]{Delft Institute of Applied Mathematics\\
	Delft University of Technology\\
	P.O.~Box 5031 \\ 
	2600 GA Delft \\
	The Netherlands.}

\email{j.willems@tudelft.nl}


\title[Analysis of fractional parabolic stochastic evolution equations]{%
	Regularity theory for a new class of fractional \\ 
	parabolic stochastic evolution equations}




\date{}

\begin{abstract}
	A new class of fractional-order stochastic evolution equations
	of the form 
	$(\partial_t + A)^\gamma X(t) = \dot{W}^Q(t)$, 
	$t\in[0,T]$,
	$\gamma \in (0,\infty)$, 
	is introduced, 
	where $-A$ generates a $C_0$-semigroup 
	on a separable Hilbert space~$H$ 
	and 
	the spatiotemporal driving noise~$\dot{W}^Q$ is
	the formal time derivative of
	an $H$-valued cylindrical $Q$-Wiener process.
	Mild and weak solutions are defined; 
	these concepts are shown to be equivalent 
	and to lead to well-posed problems.
	Temporal and spatial regularity of the solution process 
	$X$ are investigated,
	the former being measured by
	mean-square or pathwise smoothness  
	and the latter by using
	domains of fractional powers of~$A$.
	In addition, the covariance of~$X$ 
	and its long-time behavior 
	are analyzed. 
	
	These abstract results are applied to 
	the cases when 
	$A := L^\beta$ and $Q:=\widetilde{L}^{-\alpha}$ 
	are fractional powers of 
	symmetric, 
	strongly elliptic second-order
	differential operators 
	defined on (i) bounded Euclidean domains 
	or (ii) smooth, compact surfaces.
	In these cases, the 
	Gaussian solution processes 
	can be seen as  
	generalizations of 
	merely spatial 
	{(Whittle--)}Mat\'ern fields 
	to space--time. 
\end{abstract}

\maketitle

\section{Introduction}\label{section:intro}

\subsection{Motivation and background} 

Gaussian processes play an important role 
for modeling in spatial statistics. 
Typical applications arise in the environmental sciences, 
where 
geographically indexed data is collected, 
including 
climatology~\cite{Alexeeff2018, Sang2011}, 
oceanography~\cite{Bertino2003}, 
meteorology~\cite{Handcock1994}, 
and 
forestry~\cite{Beguin2017, Jost2005, Matern1960spatial}. 
More generally, hierarchical models 
based on Gaussian processes 
have been used 
in various disciplines, 
where spatially dependent (or spatiotemporal) 
data is recorded,  
such as 
demography~\cite{DeIaco2015, Pereira2021},  
epidemiology~\cite{Lawson2014},  
finance~\cite{Fernandez2016},   
and 
neuroimaging~\cite{Mejia2020}.  

Since a Gaussian process $(X(j))_{j\in\cI}$ 
is fully characterized by its 
mean and its covariance function,  
\emph{second-order-based approaches}  
focus on the construction of 
appropriate covariance classes. 
In the case that the index set~$\cI$ is given by a 
spatial domain in the Euclidean space 
$\cI= \cD \subseteq \bbR^d$\!, 
the \emph{Mat\'ern covariance class} \cite{Matern1960spatial}
is an important and widely used model.  
The Mat\'ern covariance function is given by 
\begin{equation}\label{eq:Matern-Cov} 
\varrho(x,y)
=
2^{1-\nu}\sigma^2[\Gamma(\nu)]^{-1} 
(\kappa \norm{x-y}{\R^d})^\nu 
K_\nu(\kappa \norm{x-y}{\R^d}), 
\quad 
x,y\in \mathcal D, 
\end{equation}
where 
$K_\nu$ denotes 
the modified Bessel function of the second kind. 
It is indexed by  
the three interpretable parameters 
$\nu,\kappa,\sigma^2\in(0,\infty)$, 
which determine \emph{smoothness},
\emph{correlation length} and \emph{variance} 
of the process. 
It is this feature that 
renders the Mat\'ern class 
particularly suitable for making 
inference about spatial data~\cite{Stein1999interpolation}. 

When considering \emph{spatiotemporal} 
phenomena, the following two difficulties occur: 
\begin{enumerate}[label=\arabic*., leftmargin=1cm,topsep=1mm] 
	\item It is desirable to control the properties 
	of the stochastic process named above 
	(in particular, smoothness and correlation lengths) 
	separately in space and time. For this reason, 
	considering \eqref{eq:Matern-Cov} 
	in $d+1$ dimensions is not expedient 
	and it is a difficult task to construct appropriate 
	spatiotemporal covariance 
	models, see e.g.~\cite{Cressie1999, Fuentes2008, Gneiting2002,  
		Porcu2021, Porcu2007, Stein2005space}. 
	\item Second-order-based approaches require 
	the factorization of, in general, dense covariance matrices, 
	causing computational costs which 
	are
	cubic 
	in the number of observations. 
	The two common assumptions 
	imposed on spatiotemporal covariance models 
	to reduce the computational 
	costs---separability 
	(factorization into merely 
	spatial and temporal covariance functions) 
	and 
	stationarity 
	(invariance under translations)---have 
	proven unrealistic in many situations, 
	see~\cite{Cressie2011, Mateu2008, Stein2005space}. 
	In particular, Stein \cite{Stein2005space} 
	criticized the  
	behavior of separable covariance functions 
	with respect to their differentiability. 
\end{enumerate} 

Owing to these problems, 
the class of \emph{dynamical models} has 
gained popularity. 
The name originates from focusing on 
the dynamics of the stochastic process 
which are described 
either 
by means of 
conditional probability distributions 
or by representing 
the process as a solution 
of a stochastic partial differential equation (SPDE). 
The latter approach 
was originally proposed in the merely spatial case,  
motivated by the following observation made by 
Whittle \cite{Whittle1963}: 
A stationary process $(X(x))_{x\in\cD}$ 
indexed by the entire 
Euclidean space $\cD=\bbR^d$ 
which solves the 
SPDE 
\begin{equation}\label{eq:maternspde}
\bigl( \kappa^2-\Delta\bigr)^\beta X(x) 
= \cW(x), 
\quad 
x\in\cD, 
\end{equation}
has a covariance function 
of Mat\'ern type~\eqref{eq:Matern-Cov} 
with $\nu = 2\beta - \nicefrac{d}{2}$. 
Here, 
$\Delta$ denotes the
Laplacian and $\cW$ is Gaussian white noise.  
This relation gave rise to the SPDE approach
proposed by Lindgren, Rue, and Lindstr\"om~\cite{Lindgren2011explicit},  
where 
the SPDE~\eqref{eq:maternspde} is considered 
on a bounded domain $\cD\subsetneq\bbR^d$ 
and augmented with Dirichlet or Neumann boundary conditions.  
Besides enabling the applicability of efficient numerical 
methods available for (S)PDEs,  
such as 
finite element methods 
\cite{Bolin2020rational, Bolin2018weak, Bolin2020strong, 
	Cox2020regularity, Herrmann2020, Lindgren2011explicit} 
or wavelets \cite{Bolin2013wavelets, Harbrecht2021multilevel}, 
this approach has the 
advantage of allowing for 
\begin{enumerate}[label={(\alph*)}, leftmargin=1cm,topsep=1mm] 
	\item 
	nonstationary or anisotropic generalizations 
	by replacing the operator 
	$\kappa^2-\Delta$ 
	in \eqref{eq:maternspde} 
	with more general  
	strongly elliptic 
	second-order 
	differential operators, 
	\begin{equation}\label{eq:DO} 
	(L v)(x)
	= 
	\kappa^2(x) v(x)  
	- \nabla \cdot ( a(x) \, \nabla v(x) ), 
	\quad 
	x \in \cD, 
	\end{equation}
	where $\kappa\from \cD\to\bbR$ and  
	$a\from\cD\to\bbR_{\rm sym}^{d\times d}$ 
	are functions   
	\cite{Bakka2019nonstationary, 
		Bolin2020rational, Bolin2018weak, Bolin2020strong, 
		Cox2020regularity, 
		Fuglstad2015, Herrmann2020, Lindgren2011explicit}; 
	\item 
	more general domains,  
	such as surfaces~\cite{Bolin2011ozone, Herrmann2020} 
	or manifolds~\cite{Harbrecht2021multilevel}. 
\end{enumerate} 

In the SPDE \eqref{eq:maternspde} 
the fractional exponent 
$\beta$ defines the (spatial) differentiability
of its solution, see e.g.~\cite{Cox2020regularity}. A realistic  
description of spatiotemporal phenomena  
necessitates  
controllable differentiability in space and time. 
This motivates to consider 
the space--time fractional 
SPDE model 
\begin{equation}\label{eq:spde-fractional-parabolic} 
\begin{cases} 
\left( \partial_t + L^\beta \right)^\gamma X(t,x) 
=
\dot{\cW}(t,x), 
&	  
\; t \in [0,T], \quad x \in \cD, 
\\ 
\hspace{1.7cm}  
X(0,x) = X_0(x), 
&  
\; x\in\cD, 
\end{cases} 
\end{equation} 
where $L$ in \eqref{eq:DO} 
is augmented with 
boundary conditions on $\partial\cD$,  
$(X_0(x))_{x\in\cD}$ is the initial random field,   
$\dot{\mathcal{W}}$ denotes space--time Gaussian white noise, 
and $T \in (0,\infty)$ is the time horizon. 
Whenever $\beta=\gamma=1$, 
the SPDE \eqref{eq:spde-fractional-parabolic} simplifies to   
the stochastic heat equation 
and this spatiotemporal model 
had already been mentioned in~\cite{Lindgren2011explicit} 
and it was used for statistical inference 
in~\cite{Cameletti2013, Sigrist2015}.  
The novelty and sophistication of the  
SPDE model \eqref{eq:spde-fractional-parabolic} 
lies in the fractional power~$\gamma\in(0,\infty)$ 
of the parabolic operator. 
Notably, it is 
the interplay of the 
parameters 
$\beta$ and $\gamma$ 
that will facilitate controlling     
spatial and temporal smoothness 
of the solution process. 
For $\cD=\bbR^d$\!, 
this has recently been  
investigated  
via Fourier techniques 
in~\cite{Bakka2023diffusionbased}, 
see also~\cite{Angulo2008Spatiotemporal, Carrizo2018general, Kelbert2005}.  

Besides   
the aforementioned benefits 
of the SPDE approach and 
in contrast to 
the SPDE $\bigl( \partial_t^\gamma + L^\beta \bigr) X = \dot{\cW}$, 
considered for instance in~\cite{Bonaccorsi2009,Desch2011},  
the SPDE model~\eqref{eq:spde-fractional-parabolic}
furthermore exhibits a long-time behavior  
resembling the spatial model~\eqref{eq:maternspde}. 

\subsection{Contributions} 

We introduce a novel interpretation of \eqref{eq:spde-fractional-parabolic} 
with $X_0=0$ 
as a fractional parabolic stochastic evolution equation, 
and correspondingly define mild and weak solutions for it. 
To this end, we first give a meaning to fractional powers 
of an operator of the form $\partial_t + A$, 
where $-A$ generates a $C_0$-semigroup. 
Generalizing the approach taken for $\gamma=1$ 
in \cite[Chapter~5]{DaPrato2014}, 
we prove that mild and weak solutions 
are equivalent 
under natural assumptions, and we investigate 
their existence, uniqueness, 
regularity, and covariance. 
Our main findings are that 
the problem~\eqref{eq:spde-fractional-parabolic} 
is well-posed, and 
the properties of its solution $X$ 
with respect to  
smoothness   
and covariance structure 
generalize those 
of the spatial Whittle--Mat\'ern SPDE model \eqref{eq:maternspde} and 
relate to the parameters 
$\beta,\gamma\in(0,\infty)$ in the desired way. 
Restricting the analysis to a 
zero initial field is justified 
by our primary interest in 
regularity related to the dynamics 
of~\eqref{eq:spde-fractional-parabolic} and the 
long-time behavior of solutions. 

In comparison 
with~\cite{Biswas2021Harnack,Biswas2021,Litsgard2023,Nystrom2016,Stinga2017}---the 
only previous works on an equation 
of the form $(\partial_t + L)^\gamma u = f$ 
known to the authors---the main contributions 
of this work, besides considering 
a stochastic right-hand side, are 
the fractional 
power~$\beta$~in~\eqref{eq:spde-fractional-parabolic}  
and 
the method of proving regularity 
using semigroups. 
As opposed to the extension approach  
in \cite{Biswas2021Harnack,Biswas2021,Litsgard2023,Nystrom2016,Stinga2017}, 
this setting does not require a Euclidean structure. 

\subsection{Outline} 

Preliminary notation and theory will be introduced 
in Section~\ref{sec:prelims}. 
In Section~\ref{sec:analysis-zeroIC} 
we give a meaning to the parabolic operator $\partial_t + A$ 
and its fractional powers in order to introduce 
well-defined mild and 
weak solutions of~\eqref{eq:spde-fractional-parabolic} with ${X_0=0}$.
Subsequently, we analyze these
in terms of spatiotemporal regularity.
Section~\ref{sec:covariance-structure} 
is concerned with 
the covariance structure of solutions. 
Finally, in Section~\ref{sec:Whittle-Matern} 
we apply our results to 
the space--time 
Whittle--Mat\'ern SPDE~\eqref{eq:spde-fractional-parabolic}  
considered on a bounded Euclidean domain 
or on a surface. 
This article is supplemented by two appendices: 
Appendix~\ref{app:auxiliary} contains several technical auxiliary results 
used in the proofs of 
Section~\ref{sec:analysis-zeroIC}.  
Appendix~\ref{app:functional-calc} collects necessary definitions 
and results from functional calculus.

\section{Preliminaries}\label{sec:prelims}

\subsection{Notation} 

The sets $\N := \{1,2,3,\dots\}$ and $\N_0 := \N \cup \{0\}$ 
denote the positive and non-negative integers, 
respectively. 
We write $s\wedge t$ (or $s \vee t$) for the 
minimum (or maximum)  
of two real numbers $s,t\in\R$. 
The real and imaginary parts of 
a complex number $z \in \C$ 
are denoted by $\Re z$ and $\Im z$, respectively; 
its argument, denoted $\arg z$, 
takes its values in $(-\pi,\pi]$. 
We write $\mathbf 1_D$ for the indicator function 
of a set~$D$. 
The restriction of 
a function $f \colon D \to E$ to 
a subset $D_0 \subseteq D$ 
is denoted by $f|_{D_0}\from D_0 \to E$;  
the image of $D_0$ under a linear mapping $T$ 
is written as $T D_0$. 
Given two parameter sets  
$\mathscr{P},\mathscr{Q}$ 
and two mappings 
$\mathscr{F},\mathscr{G}\from \mathscr{P}\times\mathscr{Q}\to\R$, 
we use the expression 
$\mathscr{F} (p,q) \lesssim_{q} \mathscr{G} (p,q)$ 
to indicate that for each $q \in \mathscr{Q}$ 
there exists a constant ${C_q \in(0,\infty)}$ 
such that 
$\mathscr{F} (p,q) \leq C_q \, \mathscr{G} (p,q)$ 
for all ${p\in \mathscr{P}}$. 
We write $\mathscr{F} (p,q) \eqsim_q \mathscr{G} (p,q)$ 
if both relations, 
$\mathscr{F} (p,q) \lesssim_q \mathscr{G}(p,q)$ 
and $\mathscr{G} (p,q) \lesssim_q \mathscr{F}(p,q)$, 
hold simultaneously.

\subsection{Banach spaces and operators} 

If not specified otherwise, 
$E$ or $F$ denote separable Banach spaces. 
We instead write $H$ or $U$ 
if we work with separable Hilbert spaces 
and wish to emphasize this. 
The scalar field $\K$ is either given 
by the real numbers $\R$ 
or the complex numbers~$\C$. 
A norm on $E$ will be denoted  
by~$\norm{\,\cdot\,}{E}$ 
and an inner product on $H$ 
by $\scalar{\,\cdot\,, \,\cdot\,}{H}$. 
We write $I$ for the identity operator. 
The notation ${E \hookrightarrow F}$ 
indicates that $E$ is continuously embedded in $F$, 
i.e., there exists a bounded injective map from $E$ to~$F$. 
The dual space of $E$ is denoted by $\dual E$\!. 
We write $\overline{E_0}^E$ for the closure 
of a subset $E_0\subseteq E$ 
with respect to the norm on $E$; 
the superscript may be omitted 
when there is no risk of confusion. 
The Borel $\sigma$-algebra of $E$ 
is denoted by $\mathcal B(E)$.

We write ${T \in \mathscr L(E;F)}$ 
if the linear operator 
$T\from E \to F$ is bounded.  
Whenever $E=F$, we 
abbreviate 
$\mathscr L(E) := \mathscr L(E;E)$, 
and this convention holds also 
for all other spaces of operators to be introduced. 
The space $\mathscr L(E;F)$ is rendered 
a Banach space when equipped 
with the usual operator norm; 
the space of Hilbert--Schmidt operators 
$\mathscr L_2(U;H)\subseteq \mathscr L(U;H)$
is a Hilbert space 
with respect to 
the inner product 
$(T,S)_{\mathscr L_2(U;H)} := \sum_{j \in \N}(Te_j, Se_j)_H$, 
where
$(e_j)_{j\in\N}$ is any orthonormal basis for~$U$.  
We write ${\dual T \in \mathscr L(\dual F; \dual E)}$ 
for the adjoint operator of $T\in \mathscr L(E;F)$. 
In the case that $T \in \mathscr L(U;H)$,  
we identify $\dual U = U$ 
and $\dual H = H$ via the Riesz 
maps, so that  
$\dual T\in \mathscr L(H;U)$. 
An operator $T\in\mathscr L(H)$ 
is said to be 
self-adjoint 
if $\dual T = T$, 
non-negative 
if $\scalar{Tx,x}{H}\geq 0$ holds for all $x\in H$, 
and strictly positive if 
there exists a constant  $\theta\in(0,\infty)$ such that 
$\scalar{Tx,x}{H} \geq \theta \|x\|_H^2$ 
holds for all $x\in H$.  

A linear operator $A$ on $E$ with domain $\mathsf D(A)$ 
is denoted by ${A\from \mathsf D(A) \subseteq E \to E}$  
and its range by~$\mathsf R(A)$. 
We call $A$ closed if its graph 
$\mathsf G(A):= {\{(x,Ax): x \in \mathsf D(A)\}}$ 
is closed with respect to the graph norm 
$\norm{(x,Ax)}{\mathsf G(A)} := \norm{x}{E} + \norm{Ax}{E}$, 
and densely defined if $\mathsf D(A)$ is dense in $E$. 
The definition 
$\norm{x}{\mathsf D(A)} := \norm{(x,Ax)}{\mathsf G(A)}$ 
yields a norm on $\mathsf D(A)$. 
If $\mathsf G(A) \subseteq \mathsf G(\widetilde A)$ 
for another linear operator $\widetilde A$ on $E$, 
then $\widetilde A$ is called an extension of $A$ 
and we write $A \subseteq \widetilde A$. 
If $\clos{\mathsf G(A)}$ is the graph of a linear operator, 
then we call this operator the closure of $A$, denoted $\clos A$.

\subsection{Function spaces}

Let a measure space $(S, \cS, \mu)$ be given. 
We abbreviate the phrases 
``almost everywhere'' and 
``almost all'' 
by 
``a.e.''\ and ``a.a.''\!,  
respectively. 

We say that a function $f \from S \to E$ 
is strongly measurable if it is the $\mu$-a.e.\ limit 
of measurable simple functions. 
For $p \in [1,\infty]$, the Bochner space of 
(equivalence classes of)
strongly measurable, 
$p$-integrable functions is denoted by $L^p(S;E)$. 
It is equipped with the norm 
\[
\norm{f}{L^p(S;E)} 
:= 
\begin{cases} 
\left( \int_S \norm{f(t)}{E}^p \rd \mu(t) \right)^{\nicefrac{1}{p}} 
&\text{if $p \in [1,\infty)$}, \\
\operatorname{ess\,sup}_{t\in S} \norm{f(t)}{E}  
&\text{if $p = \infty$},
\end{cases}  
\]
where $\operatorname{ess\,sup}$ 
denotes the essential supremum. 
The norm on $L^2(S;H)$ 
is induced by the inner product 
$(f,g)_{L^2(S;H)} := \int_S (f(t),g(t))_H \rd \mu(t)$. 

Now let $S$ be an interval $S := J \subseteq \R$,  
equipped with the Borel $\sigma$-algebra and the Lebesgue measure. 
The space of continuous functions from $J$ to $E$ 
will be denoted by~$C(J;E)$ or $C^{0,0}(J;E)$ 
and be endowed with the supremum norm. 
For $\alpha\in(0,1]$, we consider the space 
$C^{0,\alpha}(J;E)$ of $\alpha$-H\"older continuous functions 
with norm 
\[
\norm{f}{C^{0,\alpha}(J;E)} 
:= \seminorm{f}{C^{0,\alpha}(J;E)} + \norm{f}{C(J;E)},  
\; \text{where} \;\; 
\seminorm{f}{C^{0,\alpha}(J;E)} 
:= \sup_{ t, s\in J, \, t\neq s } 
\tfrac{\norm{f(t)-f(s)}{E}}{|t-s|^\alpha} 
\]
is the $\alpha$-H\"older seminorm. 
For $n \in \N_0$ and $0 \leq \alpha \leq 1$, 
the space $C^{n,\alpha}(J;E)$ 
consists of functions whose $n$th derivative exists 
and belongs to $C^{0,\alpha}(J;E)$. 
On this space we use the norm
$\norm{f}{C^{n,\alpha}(J;E)} 
:= 
\norm{f^{(n)}}{C^{0,\alpha}(J;E)} 
+ 
\sum_{k=0}^{n-1} \norm{f^{(k)}}{C(J;E)}$, 
where $f^{(k)}$ denotes the $k$th derivative of $f$.
Moreover, we define 
$C^{\infty}(J;E) := \bigcap_{n\in\N} C^{n,0}(J;E)$. 
We say that $f \in C^{n,\alpha}(J;E)$ is 
compactly supported if the support of $f$, 
defined by 
$\operatorname{supp} f := \overline{\{t\in J: f(t)\neq0\}}^J$\!, 
is compact. The space consisting of such functions 
is denoted by $C_c^{n,\alpha}(J;E)$.
If $f$ vanishes at a point $t \in J$, 
then we use the notation ${f \in C^{n,\alpha}_{0,\{t\}}(J;E)}$. 
The spaces $C_c^\infty(J;E)$ 
and $C_{0,\{t\}}^\infty(J;E)$ are defined analogously. 

For an open interval $J$, 
we say that $u \in L^2(J;E)$ belongs to $H^1(J;E)$ 
if there exists a function $v \in L^2(J;E)$ 
such that 
$\int_J v(t) \phi(t)\rd t = -\int_J u(t) \phi'(t)\rd t$ 
holds for all $\phi\in C_c^\infty(J;\R)$. 
The function ${\partial_t u := v}$ 
is called the 
weak derivative of $u$  
and the norm on $H^1(J;E)$ is  
$\norm{u}{H^1(J;E)} 
:= 
\bigl( 
\norm{u}{L^2(J;E)}^2 
+ 
\norm{\partial_t u}{L^2(J;E)}^2 
\bigr)^{\nicefrac{1}{2}}$\!. 
The completion of $C_c^\infty((0,\infty);E)$   
with respect to $\norm{\,\cdot\,}{H^1(0,\infty;E)}$ 
defines the space 
$H^1_{0,\{0\}}(0,\infty;E)$.  
Elements of $H^1_{0,\{0\}}(J;E)$ 
are restrictions of functions in $H^1_{0,\{0\}}(0,\infty;E)$ 
to ${J\subseteq (0,\infty)}$. 

Whenever the function space contains 
functions mapping to $E=\R$, 
we omit the codomain, 
e.g., we write $L^p(S) := L^p(S;\R)$ 
for the Lebesgue spaces.

\subsection{Vector-valued stochastic processes} 

Throughout this article, 
$(\Omega,\cF,\bbP)$ denotes 
a complete probability space 
that is equipped 
with a normal filtration $(\cF_t)_{t\geq 0}$, i.e., 
$\cF_0$ contains all elements $B\in\cF$ with 
$\bbP(B)=0$ and  
$\cF_t = \bigcap_{s>t} \cF_s$ for all $t\geq 0$. 
Statements which 
hold $\bbP$-almost surely 
are marked with ``$\bbP$-a.s.''\!.  

We call every strongly measurable function 
$Z\from \Omega \to E$ 
a (vector-valued) random variable, 
and the expectation of $Z \in L^1(\Omega;E)$ 
is defined as the Bochner integral
$\E[Z] := \int_{\Omega} Z(\omega) \rd \bbP(\omega)$. 
An $E$-valued stochastic process 
$X=(X(t))_{t\in[0,T]}$ 
indexed by the interval $[0,T]$, 
$T\in(0,\infty)$, 
is called integrable if 
${(X(t))_{t\in[0,T]} \subseteq L^p(\Omega;E)}$ 
holds for $p=1$,  
and square-integrable 
if this inclusion is true for~$p=2$. 
It is said to be predictable 
if it is strongly measurable 
as a mapping from $[0,T]\times \Omega$ to~$E$, 
where the former set is equipped 
with the $\sigma$-algebra generated by the family
\[
\{ (s,t] \times F_s: 0\leq s < t\leq T,\, F_s\in \cF_s \} 
\cup 
\{ \{0\}\times F_0 : F_0\in \cF_0 \}.
\] 
Given another 
$E$-valued
process $\widetilde X := (\widetilde X(t))_{t\in[0,T]}$,
we call $\widetilde X$ a modification
of $X$, provided that  
$\bbP(X(t)=\widetilde X(t))=1$ 
holds for all $t\in[0,T]$. 
Moreover, $X$ and $\widetilde X$
are said to be indistinguishable if 
$\bbP(\forall t\in[0,T] : X(t) = \widetilde X(t))=1$. 

For a self-adjoint strictly positive 
operator $Q \in \mathscr L(H)$, 
$(W^Q(t))_{t\geq 0}$ denotes 
a cylindrical $Q$-Wiener process
with respect to $(\cF_t)_{t\geq 0}$ which takes
its values in $H$, cf.~\cite[Proposition~2.5.2]{LiuRockner2015};
if $Q=I$, we omit the superscript 
and call $(W(t))_{t\geq 0}$ a cylindrical Wiener process.

\section{Analysis of the fractional stochastic evolution equation}
\label{sec:analysis-zeroIC}

The aim of this section is to define and analyze solutions to 
the following stochastic evolution equation 
of the general fractional order $\gamma\in(0,\infty)$:  
\begin{equation}
\label{eq:fractional-parabolic-spde-zeroIC} 
(\partial_t + A)^\gamma X (t)
= 
\dot{W}^{Q}(t),  
\quad 
t\in[0,T], 
\qquad 
X(0)=0. 
\end{equation}
We interpret this as an abstraction of 
\eqref{eq:spde-fractional-parabolic} 
with $X_0 = 0$. 
As noted in the introduction, 
we restrict the discussion to a zero initial field, 
since we are primarily interested in 
properties resulting from the 
dynamics of the SPDE~\eqref{eq:spde-fractional-parabolic}, 
respectively~\eqref{eq:fractional-parabolic-spde-zeroIC}, 
and 
the long-time 
behavior 
for $0\ll T < \infty$ 
of its solution. 
We also note that imposing non-zero 
boundary data for fractional problems is, 
in general, highly non-trivial, 
see e.g.~the recent works~\cite{AbatangeloDupaigne2017,HarbirPfeffererRogovs2018} 
on the fractional Laplacian. 

In Subsection~\ref{subsec:analysis-zeroIC:parabolic-operator} 
we investigate the \emph{parabolic operator} $\cB$, 
which is defined as the closure of the sum operator $\partial_t + A$ 
on an appropriate domain. 
In particular, we consider the $C_0$-semigroup generated 
by $-\cB$, which is used to define 
fractional powers $\cB^\gamma$ for $\gamma \in \R$. 
Interpreting the expression 
$(\partial_t + A)^\gamma$ appearing 
in~\eqref{eq:fractional-parabolic-spde-zeroIC} 
as~$\cB^\gamma$\!, 
we use this result to define mild solutions 
in Subsection~\ref{subsec:analysis-zeroIC:solution-concepts}. 
In this part, we furthermore introduce  
a weak solution concept 
for \eqref{eq:fractional-parabolic-spde-zeroIC}, 
and prove equivalence 
of the two solution concepts  
as well as   
existence and uniqueness 
of mild and weak solutions. 
Spatiotemporal regularity 
of solutions 
is the subject of  
Subsection~\ref{subsec:analysis-zeroIC:existence-uniqueness-regularity}.

\subsection{The parabolic operator and its fractional powers}
\label{subsec:analysis-zeroIC:parabolic-operator}

In this subsection we define the parabolic operator $\cB$ 
and fractional powers $\cB^\gamma$\!. 
We start by formulating several assumptions on 
the linear operator $A$, 
to which we shall refer throughout the remainder of this work. 
For an overview of the theory of $C_0$-semigroups, 
we refer the reader to~\cite{Engel1999} 
or~\cite{Pazy1983}. 
The complexification of a normed space or operator
is indicated by the subscript $\C$; 
see Subsection~\ref{subsubsec:functional-calc:complexification} in 
Appendix~\ref{app:functional-calc} 
for details. 
\begin{assumption}\label{assumption:A}
	Let $H$ be a separable Hilbert space
	over the real scalar field~$\R$. 
	We assume that the linear operator
	${A\from \mathsf D(A)\subseteq H \to H}$ satisfies 
	\begin{enumerate}[leftmargin=1cm, label=(\roman*)]
		\item\label{assumption:A:semigroup}
		$-A$ generates a $C_0$-semigroup $(S(t))_{t\geq 0}$. 
	\end{enumerate}
	Sometimes we additionally require 
	one or more of the following conditions:
	\begin{enumerate}[leftmargin=1cm, label=(\roman*)] 
		\setcounter{enumi}{1}
		\item\label{assumption:A:bdd-analytic}  
		$(S(t))_{t\geq 0}$ is (uniformly) bounded analytic, 
		i.e., the mapping ${t \mapsto S_{\C}(t)}$, 
		where ${S_{\C}(t) := [S(t)]_{\C}}$, 
		extends to a bounded holomorphic function 
		on an open sector $\Sigma_\omega \subseteq \C$ 
		for some angle 
		$\omega\in(0,\pi)$ 
		(see Definition~\ref{def:sectorial} 
		in Appendix~\ref{app:functional-calc}); 
		\item\label{assumption:A:bdd-Hinfty}  
		$A_{\C}$ admits a bounded $H^\infty$-calculus 
		with $\omega_{H^\infty}(A_\C)<\tfrac{\pi}{2}$, 
		see Definition~\ref{def:bdd-Hinfty-calc}; 
		\item\label{assumption:A:bddly-inv}  
		$A$ has a bounded inverse. 
	\end{enumerate}
\end{assumption} 
Under Assumption~\ref{assumption:A}\ref{assumption:A:semigroup}, 
Lemma~\ref{lem:complexifications-semigroups} allows us to 
use several results from~\cite{Engel1999,Haase2006,Pazy1983} 
for $C_0$-semigroups and their generators 
on complex spaces 
also for $(S(t))_{t\geq 0}$ and $-A$. 
For instance, by~\cite[Theorem~II.1.4]{Engel1999} 
and~\cite[Chapter~1, Theorem~2.2]{Pazy1983}
the operator $A$ is closed and densely defined, 
and the $C_0$-semigroup $(S(t))_{t\geq 0}$ satisfies
\begin{equation}\label{eq:exp-estimate-semigroup} 
\exists 
M \in [1,\infty), \; 
w \in \R: 
\quad 
\norm{S(t)}{\mathscr L(H)}
= \norm{S_\C(t)}{\mathscr L(H_\C)}
\leq 
M e^{-wt} 
\quad 
\forall t \geq 0.   
\end{equation}
If the conditions  
\ref{assumption:A:semigroup}, \ref{assumption:A:bdd-analytic} and
\ref{assumption:A:bddly-inv} 
are satisfied, 
then \eqref{eq:exp-estimate-semigroup} 
holds for some $w \in (0,\infty)$,
see e.g.\ \cite[p.~70]{Pazy1983}. 
In this case, $(S(t))_{t\geq 0}$ is said to be \emph{exponentially stable}. 
Moreover, we note that 
Assumption~\ref{assumption:A}\ref{assumption:A:bdd-analytic} 
is equivalent to the operator $A_\C$ being sectorial 
with $\omega(A_\C) < \tfrac{\pi}{2}$ 
by Theorem~\ref{thm:sectoriality-analytic-semigroups},
and that consequently condition  
\ref{assumption:A:bdd-Hinfty} implies \ref{assumption:A:bdd-analytic} 
since $\omega(A_\C) \le \omega_{H^\infty}(A_\C)$ 
by Remark~\ref{rem:Hinfty-sectoriality-angles}.
Whenever the conditions  
\ref{assumption:A:semigroup} and~\ref{assumption:A:bdd-analytic}
are satisfied, we have 
the following useful estimate 
(see \cite[Proposition~3.4.3]{Haase2006}): 
\begin{equation}\label{eq:analytic-est-1}
\forall c \in [0,\infty): 
\quad 
\| A^c S(t) \|_{\mathscr L(H)} 
=
\| A_\C^c S_\C(t) \|_{\mathscr L(H_\C)} 
\lesssim_c 
t^{-c} 
\quad 
\forall 
t \in (0,\infty). 
\end{equation}

As a first step towards defining the parabolic operator $\cB$, 
we define the Bochner space counterpart 
$\cA\from \mathsf D(\cA) \subseteq L^2(0,T;H) \to L^2(0,T;H)$ 
of $A$ by
\begin{equation}\label{eq:def-cA} 
\begin{split}
&[\cA v](\vartheta)
:= Av(\vartheta), 
\quad\;\; 
v \in \mathsf D(\cA), 
\ 
\text{a.a.\ }\vartheta\in (0,T), 
\\
&\mathsf D(\cA) 
= L^2(0,T;\mathsf D(A)) 
:= 
\bigl\{ v \in L^2(0,T;H): \norm{ Av(\,\cdot\,)}{L^2(0,T;H)} < \infty \bigr\}.   
\end{split}
\end{equation}
The $C_0$-semigroup $(S(t))_{t\geq 0}$ on $H$, 
generated by $-A$, 
can be associated to a family of 
operators $(\cS(t))_{t\geq 0}$ on $L^2(0,T;H)$ 
in a similar way: 
\begin{equation}\label{eq:def-cS}
[\cS(t)v](\vartheta) 
:= 
S(t)v(\vartheta), 
\quad\;\; 
t \geq 0, 
\   
v \in L^2(0,T;H),
\ 
\text{a.a.\ }\vartheta \in (0,T).
\end{equation}
It turns out that 
$(\cS(t))_{t\geq 0}\subseteq\mathscr L(L^2(0,T;H))$ 
is again a $C_0$-semigroup,  
with infinitesimal generator $-\cA$, 
see Proposition~\ref{prop:curly-A-semigroup} 
in Appendix~\ref{app:auxiliary}.  

In addition, 
we consider the family of zero-padded 
right-translation operators 
$(\cT(t))_{t\geq 0}$ on $L^2(0,T; H)$, defined by
\begin{equation}\label{eq:def-cT} 
[\cT(t)v](\vartheta) 
:= 
\widetilde v(\vartheta - t), 
\quad\;\;   
t\geq 0, 
\ 
v \in L^2(0,T; H), 
\ 
\text{a.a.\ }\vartheta \in (0,T),
\end{equation}
where $\widetilde v \in L^2(-\infty,T;H)$ 
denotes the extension 
of $v$ by zero to $(-\infty,T)$. 
As shown in Proposition~\ref{prop:translation-semigroup} 
in Appendix~\ref{app:auxiliary}, 
also $(\cT(t))_{t\geq 0} \subseteq \mathscr L(L^2(0,T;H))$ 
is a $C_0$-semigroup 
and its infinitesimal generator 
is given by $-\partial_t$, 
where 
\begin{equation}\label{eq:def-partial-t} 
\partial_t \from \mathsf D(\partial_t) \subseteq L^2(0,T;H) \to L^2(0,T;H), 
\qquad 
\mathsf D(\partial_t) = H^1_{0,\{0\}}(0,T; H), 
\end{equation} 
denotes the Bochner--Sobolev 
vector-valued weak derivative. 
We point out that the domain 
$\mathsf D(\partial_t)= H^1_{0,\{0\}}(0,T; H)$ 
encodes the zero initial condition 
of the SPDE~\eqref{eq:fractional-parabolic-spde-zeroIC}.
Furthermore, note that 
it readily follows 
from the definitions in 
\eqref{eq:def-cS} and \eqref{eq:def-cT} that, 
for all $t\geq 0$,  
every $ v \in L^2(0,T;H)$, 
and a.a.\ $\vartheta \in (0,T)$,
\[ 
[\cS(t) \cT(t)v](\vartheta) 
= 
[\cT(t) \cS(t)v](\vartheta) 
= 
S(t) \widetilde v(\vartheta-t), 
\] 
i.e., the semigroups $(\cS(t))_{t\geq 0}$ 
and $(\cT(t))_{t\geq 0}$ commute.

We now define the sum operator 
$\partial_t + \cA 
\from 
\mathsf D(\partial_t+\cA)\subseteq L^2(0,T;H) 
\to L^2(0,T;H)$ 
on its natural domain, that is 
\begin{equation}\label{eq:sum-operator} 
( \partial_t+ \cA)v 
:= 
\partial_t v + \cA v, 
\quad 
v\in \mathsf D(\partial_t + \cA) 
= H^1_{0,\{0\}}(0,T; H) \cap L^2(0,T; \mathsf D(A)), 
\end{equation}
with $\cA$ and $\partial_t$ 
as given in \eqref{eq:def-cA} 
and \eqref{eq:def-partial-t}, respectively. 
The next proposition shows that 
the closure of $-(\partial_t + \cA)$ 
again generates a $C_0$-semigroup, 
namely the product semigroup 
of $(\cS(t))_{t\geq 0}$ and $(\cT(t))_{t\geq 0}$. 

\begin{proposition}\label{prop:parabolic-operator-semigroup}
	Let Assumption~\textup{\ref{assumption:A}\ref{assumption:A:semigroup}} 
	be satisfied. 
	The closure 
	$\cB := \clos{ \partial_t + \cA }$
	of the sum operator $\partial_t + \cA$ 
	defined in \eqref{eq:sum-operator} exists and $-\cB$
	generates the $C_0$-semigroup 
	$(\cS(t) \cT(t))_{t\geq 0}$ 
	on $L^2(0,T;H)$, 
	which satisfies 
	\[ 
	\norm{ \cS(t) \cT(t) }{\mathscr L(L^2(0,T;H))}  
	=
	\norm{ \cT(t) \cS(t) }{\mathscr L(L^2(0,T;H))}  
	= 
	\begin{cases}
	\norm{S(t)}{\mathscr L(H)}   
	& 
	\text{if } 
	0 \leq t < T, 
	\\
	0
	& 
	\text{if } 				
	t \geq T, 
	\end{cases} 
	\] 
	where $(\cS(t))_{t\geq 0}$ and $(\cT(t))_{t\geq 0}$ are defined 
	as in \eqref{eq:def-cS} and \eqref{eq:def-cT}, respectively.
\end{proposition}

\begin{proof}
	By the commutativity of the semigroups 
	$(\cS(t))_{t\geq 0}$ and $(\cT(t))_{t\geq 0}$,  
	we may conclude that $( \cT(t) \cS(t) )_{t\geq 0}$ 
	is a $C_0$-semigroup whose generator is an extension 
	of $-(\partial_t + \cA)$, 
	and the domain of the generator contains 
	$H^1_{0,\{0\}}(0,T;H) \cap L^2(0,T;\mathsf D(A))$ 
	as a subspace that is dense  
	with respect to the graph norm, 
	see~\cite[Example~II.2.7]{Engel1999}. 
	Subsequently, Lemma~\ref{lem:generators} 
	shows that the generator is 
	the closure of $-(\partial_t + \cA)$.
	
	Fix $t\in[0,T)$.  
	The inequality 
	$\norm{ \cT(t) \cS(t) }{\mathscr L(L^2(0,T;H))} 
	\leq \norm{S(t)}{\mathscr L(H)}$ 
	follows by the contractivity of $\cT(t)$ 
	and the operator norm isometry from    
	Lemma~\ref{lem:Bochner-space-operator}\ref{lem:Bochner-space-operator-a}. 
	Now we turn to the reverse inequality. 
	By definition of the operator norm on $\mathscr L(H)$, 
	there exists a normalized sequence 
	$(x_n)_{n\in\bbN}$ in $H$ 
	such that 
	$\norm{S(t)x_n}{H} \geq \norm{S(t)}{\mathscr L(H)} - \frac{1}{n}$  
	holds for all $n\in\bbN$. 
	Correspondingly, 
	define 
	the sequence $(v_n)_{n\in\bbN}$  
	in $L^2(0,T;H)$ by 
	$v_n(\vartheta) 
	:= 
	(T-t)^{-\nicefrac{1}{2}} \mathbf 1_{(0,T-t)}(\vartheta) x_n$ 
	for every $\vartheta\in(0,T)$ and all $n\in\N$. 
	Note that 
	$\norm{v_n}{L^2(0,T;H)} = 1$  
	for every $n \in \N$, 
	and 
	\[
	\norm{ \cT(t) \cS(t) v_n}{L^2(0,T;H)} 
	= 
	\norm{(T-t)^{-\nicefrac{1}{2}}\mathbf 1_{(t,T)}}{L^2(0,T)} 
	\norm{S(t)x_n}{H} 
	\geq 
	\norm{S(t)}{\mathscr L(H)} - \tfrac{1}{n}.
	\] 
	As this holds for all $n \in \N$, we conclude that 
	$\norm{ \cT(t) \cS(t)}{\mathscr L(L^2(0,T;H))} 
	\geq \norm{S(t)}{\mathscr L(H)}$.
	The final assertion for $t\geq T$ 
	follows from the fact 
	that $\cT(t) = 0$ for $t \geq T$.
\end{proof}

\begin{remark}
	The closure $\cB =\clos{\partial_t + \cA}$ 
	appearing in Proposition~\ref{prop:parabolic-operator-semigroup} 
	raises the question of when the sum operator itself is closed. 
	The answer is intimately related 
	to the subject of maximal $L^p$-regularity;
	we refer the reader to~\cite{DRH2003} or~\cite{KunstmannWeis2004} for detailed accounts of this theory.
	In the Hilbert~space setting, 
	the sum turns out to be closed under 
	Assumptions~\ref{assumption:A}\ref{assumption:A:semigroup},\ref{assumption:A:bdd-analytic}. 
	Indeed,~$[\partial_t]_\C$ has a bounded $H^\infty$-calculus 
	with ${\omega_{H^\infty}([\partial_t]_\C) \leq \tfrac{\pi}{2}}$ 
	since $(\cT(t))_{t\geq 0}$ and $(\cT_\C(t))_{t\ge0}$ 
	are contractive, see 
	Definition~\ref{def:bdd-Hinfty-calc} 
	in Appendix~\ref{app:functional-calc} and 
	\cite[Theorem~10.2.24]{AnalysisInBanachSpacesII}. 
	By Assumption~\ref{assumption:A}\ref{assumption:A:bdd-analytic} 
	and Theorem~\ref{thm:sectoriality-analytic-semigroups},  
	we have $\omega(A_\C) < \tfrac{\pi}{2}$, 
	and the same follows for $\cA_\C$  
	by applying 
	Lemma~\ref{lem:Bochner-space-operator}\ref{lem:Bochner-space-operator-a} 
	to its resolvent operators. 
	Thus, we may conclude 
	with~\cite[Theorem~12.13]{KunstmannWeis2004} 
	that $[\partial_t + \cA]_\C$ is closed, so that 
	the same holds for $\partial_t + \cA$.
\end{remark}

We are now in the position to define fractional powers 
of the parabolic operator. 
For $\gamma \in(0,\infty)$ 
we work with the 
following representation 
(see Appendix~\ref{subsec:appendix-functional-calc:frac-powers}): 
\begin{equation}\label{eq:def-neg-fractional-parabolic} 
\cB^{-\gamma} 
:= 
\frac{1}{\Gamma(\gamma)} 
\int_0^\infty s^{\gamma-1} \cS(s) \cT(s) \rd s 
= 
\frac{1}{\Gamma(\gamma)} 
\int_0^T s^{\gamma-1} \cS(s) \cT(s) \rd s.
\end{equation}
Note that, for any $\gamma\in(0,\infty)$, 
this definition yields a well-defined bounded linear operator 
on $L^2(0,T;H)$, since the product semigroup 
$(\cS(t)\cT(t))_{t\geq 0}$ was seen to be exponentially 
stable (in fact, eventually zero) in 
Proposition~\ref{prop:parabolic-operator-semigroup}. 

The next result shows that the pointwise evaluation 
of $\cB^{-\gamma}f$ at $t\in[0,T]$ is meaningful,  
provided that $\gamma>\frac{1}{2}$. 

\begin{proposition}\label{prop:neg-frac-C00}
	Suppose Assumption~\textup{\ref{assumption:A}\ref{assumption:A:semigroup}} 
	and let $p\in(1,\infty), \gamma\in(\nicefrac{1}{p},\infty)$.~Then 
	\begin{equation}\label{eq:def:Bgammap} 
	f \mapsto \mathscr{B}_{\gamma,p} f, 
	\quad\;\;  
	[\mathscr{B}_{\gamma,p} f](t) 
	:= 
	\frac{1}{\Gamma(\gamma)} 
	\int_0^t (t-s)^{\gamma-1} S(t-s)f(s) \rd s 
	\quad 
	\forall t \in [0,T], 
	\end{equation} 
	defines a bounded linear operator, 
	mapping $f\in L^p(0,T;H)$ 
	into $C_{0,\{0\}}([0,T];H)$. 
	
	In particular, if $\gamma\in(\nicefrac{1}{2}, \infty)$, 
	we have for 
	the negative fractional parabolic operator 
	${\cB^{-\gamma}}$ defined 
	by \eqref{eq:def-neg-fractional-parabolic}  
	when acting on $f\in L^2(0,T;H)$ 
	the pointwise formula
	\begin{equation}\label{eq:neg-parabolic-formula} 
	[\cB^{-\gamma} f](t)
	= 
	[\mathscr{B}_{\gamma,2} f](t) 
	= 
	\frac{1}{\Gamma(\gamma)} 
	\int_0^t (t-s)^{\gamma-1} S(t-s)f(s) \rd s 
	\quad 
	\forall t \in [0,T]. 
	\end{equation}
\end{proposition}

\begin{proof} 
	By \cite[Proposition~5.9]{DaPrato2014}, 
	for $p\in(1,\infty)$ and $\gamma\in(\nicefrac{1}{p},\infty)$,  
	the operator $\mathscr{B}_{\gamma,p}$ 
	defined by~\eqref{eq:def:Bgammap} 
	maps continuously from $L^p(0,T;H)$ 
	to $C_{0,\{0\}}([0,T];H)$. 
	
	Next, note that for all $f \in L^2(0,T;H)$ 
	and a.a.\ $t \in [0,T]$, we obtain 
	by \eqref{eq:def-neg-fractional-parabolic}  
	\begin{align*}
	[\cB^{-\gamma}f ](t) 
	&= 
	\frac{1}{\Gamma(\gamma)} 
	\int_0^\infty s^{\gamma-1} [ \cS(s) \cT(s) f ](t)\rd s
	= 
	\frac{1}{\Gamma(\gamma)} \int_0^t s^{\gamma-1} S(s) f(t-s) \rd s 
	\\
	&= 
	\frac{1}{\Gamma(\gamma)} \int_0^t (t-s)^{\gamma-1} S(t-s) f(s) \rd s 
	= 
	[\mathscr{B}_{\gamma,2} f](t).
	\end{align*}
	Thus, by the first part of this proposition, 
	for every $\gamma\in(\nicefrac{1}{2},\infty)$, 
	we have that $\mathsf R(\cB^{-\gamma})\subseteq C_{0,\{0\}}([0,T];H)$ 
	and the above identities hold pointwise in $t\in[0,T]$. 
\end{proof}

\begin{remark}
	Propositions~\ref{prop:parabolic-operator-semigroup} 
	and~\ref{prop:neg-frac-C00} require only 
	Assumption~\ref{assumption:A}\ref{assumption:A:semigroup}, i.e., 
	that $-A$ generates the $C_0$-semigroup $(S(t))_{t\geq 0}$. 
	Exponential stability or uniform boundedness 
	of $(S(t))_{t\geq 0}$ are not needed, 
	since we consider linear operators on $L^2(0,T;H)$ 
	(instead of $L^2(0,\infty;H)$), 
	allowing us to use uniform boundedness 
	of $(S(t))_{t\geq 0}$ on the compact interval $[0,T]$ 
	to derive exponential stability of $(\cS(t)\cT(t))_{t\geq 0}$. 
\end{remark} 

In what follows, we may also consider the  
operator $\cB^{-\gamma*} := (\cB^{-\gamma})^*$\!. 
More specifically, we will use it 
in the next section 
to define a weak solution  
to the fractional parabolic 
SPDE~\eqref{eq:fractional-parabolic-spde-zeroIC}. 
The following lemma 
provides useful results 
for the adjoint $\cB^{-\gamma*}$ 
which are analogous to those 
for $\cB^{-\gamma}$ in 
Proposition~\ref{prop:neg-frac-C00}.  
For ease of presentation, the 
proof has been moved to 
Subsection~\ref{app:subsec:proof-lemma-adjoint} 
of Appendix~\ref{app:auxiliary}. 

\begin{lemma}\label{lem:adjoint-neg-frac}
	Suppose Assumption~\textup{\ref{assumption:A}\ref{assumption:A:semigroup}}  
	and let $\gamma \in (\nicefrac{1}{2},\infty)$. 
	The adjoint negative fractional parabolic operator ${\cB^{-\gamma*}}$ 
	maps $g\in L^2(0,T;H)$ into $C_{0,\{T\}}([0,T];H)$, and  
	\begin{equation}\label{eq:adjoint-neg-parabolic-formula} 
	[ \cB^{-\gamma*}g ](s) 
	= 
	\frac{1}{\Gamma(\gamma)} 
	\int_s^T (t-s)^{\gamma-1} \dual{ [S(t-s)] } g(t) \rd t
	\quad 
	\forall s \in [0,T]. 
	\end{equation}
\end{lemma} 

Finally, we note that $\cB^{-\gamma*} = (\cB^*)^{-\gamma}$\!. 
To see that the fractional power on the right-hand side 
is indeed well-defined, we use 
\cite[Chapter~1, Corollary 10.6]{Pazy1983}
and conclude that $-\cB^*$ is the generator 
of the $C_0$-semigroup 
$([\cS(t)\cT(t)]^*)_{t\geq 0}$, 
which clearly inherits the exponential stability 
from $( \cS(t) \cT(t))_{t\geq 0}$ 
since their norms are equal. 
The identity is then obtained as follows, 
\[
\cB^{-\gamma*} 
= 
\biggl(\frac{1}{\Gamma(\gamma)} 
\int_0^\infty s^{\gamma-1} \cS(s) \cT(s) \rd s\biggr)^* 
= 
\frac{1}{\Gamma(\gamma)} \int_0^\infty
s^{\gamma-1} [ \cS(s) \cT(s)]^*\rd s 
= 
(\cB^*)^{-\gamma}\!,
\]
where the first and last identities are 
due to \eqref{eq:def-neg-fractional-parabolic} and the second 
is a consequence of the general ability 
to interchange Bochner integrals and duality pairings.

\subsection{Solution concepts, existence and uniqueness}
\label{subsec:analysis-zeroIC:solution-concepts}

We now turn towards defining solutions 
to \eqref{eq:fractional-parabolic-spde-zeroIC} 
for fractional powers $\gamma\in(0,\infty)$. 
Recall from Section~\ref{sec:prelims} 
that $(\Omega, \cF, \bbP)$ is a complete probability space  
equipped with a normal filtration $(\cF_t)_{t\geq 0}$,
and that $(W^Q(t))_{t\ge0}$ is a cylindrical $Q$-Wiener process  on $H$
with respect to $(\cF_t)_{t\geq 0}$,
where
$Q \in \mathscr L(H)$ is self-adjoint 
and strictly positive. 

Having defined and investigated the parabolic operator $\cB$, 
its domain and its fractional powers, 
we are now in particular able to invert 
the fractional parabolic operator $\cB^\gamma$\!.
Equation \eqref{eq:neg-parabolic-formula} suggests 
the following definition of a fractional stochastic convolution 
as a \emph{mild solution} 
to~\eqref{eq:fractional-parabolic-spde-zeroIC}. 

\begin{definition}\label{def:frac-mild-sol-zeroIC} 
	Let Assumption~\ref{assumption:A}\ref{assumption:A:semigroup} 
	hold and define, for $\gamma \in (0,\infty)$,
	the stochastic convolution 
	\begin{equation}\label{eq:stoch-conv} 
	\widetilde{Z}_\gamma(t)
	:=
	\frac{1}{\Gamma(\gamma)}
	\int_0^t 
	(t-s)^{\gamma-1} S(t-s)
	\rd W^Q(s), 
	\quad 
	t\in[0,T]. 
	\end{equation} 
	A predictable $H$-valued stochastic process 
	$Z_\gamma := (Z_\gamma(t))_{t\in[0,T]}$ is called a \emph{mild solution} 
	to \eqref{eq:fractional-parabolic-spde-zeroIC}  
	if, for all $t\in[0,T]$, it satisfies 
	$Z_\gamma(t) 
	=
	\widetilde{Z}_\gamma(t)$, 
	$\bbP$-a.s.
\end{definition} 

We first address 
existence and mean-square continuity 
of mild solutions. 
Furthermore, 
we adapt the  
Da Prato--Kwapie\'{n}--Zabczyk 
factorization method 
(see \cite{DaPratoKwapienZabczyk1987}, \cite[Section~5.3]{DaPrato2014}) 
to establish the existence of 
a pathwise continuous modification.  

\begin{theorem}\label{thm:existence-mild}  
	Let Assumption~\textup{\ref{assumption:A}\ref{assumption:A:semigroup}}  
	be satisfied and let $\gamma \in (0,\infty)$ be such that
	\begin{equation}\label{eq:mildsol-cond-fractional} 
	\exists \, \delta \in [0,\gamma): 
	\quad 
	\int_0^T 
	\bigl\| t^{\gamma-1-\delta} S(t) Q^{ \frac{1}{2} } \bigr\|_{\mathscr L_2(H)}^2 
	\rd t 
	< \infty. 
	\end{equation} 
	The stochastic convolution 
	$\widetilde{Z}_\gamma(t)$ in \eqref{eq:stoch-conv}
	belongs to  
	$L^2(\Omega;H)$  
	for all $t\in [0,T]$ if and only if~\eqref{eq:mildsol-cond-fractional}
	holds with $\delta=0$.
	In this case,
	the mapping 
	$t\mapsto \widetilde{Z}_\gamma(t)$ 
	is an element of $C([0,T];L^p(\Omega;H))$ 
	for all $p\in[1,\infty)$; in particular, 
	there exists a mild solution  
	in the sense of Definition~\textup{\ref{def:frac-mild-sol-zeroIC}}, 
	and it is mean-square continuous. 
	
	Whenever~\eqref{eq:mildsol-cond-fractional}
	holds for some $\delta\in(0,\gamma)$, 
	then for every $p \in [1,\infty)$ 
	there exists a modification 
	of $\widetilde Z_\gamma$ with continuous sample paths 
	belonging 
	to $L^p(\Omega;C([0,T];H))$. 
	In particular, the mild solution has a modification 
	with continuous sample paths.
\end{theorem}

\begin{proof} 
	We first consider the case $\delta=0$ 
	in \eqref{eq:mildsol-cond-fractional}.
	By the It\^o isometry 
	(see e.g.\ \cite[Proposition 2.3.5 and p.~32]{LiuRockner2015}), 
	we obtain the identity  
	\[
	\sup_{t\in[0,T]} 
	\| \widetilde{Z}_\gamma(t) \|_{L^2(\Omega;H)}^2
	=  
	\frac{1}{|\Gamma(\gamma)|^2}
	\int_0^T 
	\bigl\| t^{\gamma-1} S(t) Q^{ \frac{1}{2} } \bigr\|_{\mathscr L_2(H)}^2 
	\rd t . 
	\]
	Therefore, 
	$\widetilde{Z}_\gamma(t)\in L^2(\Omega;H)$  
	holds 
	for all $t\in[0,T]$  
	if and only if \eqref{eq:mildsol-cond-fractional} 
	is satisfied with $\delta=0$. 
	The fact  
	that $t\mapsto \widetilde{Z}_\gamma(t)$ 
	belongs to $C([0,T];L^p(\Omega;H))$ 
	for all $p\in[1,\infty)$ 
	will be shown in greater generality 
	in Proposition~\ref{prop:continuity-in-time}, 
	see Subsection~\ref{subsubsec:proof-temporal-reg}.  
	
	Moreover, note that 
	$\widetilde{Z}_\gamma \from [0,T]\times\Omega \to H$  
	is measurable and 
	$(\cF_t)_{t\in[0,T]}$-adapted, 
	and that mean-square continuity implies 
	continuity in probability, so that 
	we may apply 
	\cite[Proposition~3.21]{Peszat_zabczyk_2007} 
	to conclude that 
	there exists a 
	predictable modification $Z_\gamma$ 
	of $\widetilde{Z}_\gamma$. 
	Then, $Z_\gamma$ is    
	a mild solution to \eqref{eq:fractional-parabolic-spde-zeroIC} 
	in the sense of Definition~\ref{def:frac-mild-sol-zeroIC}. 
	
	Now suppose that~\eqref{eq:mildsol-cond-fractional}
	holds for some $\delta \in (0,\gamma)$ 
	and let $p\in( \nicefrac{1}{\delta} \vee 1, \infty)$. 
	By the above considerations,
	$\widetilde Z_{\gamma-\delta}$ and $\widetilde Z_{\gamma}$ exist
	as elements of $C([0,T];L^p(\Omega;H))$. 
	In particular, we have
	\[
		\widetilde Z_{\gamma-\delta} \in L^p(0,T; L^p(\Omega; H)) \cong L^p(\Omega; L^p(0,T;H)),
	\]
	where the latter identification holds by Fubini's theorem. 
	For this reason, there exists a set 
	$\Omega_0\in\cF$ 
	with $\bbP(\Omega_0)=0$ 
	such that 
	\[
		\forall\omega \in \Omega_0^c = \Omega\setminus\Omega_0 : \quad 
		\widetilde Z_{\gamma-\delta} (\,\cdot\,, \omega) 
		\in L^p(0,T;H).
	\]
	We recall the linear operator 
	\[
		\mathscr{B}_{\delta,p} \from L^p(0,T;H) \to C_{0,\{0\}}([0,T];H)
	\]
	from \eqref{eq:def:Bgammap} and 
	claim that the process~$\widehat{Z}_\gamma$ 
	defined for $t\in[0,T]$ and $\omega\in\Omega$ by  
	\[
	\widehat{Z}_\gamma(t, \omega) 
	:= 
	\begin{cases} 
	\bigl[\mathscr{B}_{\delta,p} 
	\widetilde Z_{\gamma-\delta}\bigr](t, \omega) 
	&\text{if }(t,\omega) \in [0,T]\times \Omega_0^c , 
	\\
	0 
	&\text{if }(t,\omega) \in [0,T]\times \Omega_0 , 
	\end{cases} 
	\]
	is
	the desired continuous modification of $\widetilde Z_\gamma$.
	To this end, firstly note that for every $\omega\in\Omega$ 
	the mapping $t \mapsto \widehat{Z}_\gamma(t, \omega)$ 
	indeed is continuous and  
	$\widehat{Z}_\gamma\in L^p(\Omega; C([0,T];H))$; 
	this follows 
	from Proposition~\ref{prop:neg-frac-C00} since 
	$\delta \in( \nicefrac{1}{p}, \infty)$. 
	In order to show that $\widehat{Z}_\gamma$ 
	is a modification of $\widetilde{Z}_\gamma$, 
	we fix $t \in [0,T]$ and employ formulas~\eqref{eq:def:Bgammap}
	and~\eqref{eq:stoch-conv} along with the semigroup property to obtain that 
	\begin{align}
	&\widehat{Z}_{\gamma}(t)
	= 
	\bigl[\mathscr{B}_{\delta,p} 
	\widetilde Z_{\gamma-\delta}\bigr](t) 
	= 
	\frac{1}{\Gamma(\delta)} 
	\int_0^t (t-s)^{\delta-1} S(t-s) \widetilde Z_{\gamma-\delta}(s) 
	\rd s
	\notag
	\\
	&= 
	\frac{1}{\Gamma(\delta)\Gamma(\gamma-\delta)} 
	\int_0^t 
	(t-s)^{\delta-1} S(t-s) 
	\biggl[
	\int_0^s (s-r)^{\gamma-\delta-1}S(s-r)\rd W^Q(r)
	\biggr]
	\rd s 
	\notag
	\\
	&= 
	\frac{1}{\Gamma(\delta)\Gamma(\gamma-\delta)} 
	\int_0^t \int_0^s
	(t-s)^{\delta-1} 
	(s-r)^{\gamma-\delta-1} S(t-r) \rd W^Q(r)
	\rd s , 
	\quad 
	\bbP\text{-a.s.} 
	\label{eq:factorization-identity-1}
	\end{align}
	We set   
	$\widetilde M_T := \sup_{t\in[0,T]} \norm{S(t)}{\mathscr L(H)}$,  
	$K_T := \int_0^T 
	\bigl\| t^{\gamma-1-\delta} S(t) Q^{ \frac{1}{2} } \bigr\|_{\mathscr L_2(H)}^2 
	\rd t$
	and find 
	\begin{align*}
	&\int_0^t 
	\biggl[
	\int_0^s
	\bigl\| 
	(t-s)^{\delta-1} 
	(s-r)^{\gamma-\delta-1} S(t-r)Q^{\frac{1}{2}}
	\bigr\|_{\mathscr L_2(H)}^2
	\rd r
	\biggr]^{\nicefrac{1}{2}}
	\rd s 
	\\
	&\quad \leq 
	\widetilde M_T
	\int_0^t 
	(t-s)^{\delta-1}
	\biggl[
	\int_0^s
	\bigl\| 
	(s-r)^{\gamma-\delta-1} S(s-r)Q^{\frac{1}{2}}
	\bigr\|_{\mathscr L_2(H)}^2
	\rd r
	\biggr]^{\nicefrac{1}{2}}
	\rd s 
	\\
	&\quad
	=
	\widetilde M_T
	\int_0^t 
	(t-s)^{\delta-1}
	\biggl[
	\int_0^s
	\bigl\| 
	r^{\gamma-1-\delta} S(r)Q^{\frac{1}{2}}
	\bigr\|_{\mathscr L_2(H)}^2
	\rd r
	\biggr]^{\nicefrac{1}{2}}
	\rd s
	\leq 
	\frac{\widetilde M_T T^\delta \sqrt{ K_T } }{ \delta } < \infty.
	\end{align*}
	This estimate shows that 
	\[
	s\mapsto 
	\mathbf 1_{(0,t)}(s)
	\mathbf 1_{(0,s)}(\,\cdot\,)
	(t-s)^{\delta-1} 
	(s-\,\cdot\,)^{\gamma-\delta-1} 
	S(t-\,\cdot\,)Q^{\frac{1}{2}} 
	\in L^1(0,T;L^2(0,T;\mathscr L_2(H))) , 
	\] 
	and the stochastic Fubini theorem 
	\cite[Theorem~8.14]{Peszat_zabczyk_2007} 
	may be used
	in~\eqref{eq:factorization-identity-1}, yielding
	\[ 
	\widehat Z_{\gamma}(t)
	= 
	\frac{1}{\Gamma(\delta)\Gamma(\gamma-\delta)} 
	\int_0^t \biggl[\int_r^t
	(t-s)^{\delta-1} 
	(s-r)^{\gamma-\delta-1}  
	\rd s\biggr] S(t-r) \rd W^Q(r),
	\quad 
	\bbP\text{-a.s.} 
	\] 
	Using the change of variables $u(s) := \frac{s-r}{t-r}$
	and \cite[Formula~5.12.1]{Olver2010}, we derive that 
	\begin{align*}
	(t-r)^{1-\gamma}\int_r^t
	(t-s)^{\delta-1} 
	(s-r)^{\gamma-\delta-1}  
	\rd s
	=
	\int_0^1
	(1-u)^{\delta-1} 
	u^{\gamma-\delta-1}  
	\rd u
	=
	\frac{\Gamma(\gamma-\delta)\Gamma(\delta)}{\Gamma(\gamma)},
	\end{align*}
	which shows that 
	$\widehat{Z}_{\gamma}(t) = \widetilde{Z}_{\gamma}(t)$ 
	holds $\bbP$-a.s. 
	Since $t \in [0,T]$ was arbitrary this implies 
	that $\widehat{Z}_{\gamma}$ 
	is a modification of $\widetilde{Z}_{\gamma}$ 
	and completes the proof  
	for $p \in ( \nicefrac{1}{\delta} \vee 1,\infty)$. 
	Finally, the case $p\in [1, \nicefrac{1}{\delta} \vee 1 ]$ follows  
	from the nestedness of the $L^p(\Omega;C([0,T];H))$ spaces.
\end{proof}

In order to provide a more rigorous justification for 
the Definition~\textup{\ref{def:frac-mild-sol-zeroIC}}
of a mild solution 
to \eqref{eq:fractional-parabolic-spde-zeroIC}, 
we proceed as follows: 
We seek a further suitable solution concept 
of a \emph{weak solution}, 
which follows ``naturally'' 
from~\eqref{eq:fractional-parabolic-spde-zeroIC} 
using $L^2(0,T;H)$ inner products, 
and show that weak and mild solutions 
are equivalent. 
For this, we first define the \emph{weak stochastic It\^o integral} 
for $f\colon (0,T)\to \mathscr L(H)$ and 
$g\colon (0,T) \to H$
by 
\[ 
\int_0^t 
\bigl( f(s) \rd W^Q(s), g(s) \bigr)_{H} 
:=
\int_0^t \widetilde{f}_g(s) \rd W^Q(s), 
\quad t \in [0,T],
\]
where 
$\int_0^T \bigl\| Q^{\frac{1}{2}} \dual{[f(s)]} g(s) \bigr\|_{H}^2 \rd s < \infty$  
and 
$\widetilde{f}_g \from (0,T) \to \mathscr L(H;\R)$ 
is defined by
\[ 
\widetilde{f}_g(s) x 
:= 
( f(s) x, g(s) )_{H}  
\quad  
\forall 
x \in H, 
\quad 
\forall 
s\in(0,T), 
\]
cf.~\cite[Lemma~2.4.2]{LiuRockner2015}.

\begin{definition}\label{def:frac-weak-sol-zeroIC} 
	Let Assumption~\textup{\ref{assumption:A}\ref{assumption:A:semigroup}}  
	hold and let $\gamma \in (0,\infty)$.
	A predictable $H$-valued stochastic process 
	$Y_\gamma:=(Y_\gamma(t))_{t\in[0,T]}$ is called a \emph{weak solution} 
	to \eqref{eq:fractional-parabolic-spde-zeroIC} if 
	it is mean-square continuous  
	and, in addition, 
	\begin{equation}\label{eq:fractional-weaksol-zeroIC}  
	\forall 
	\psi \in \mathsf D(\cB^{\gamma*}) : 
	\quad 
	( Y_\gamma, \cB^{\gamma*}\psi)_{L^2(0,T;H)} 
	= 
	\int_0^T \bigl( \rd W^Q(t), \psi(t) \bigr)_H, 
	\quad 
	\bbP\text{-a.s.} 
	\end{equation}
\end{definition}

\begin{remark} For $\gamma=1$, a natural weak solution concept 
	is the formulation given in 
	\cite[Definition~9.11]{Peszat_zabczyk_2007}: 
	A predictable $H$-valued process $(Y_1(t))_{t\in[0,T]}$ 
	is a weak solution to \eqref{eq:fractional-parabolic-spde-zeroIC} 
	if $\sup_{t\in[0,T]} \norm{ Y_1(t) }{L^2(\Omega;H)} < \infty$ 
	and, for all $t \in [0,T]$ and $y \in \mathsf D(\dual A)$, 
	\[ 
	(Y_1(t), y)_H 
	= 
	- 
	\int_0^t (Y_1(s), \dual A y)_H \rd s 
	+ 
	\bigl( W^Q(t), y \bigr)_H, 
	\quad 
	\text{$\bbP$-a.s.}
	\] 
	Provided that Assumption~\textup{\ref{assumption:A}\ref{assumption:A:semigroup}}
	and \eqref{eq:mildsol-cond-fractional} are satisfied, 
	by \cite[Theorem~9.15]{Peszat_zabczyk_2007} 
	an $H$-valued stochastic process is 
	a weak solution in this sense 
	if and only if it 
	is a mild solution 
	in the sense of 
	Definition~\textup{\ref{def:frac-mild-sol-zeroIC}} 
	with $\gamma=1$. 
\end{remark} 

In the next proposition 
we generalize this result 
to an arbitrary fractional power~$\gamma$ 
and show that, 
under the same conditions, 
the mild solution 
in the sense of Definition~\textup{\ref{def:frac-mild-sol-zeroIC}}  
is equivalent to the weak solution 
in the sense of Definition~\textup{\ref{def:frac-weak-sol-zeroIC}}. 

\begin{proposition}\label{prop:fractional-mild-equiv-weak-zeroIC}
	Suppose that Assumption~\textup{\ref{assumption:A}\ref{assumption:A:semigroup}} 
	holds and let $\gamma \in (0,\infty)$ be such that
	\eqref{eq:mildsol-cond-fractional} is satisfied. 
	Then, a stochastic process 
	is a mild solution 
	in the sense of Definition~\textup{\ref{def:frac-mild-sol-zeroIC}} 
	if and only if 
	it is a weak solution 
	in the sense of Definition~\textup{\ref{def:frac-weak-sol-zeroIC}}. 
	Moreover, mild and weak solutions are unique 
	up to modification. 
	If one requires continuity  
	of the sample paths, 
	mild and weak solutions 
	are unique up to indistinguishability.  
\end{proposition}

\begin{proof}
	First, we show 
	that a mild solution $Z_\gamma$ 
	is a weak solution. 
	Note that mean-square continuity follows 
	from Theorem~\ref{thm:existence-mild}. 
	Fix an arbitrary $\psi \in \mathsf D(\cB^{\gamma*})$. 
	Then, 
	\begin{align}
	( Z_\gamma&, \cB^{\gamma*}\psi )_{L^2(0,T;H)} 
	= 
	\frac{1}{\Gamma(\gamma)} 
	\int_0^T 
	\biggl (\int_0^t (t-s)^{\gamma-1} S(t-s) \rd W^Q(s), 
	[\cB^{\gamma*}\psi](t) \biggr)_H \rd t 
	\notag 
	\\
	&= \frac{1}{\Gamma(\gamma)} 
	\int_0^T \int_0^T 
	\bigl( \mathbf 1_{(0,t)}(s) (t-s)^{\gamma-1}S(t-s) \rd W^Q(s), 
	[\cB^{\gamma*}\psi](t) \bigr)_H \rd t 
	\label{eq:proof:prop:fractional-mild-equiv-weak-zeroIC}
	\end{align} 
	holds $\bbP$-a.s. 
	Here, we used that  
	$\bigl( \int_0^T f(s) \rd W^Q(s), x \bigr)_H 
	= \int_0^T (f(s) \rd W^Q(s), x)_H$ 
	for all 
	$f \from (0,T) \to \mathscr L(H)$ 
	and $x \in H$, which
	readily is derived from the definition 
	of the weak stochastic integral 
	and the continuity of inner products. 
	We now would like to apply the stochastic Fubini theorem, 
	see e.g.\ \cite[Theorem~8.14]{Peszat_zabczyk_2007}, 
	in order to interchange 
	the inner weak stochastic integral 
	and the outer deterministic integral. 
	Again by the definition of the weak stochastic integral 
	we have, 
	for a.a.\ $t\in(0,T)$, 
	\[ 
	\int_0^T 
	\bigl( \mathbf 1_{(0,t)}(s) (t-s)^{\gamma-1} S(t-s) \rd W^Q(s), 
	[\cB^{\gamma*}\psi](t) \bigr)_H 
	= \int_0^T \Psi(s,t) \rd W^Q(s), 
	\quad 
	\bbP\text{-a.s.},  
	\] 
	where the integrand 
	$\Psi(s,t)\from H \to \R$ 
	is deterministic and, 
	for $s,t \in (0,T)$,  
	defined~by 
	\begin{equation}\label{eq:def-weak-integrand} 
	\Psi(s,t) x 
	:= 
	\bigl( \mathbf 1_{(0,t)}(s) (t-s)^{\gamma-1} S(t-s) x, 
	[\cB^{\gamma*}\psi](t) \bigr)_H 
	\quad 
	\forall x\in H.   
	\end{equation}
	Thus, the usage of the stochastic Fubini theorem 
	is justified if  
	$t\mapsto \Psi(\,\cdot\,, t)Q^{\frac{1}{2}}$ 
	is in  
	$L^1(0,T; L^2(0,T;\mathscr L_2(H;\R)))$. 
	Given an orthonormal basis 
	$(g_j)_{j\in\N}$ for $H$, we obtain  
	\begin{align*}
	\bigl\| \Psi(s,t) Q^{\frac{1}{2}} \bigr\|_{\mathscr L_2(H;\R)}^2 
	&= 
	\sum_{j=1}^\infty 
	\bigl| \bigl( \mathbf 1_{(0,t)}(s) (t-s)^{\gamma-1}S(t-s) Q^\frac{1}{2} g_j, 
	[\cB^{\gamma*}\psi](t) \bigr)_H \bigr|^2 
	\\
	&\leq 
	\bigl\| \mathbf 1_{(0,t)}(s) (t-s)^{\gamma-1} S(t-s) Q^\frac{1}{2} 
	\bigr\|_{\mathscr L_2(H)}^2 
	\bigl\| [\cB^{\gamma*}\psi](t) \bigr\|_{H}^2 
	\end{align*}
	by the Cauchy--Schwarz inequality  
	on $H$. 
	From this, it follows that
	\begin{align*}
	\bigl\| 
	t\mapsto \Psi(\,\cdot\,,t) &Q^{\frac{1}{2}} \bigr\|_{
		L^1(0,T;L^2(0,T;\mathscr{L}_2(H;\R)))} 
	= 
	\int_0^T 
	\biggl( \int_0^T 
	\bigl\| \Psi(s,t) Q^{\frac{1}{2}} \bigr\|_{\mathscr L_2(H;\R)}^2 
	\rd s \biggr)^{\nicefrac{1}{2}}
	\rd t 
	\\
	&\leq 
	\int_0^T 
	\biggl( \int_0^t 
	\bigl\| (t-s)^{\gamma-1}S(t-s) Q^{\frac{1}{2}} \bigr\|_{\mathscr L_2(H)}^2
	\bigl\| [\cB^{\gamma*}\psi](t) \bigr\|_{H}^2 \rd s 
	\biggr)^{\nicefrac{1}{2}} 
	\rd t 
	\\
	&= 
	\int_0^T 
	\biggl(\int_0^t 
	\bigl\| s^{\gamma-1}S(s)Q^\frac{1}{2} \bigr\|_{\mathscr L_2(H)}^2  
	\rd s\biggr)^{\nicefrac{1}{2}} 
	\bigl\| [\cB^{\gamma*}\psi](t) \bigr\|_{H} \rd t 
	\\
	&\leq 
	T^{\nicefrac{1}{2}} 
	\norm{ \cB^{\gamma*}\psi }{L^2(0,T;H)}
	\biggl( 
	\int_0^T 
	\bigl\| s^{\gamma-1} S(s) Q^{\frac{1}{2}} \bigr\|_{\mathscr L_2(H)}^2  
	\rd s \biggr)^{ \nicefrac{1}{2} }  
	< \infty, 
	\end{align*}
	where we used the Cauchy--Schwarz inequality on $L^2(0,T)$ 
	in the last step. 
	Owing to \eqref{eq:mildsol-cond-fractional}, 
	the integral in the final expression is finite. 
	Applying the stochastic Fubini theorem 
	to \eqref{eq:proof:prop:fractional-mild-equiv-weak-zeroIC}, 
	taking adjoints in \eqref{eq:def-weak-integrand} 
	and using the continuity of $(\,\cdot\,,\,\cdot\,)_H$ 
	gives  
	\begin{align*}
	(Z_\gamma, \cB^{\gamma*}\psi &)_{L^2(0,T;H)} 
	= 
	\frac{1}{\Gamma(\gamma)}
	\int_0^T\int_0^T \Psi(s,t) \rd t\rd W^Q(s)
	\\
	&= 
	\int_0^T 
	\biggl( \rd W^Q(s), 
	\frac{1}{\Gamma(\gamma)}
	\int_s^T (t-s)^{\gamma-1}[S(t-s)]^*[\cB^{\gamma*}\psi](t) 
	\rd t\biggr)_H 
	\\
	&= 
	\int_0^T  
	\bigl( 
	\rd W^Q(s), 
	[\cB^{-\gamma*}\cB^{\gamma*}\psi](s) 
	\bigr)_H 
	= 
	\int_0^T \bigl( \rd W^Q(s), \psi(s) \bigr)_H, 
	\qquad 
	\bbP\text{-a.s.}, 
	\end{align*}
	where we used \eqref{eq:adjoint-neg-parabolic-formula} 
	in the third line. 
	Therefore, $Z_\gamma$ is a weak~solution.
	
	Conversely, suppose that $Y_\gamma$ is a weak solution, 
	let an arbitrary $\phi \in L^2(0,T;H)$ be given and 
	set $\psi := \cB^{-\gamma*} \phi \in \mathsf D(\cB^{\gamma*})$. 
	Substituting this into \eqref{eq:fractional-weaksol-zeroIC} 
	gives 
	\[ 
	( Y_\gamma, \phi )_{L^2(0,T;H)} 
	=  
	\int_0^T 
	\bigl( \rd W^Q(t), [\cB^{-\gamma*} \phi](t) \bigr)_H, 
	\quad 
	\bbP\text{-a.s.} 
	\] 
	Let $(\widetilde{Z}_\gamma(t))_{t\in[0,T]}$ be the stochastic convolution 
	in~\eqref{eq:stoch-conv}. 
	Since the condition for the stochastic Fubini theorem 
	still holds after replacing $\cB^{\gamma*}\psi$ 
	by $\phi$ in \eqref{eq:def-weak-integrand}, 
	the proof of the previous implication 
	can be read backwards to see that
	\[ 
	\forall 
	\phi \in L^2(0,T;H): 
	\quad 
	\bbP\bigl( 
	( Y_\gamma, \phi )_{L^2(0,T;H)} 
	= 
	( \widetilde{Z}_\gamma, \phi)_{L^2(0,T;H)} 
	\bigr) 
	=1. 
	\] 
	By separability of $H$, 
	also $\bbP( Y_\gamma = \widetilde{Z}_\gamma$ in $L^2(0,T;H) ) = 1$ 
	holds so that 
	by Fubini 
	$Y_\gamma = \widetilde{Z}_\gamma$ in $L^2(0,T;L^2(\Omega;H))$ 
	follows. 
	Since both $Y_\gamma$ and $\widetilde{Z}_\gamma$ 
	are mean-square continuous, this shows that, 
	for all $t\in[0,T]$, 
	$Y_\gamma(t) = \widetilde{Z}_\gamma(t)$ in $L^2(\Omega;H)$. 
	Therefore, for all $t\in[0,T]$, we have~that 
	$Y_\gamma(t) = \widetilde{Z}_\gamma(t)$, $\bbP$-a.s., 
	i.e., $Y_\gamma$ is a mild solution.  
	
	It thus suffices to prove uniqueness 
	only for mild solutions. 
	By Definition~\ref{def:frac-mild-sol-zeroIC},
	mild solutions are modifications of the stochastic convolution
	$\widetilde Z_\gamma$ in \eqref{eq:stoch-conv},
	hence of each other. If two mild solutions are
	moreover known to have continuous sample paths, then they are indistinguishable by
	\cite[Proposition~3.17]{Peszat_zabczyk_2007}. 
\end{proof}

\subsection{Spatiotemporal regularity of solutions}
\label{subsec:analysis-zeroIC:existence-uniqueness-regularity} 

We now investigate 
spatiotemporal regularity 
of the mild solution $Z_\gamma$ 
in Definition~\ref{def:frac-mild-sol-zeroIC}. 
We start by stating our 
main results, 
Theorem~\ref{thm:temporal-regularity} 
and Corollary~\ref{cor:temporal-regularity-pathwise}, 
in Subsection~\ref{subsubsec:main-results}. 
In Subsection~\ref{subsubsec:simplified} 
we derive 
a simplified condition 
for spatiotemporal regularity, 
which is easier to check in applications 
and 
sufficient whenever $A$ satisfies 
Assumptions~\ref{assumption:A}\ref{assumption:A:semigroup},\ref{assumption:A:bdd-Hinfty},\ref{assumption:A:bddly-inv}, 
see Proposition~\ref{prop:easier-condition-H-infty}. 
In addition, we explicitly discuss the setting of 
a Gelfand triple $V\hookrightarrow H \cong \dual H \hookrightarrow \dual V$ 
in the case that the operator $A$ is induced by a 
(not necessarily symmetric) bilinear form $\mathfrak{a} \from V\times V\to \bbR$
which is continuous and satisfies a G{\aa}rding inequality. 
Subsection~\ref{subsubsec:proof-temporal-reg} 
is devoted to 
the proof of 
Theorem~\ref{thm:temporal-regularity}. 

\subsubsection{Main results} 
\label{subsubsec:main-results} 

In Theorem~\ref{thm:temporal-regularity} below, 
temporal regularity is measured by the differentiability~$n\in\bbN_0$ 
as well as the H\"older exponent 
$\tau\in[0,1)$. 
Spatial regularity is expressed 
by means of vector spaces which are  
defined in terms of fractional powers of $A$
(see Subsection~\ref{subsec:appendix-functional-calc:frac-powers} 
in Appendix~\ref{app:functional-calc}) 
as follows:  
\[
\dot H_A^\sigma := \mathsf D\bigl( A^{\nicefrac{\sigma}{2}} \bigr), 
\qquad 
(x, y)_{\dot H_A^\sigma} 
:= 
\bigl( A^{\nicefrac{\sigma}{2}} x, A^{\nicefrac{\sigma}{2}} y \bigr)_H, \quad \sigma\in[0,\infty).
\] 
For $\sigma\in(0,\infty)$, $\dot H_A^\sigma$ is a Hilbert space 
provided that  
Assumptions~\ref{assumption:A}\ref{assumption:A:semigroup},\ref{assumption:A:bdd-analytic},\ref{assumption:A:bddly-inv} 
are satisfied. 
In this case, we have the embeddings   
$\Hdot{\sigma'} \hookrightarrow \Hdot\sigma\hookrightarrow H$ 
for all $\sigma'\geq\sigma\geq 0$. 
Note, in particular, that we do not need to assume 
that $A$ is self-adjoint. 

\begin{theorem}\label{thm:temporal-regularity}
	Suppose that 
	Assumptions~\textup{\ref{assumption:A}\ref{assumption:A:semigroup},\ref{assumption:A:bdd-analytic}} 
	are satisfied and let  
	$n \in \N_0$,  
	$\sigma\in[0,\infty)$  
	and  
	$\gamma \in \bigl( \frac{\sigma-r}{2} +n ,\infty \bigr)$, 
	where $r\in[0,\sigma]$ is such that 
	$Q^{\frac{1}{2}} \in \mathscr L(H;\dot H_A^r)$. 
	In the case that $\sigma\in(0,\infty)$, 
	suppose furthermore that 
	Assumption~\textup{\ref{assumption:A}\ref{assumption:A:bddly-inv}} 
	is fulfilled. 
	Under the condition  
	\begin{equation}\label{eq:stoch-conv-diffb-assump} 
	\int_0^T
	\bigl\| t^{\gamma-1-n} S(t) Q^{\frac{1}{2}} \bigr\|_{ 
		\mathscr L_2(H; \Hdot{\sigma})}^2 \rd t 
	< \infty,
	\end{equation} 
	the mild solution $Z_\gamma$ 
	(or, equivalently, the weak solution~$Y_\gamma$) 
	in the sense of 
	Definition~\textup{\ref{def:frac-mild-sol-zeroIC}} 
	(or~\textup{\ref{def:frac-weak-sol-zeroIC}}) 
	belongs to 
	$C^{n,0}([0,T];L^p(\Omega;\dot H^\sigma_A))$ 
	for every $p \in [1,\infty)$. 
	
	If additionally 
	$\gamma\geq n + \tau + \frac{1}{2}$ and 
	$A^{n+\tau+\frac{1}{2} -\gamma} Q^{\frac{1}{2}} \in \mathscr L_2(H; \dot H_A^\sigma)$ 
	hold for some $\tau\in (0,1)$, 
	then we have 
	$Z_\gamma\in C^{n,\tau}([0,T];L^p(\Omega;\dot H^\sigma_A))$ 
	for every $p \in [1,\infty)$. 
\end{theorem} 

An application of the Kolmogorov--Chentsov 
continuity theorem, 
see e.g.~\cite[Theorem~3.9]{Cox2021local}, 
allows us to (partially) transport 
the temporal regularity result of 
Theorem~\ref{thm:temporal-regularity} 
to the pathwise setting, 
as seen in the next corollary. 

\begin{corollary}\label{cor:temporal-regularity-pathwise}
	Suppose that 
	Assumptions~\textup{\ref{assumption:A}\ref{assumption:A:semigroup},\ref{assumption:A:bdd-analytic}} 
	are satisfied. Let  
	$\sigma\in[0,\infty)$, 
	$r\in[0,\sigma]$, 
	$\gamma \in \bigl(\frac{\sigma-r}{2},\infty\bigr)$ 
	and  
	$\tau \in (0,1)$ 
	be such that 
	$Q^{\frac{1}{2}} \in \mathscr L(H;\dot H_A^r)$ 
	and 
	$\gamma \geq \tau + \frac{1}{2}$. 
	If ${\sigma\in(0,\infty)}$, suppose
	also that 
	Assumption~\textup{\ref{assumption:A}\ref{assumption:A:bddly-inv}} 
	holds.  
	If the condition 
	\[ 
	\bigl\| A^{\tau+\frac{1}{2}-\gamma} Q^{\frac{1}{2}} 
	\bigr\|_{\mathscr L_2(H;\dot H_A^\sigma)}
	+ 
	\int_0^T
	\bigl\| t^{\gamma-1} S(t) Q^{\frac{1}{2}} 
	\bigr\|_{\mathscr L_2(H; \Hdot{\sigma})}^2 \rd t 
	< 
	\infty 
	\] 
	is satisfied, then 
	for all $p \in [1,\infty)$ 
	and every $\tau'\in[0,\tau)$ 
	there exists a modification 
	$\widehat{Z}_\gamma$ of 
	the mild solution $Z_\gamma$ 
	(or, equivalently, of the weak solution~$Y_\gamma$) 
	in the sense of 
	Definition~\textup{\ref{def:frac-mild-sol-zeroIC}}  
	(or~\textup{\ref{def:frac-weak-sol-zeroIC}}) 
	such that 
	$\widehat{Z}_\gamma$ 
	has $\tau'$-H\"older continuous 
	sample paths and 
	belongs to 
	$L^p\bigl( \Omega; C^{0,\tau'}([0,T]; \dot H^\sigma_A) \bigr)$. 
\end{corollary}

\begin{proof}
	We first invoke Theorem \ref{thm:temporal-regularity} 
	with $n = 0$ and 
	$\tau\in(0,1)$ 
	to establish that $Z_\gamma$ belongs to 
	$C^{0,\tau}([0,T]; L^q(\Omega;\dot H_A^\sigma))$ 
	for every $q \in[1,\infty)$. 
	The result then follows by choosing $q \geq 1$ 
	sufficiently large, applying the 
	Kolmogorov--Chentsov continuity theorem 
	(see e.g.\ \cite[Theorem~3.9]{Cox2021local}), 
	and using nestedness of the $L^p$ spaces. 
\end{proof}

\subsubsection{A simplified condition and its application to the G{\aa}rding inequality case} 
\label{subsubsec:simplified} 

Whenever also  
Assumption~\ref{assumption:A}\ref{assumption:A:bdd-Hinfty} holds,  
it is possible to replace 
the condition in~\eqref{eq:stoch-conv-diffb-assump} 
by one which is simpler to check in practice.
In this case, the operator $A$ satisfies 
\emph{square function estimates} 
(see Subsection~\ref{subsubsec:square-estimate} 
in Appendix~\ref{app:functional-calc}), 
one of which is used to prove the next result.
\begin{proposition}\label{prop:easier-condition-H-infty}
	Suppose that 
	Assumptions~\textup{\ref{assumption:A}\ref{assumption:A:semigroup},\ref{assumption:A:bdd-Hinfty},\ref{assumption:A:bddly-inv}} 
	are satisfied.
	Let  
	$\sigma, \delta \in [0, \infty)$
	and $\gamma \in \bigl(\frac{1}{2}+\delta,\infty\bigr) 
	\cap 
	\bigl[\frac{1}{2}+\delta+\frac{\sigma-r}{2},\infty\bigr)$, 
	where $r \in [0,\sigma]$ 
	is taken such that $Q^{\frac{1}{2}} \in \mathscr L(H; \dot H_A^r)$. 
	Then, 
	\[
	\int_0^\infty 
	\bigl\| t^{\gamma-1-\delta} S(t) Q^{\frac{1}{2}} \bigr\|_{
		\mathscr L_2(H; \Hdot{\sigma})}^2 \rd t 
	\eqsim_{(\gamma,\delta)} 
	\bigl\| A^{ \delta + \frac{1}{2} - \gamma } Q^{\frac{1}{2}} 
	\bigr\|_{\mathscr L_2(H;\Hdot \sigma)}^2.
	\] 
\end{proposition}

\begin{proof}
	Applying Lemma~\ref{lem:square-func-est}, 
	see Appendix~\ref{app:functional-calc}, 
	with $a := \gamma - \delta - \frac{1}{2} \in (0,\infty)$ and
	$x := A^{ \frac{\sigma}{2}+\delta+\frac{1}{2}-\gamma }Q^\frac{1}{2} y \in H$ 
	for $y\in H$ shows that  
	\[ 
	\int_0^\infty 
	\bigl\| t^{\gamma-1-\delta} A^{ \gamma-\delta-\frac{1}{2} } 
	S(t) A^{ \delta+\frac{1}{2}-\gamma } Q^\frac{1}{2} y \bigr\|_{\Hdot\sigma}^2 \rd t 
	\eqsim_{(\gamma,\delta)}
	\bigl\| A^{ \delta+\frac{1}{2}-\gamma } Q^\frac{1}{2} y \bigr\|_{\Hdot\sigma}^2  
	\quad 
	\forall y \in H.
	\]
	Summing both sides over an orthonormal basis for $H$ 
	and using the Fubini--Tonelli theorem to interchange integration 
	and summation on the left-hand side yields
	the desired conclusion.
\end{proof}

\begin{remark}\label{rem:easier-condition-H-infty} 
	Proposition~\ref{prop:easier-condition-H-infty} 
	shows that under the additional 
	assumption that $A_\C$ admits a 
	bounded $H^\infty$-calculus 
	with $\omega_{H^\infty}(A_\C)<\tfrac{\pi}{2}$,  
	which e.g.\ is satisfied whenever  
	$A$ is self-adjoint and strictly positive, 
	it suffices to check 
	that 
	$\gamma > n + \frac{ (\sigma-r) \vee 1 }{2} $,  
	$\gamma \geq n+\frac{1+ (\sigma-r)\vee (2\tau) }{2}$ 
	and that 
	the Hilbert--Schmidt 
	norm 
	$\| A^{ n+ \tau + \frac{1}{2} - \gamma } Q^{\frac{1}{2}} 
	\|_{\mathscr L_2(H;\Hdot \sigma)}$
	is bounded to conclude 
	the regularity results 
	of Theorem~\ref{thm:temporal-regularity}. 
	This condition coincides with 
	the one imposed in 
	\cite[Section~4,~Theorem~6]{Kovacs2013} 
	to derive regularity 
	in the non-fractional case 
	$\gamma=1$ for $p=2$, $\sigma=0$, 
	$n=0$~and~$\tau \in [0, \nicefrac{1}{2}]$. 
\end{remark} 

\begin{corollary}\label{cor:easier-condition-H-infty-Aplusdelta}
	Let $\delta \in [0,\infty)$
	and $\gamma \in\bigl(\frac{1}{2}+\delta , \infty\bigr)$.
	Suppose that $A$ satisfies
	Assumption~\textup{\ref{assumption:A}\ref{assumption:A:semigroup}} 
	and that there exists a constant $\eta \in[0,\infty)$
	such that $\widehat A := A + \eta I$ 
	satisfies 
	Assumptions~\textup{\ref{assumption:A}\ref{assumption:A:semigroup},\ref{assumption:A:bdd-Hinfty},\ref{assumption:A:bddly-inv}} 
	and 
	$\widehat A^{\delta + \frac{1}{2}-\gamma} Q^{\frac{1}{2}} 
	\in \mathscr L_2(H)$. 
	Then, the mild solution
	$Z_\gamma$ in the sense of Definition~\textup{\ref{def:frac-mild-sol-zeroIC}} 
	exists and belongs to $C([0,T];L^p(\Omega;H))$ 
	for every $p\in[1,\infty)$.
	If $\delta>0$, then 
	for every $p \in [1,\infty)$ 
	there exists a modification of $Z_\gamma$ 
	in $L^p(\Omega;C([0,T];H))$  
	which has continuous sample paths. 
\end{corollary}

\begin{proof}
	Note that 
	$S(t) = e^{\eta t} \widehat S(t)$ 
	holds for every $t\geq 0$, 
	where $( \widehat S(t) )_{t\geq 0}$ denotes
	the $C_0$-semigroup
	generated by $-\widehat A$. 
	Hence, by Proposition~\ref{prop:easier-condition-H-infty} 
	we find that 
	\begin{align*}
	\int_0^T 
	\bigl\| t^{\gamma-1-\delta} S(t) Q^{\frac{1}{2}} \bigr\|_{
		\mathscr L_2(H)}^2 \rd t 
	&\le
	e^{2 \eta T}
	\int_0^T 
	\bigl\| t^{\gamma-1-\delta} \widehat S(t) Q^{\frac{1}{2}} \bigr\|_{
		\mathscr L_2(H)}^2 \rd t
	\\ 
	&\lesssim_{(\gamma,\delta)} 
	e^{2 \eta T}
	\bigl\| \widehat A^{ \delta+\frac{1}{2}-\gamma } Q^{\frac{1}{2}} 
	\bigr\|_{\mathscr L_2(H)}^2 < \infty.
	\end{align*}
	The claim now follows from 
	Theorem~\ref{thm:existence-mild}.
\end{proof}

We illustrate the utility of Corollary~\ref{cor:easier-condition-H-infty-Aplusdelta}
in the following example. 
It is concerned with the 
case that the operator $A$ is induced by a continuous 
bilinear form $\mathfrak{a} \from V\times V\to \bbR$, 
where $V\hookrightarrow H$ is 
dense in $H$,  
and $\mathfrak{a}$ is not necessarily coercive on $V$;  
see also~\cite[Section 7.3.2]{Haase2006}. 
We note that this setting applies to a variety of  
important applications, including symmetric and non-symmetric 
differential operators of even orders. 

\begin{example}
	Let $(V, (\,\cdot\,, \,\cdot\,)_V)$ be a Hilbert space
	which is densely and continuously embedded in $H$.
	Suppose that $A \from \mathsf D(A)\subseteq H\to H$
	is induced by a bilinear form $\mathfrak a \from V\times V\to \R$ which
	is bounded and satisfies a G{\aa}rding inequality, i.e.,
	there exist constants $\alpha_0, \alpha_1 \in (0,\infty)$ and
	$\eta\in[0,\infty)$
	such that
	\begin{align}
	\qquad\quad  
	|\mathfrak a(u,v)| 
	&\leq 
	\alpha_1 \norm{u}{V}\norm{v}{V}
	&&
	\forall u,v\in V, \qquad\quad 
	\label{eq:bounded-form} 
	\\ 
	\qquad\quad 
	\mathfrak a(u,u) 
	&\geq 
	\alpha_0 \norm{u}{V}^2 - \eta \norm{u}{H}^2
	&&
	\forall u \in V. \qquad\quad 
	\label{eq:garding} 
	\end{align}
	The G{\aa}rding inequality~\eqref{eq:garding} can be interpreted 
	as coercivity of 
	the bilinear form $\hat{\mathfrak a}(u,v) := \mathfrak a(u,v) + \eta (u,v)_H$ on $V$, 
	associated with $\widehat A = A + \eta I$,
	while~\eqref{eq:bounded-form}
	implies that $\hat{\mathfrak a}$ is bounded. 
	The complexified sesquilinear form 
	$\hat{\mathfrak a}_{\C} \from V_\C \times V_\C \to \C$,
	which is defined analogously to~\eqref{eq:inner-product-HC} 
	and induces the operator 
	$\widehat A_\C$, inherits the boundedness and coercivity from $\hat{\mathfrak a}$. Thus,
	there exist $\widehat{\alpha}_0, \widehat{\alpha}_1 \in(0,\infty)$ such that 
	\begin{align*}
	\qquad\quad  
	|\hat{\mathfrak a}_\C(u,v)| 
	&\leq 
	\widehat{\alpha}_1 \norm{u}{V_\C}\norm{v}{V_\C}
	&&
	\forall u,v\in V_\C, 
	\qquad\quad  
	\\
	\qquad\quad  
	\Re{\hat{\mathfrak a}_\C(u,u)} 
	&\geq 
	\widehat{\alpha}_0 \norm{u}{V_\C}^2
	&&
	\forall u \in V_\C.
	\qquad\quad  
	\end{align*}
	Therefore, 
	$\widehat{\alpha}_0 \norm{u}{V_\C}^2 
	\leq 
	\Re{\hat{\mathfrak a}_\C(u,u)}
	\leq 
	|\hat{\mathfrak a}_\C (u,u)|
	\le
	\widehat{\alpha}_1 \norm{u}{V_\C}^2
	\leq
	\tfrac{\widehat{\alpha}_1}{\widehat{\alpha}_0} 
	\Re{\hat{\mathfrak a}_\C(u,u)}$
	follows for every $u \in V_\C$.
	If $V_\C \neq \{0\}$, these estimates imply 
	that $\widehat{\alpha}_0 \leq \widehat{\alpha}_1$
	and
	\[
	|\Im{\hat{\mathfrak a}_\C}(u,u)| 
	= 
	\sqrt{ |\hat{\mathfrak a}_\C(u,u)|^2 - | \Re{\hat{\mathfrak a}_\C}(u,u)|^2 }
	\leq 
	\Bigl( \tfrac{\widehat{\alpha}_1^2}{\widehat{\alpha}_0^2} - 1 \Bigr)^{\nicefrac{1}{2}} 
	\Re{\hat{\mathfrak a}_\C}(u,u)
	\quad
	\forall u \in V_\C.
	\]
	This shows that $-\widehat A_\C$ generates a bounded
	analytic
	$C_0$-semigroup $(\widehat{S}_\C(t))_{t\geq 0}$ 
	of contractions on $H_\C$,
	cf.\ \cite[Theorem~1.54]{Ouhabaz2005}, 
	where we used that
	$(-\infty,0) \subseteq \rho(\widehat A_\C)$ by \cite[Proposition~1.22]{Ouhabaz2005}.
	Applying \cite[Theorems~10.2.24 and~10.4.21]{AnalysisInBanachSpacesII} 
	and using that $\omega(\widehat A_\C) \in \bigl[0,\tfrac{\pi}{2} \bigr)$, 
	since $(\widehat{S}_\C(t))_{t\geq 0}$ is bounded analytic 
	(see Theorem~\ref{thm:sectoriality-analytic-semigroups}), 
	we find that $\widehat A_\C$ admits
	a bounded $H^\infty$-calculus of angle 
	$\omega_{H^\infty}(\widehat A_\C) 
	= 
	\omega(\widehat A_\C) \in \bigl[0,\tfrac{\pi}{2} \bigr)$.  
	Thus, we are in the setting 
	of Corollary~\ref{cor:easier-condition-H-infty-Aplusdelta}. 
	In particular,  
	the existence of a mean-square continuous  
	mild solution to \eqref{eq:fractional-parabolic-spde-zeroIC} 
	for $\gamma> \frac{1}{2}$
	follows if 
	$\| \widehat A^{\frac{1}{2}-\gamma} Q^{\frac{1}{2}} 
	\|_{\mathscr L_2(H)}<\infty$. 
\end{example}

\subsubsection{The proof of 
	\texorpdfstring{Theorem~\ref{thm:temporal-regularity}}{Theorem 3.12}} 
\label{subsubsec:proof-temporal-reg} 

We split the proof 
of Theorem~\ref{thm:temporal-regularity} 
into several intermediate results. 
Before stating and proving these, 
we introduce the following function, 
which generalizes the integrand  
in~\eqref{eq:stoch-conv} 
used to define mild solutions. 
Given 
$a\in\bbR$, $b \in [0,\infty)$ and 
$\sigma \in [0,\infty)$, 
define  
$\Phi_{a,b}\from (0,\infty) \to \mathscr L(H; \dot H_A^\sigma)$  
by
\begin{equation}\label{eq:def:Phi-ab} 
\Phi_{a,b} (t) := t^{a} A^b S(t)Q^\frac{1}{2}, 
\qquad 
t\in(0,\infty). 
\end{equation}
Note that   
a mild solution $Z_\gamma$ in the 
sense of Definition~\ref{def:frac-mild-sol-zeroIC} 
satisfies the relation 
\[
\forall t\in[0,T]: 
\quad 
Z_\gamma(t)
= 
\frac{1}{\Gamma(\gamma)} 
\int_0^t 
\Phi_{\gamma-1,0} (t-s) 
\rd \widehat{W}(s), 
\quad 
\text{$\bbP$-a.s.}, 
\]
where $\widehat{W}(t):=Q^{-\frac{1}{2}} W^Q(t)$, $t\geq 0$, 
is a cylindrical Wiener process. 

The first result quantifies  
spatial regularity of the  
continuous-in-time stochastic 
convolution with~$\Phi_{a,b}$ 
in $L^p(\Omega; \Hdot{\sigma})$-sense. 
Recall from Section~\ref{sec:prelims} that 
$(W(t))_{t\geq 0}$ 
denotes an (arbitrary) $H$-valued cylindrical Wiener process
with respect to
$(\cF_t)_{t\geq 0}$. 

\begin{proposition}\label{prop:continuity-in-time}
	Let 
	Assumption~\textup{\ref{assumption:A}\ref{assumption:A:semigroup}} 
	hold, and 
	let $a \in \R$, $b,\sigma \in [0,\infty)$  
	and $T \in (0,\infty)$ be given.
	If $\sigma\neq0$, then suppose moreover that 
	Assumptions~\textup{\ref{assumption:A}\ref{assumption:A:bdd-analytic},\ref{assumption:A:bddly-inv}}  
	are satisfied. 
	If the function $\Phi_{a,b}$
	defined in \eqref{eq:def:Phi-ab} 
	belongs to $L^2(0,T; \mathscr L_2(H;\dot{H}^\sigma_A))$, 
	i.e., 
	\[
	\int_0^T
	\norm{\Phi_{a,b}(t)}{\mathscr L_2(H;\dot H_A^\sigma)}^2 \rd t 
	< \infty, 
	\] 
	then $t\mapsto \int_0^t \Phi_{a,b}(t-s)\rd W(s)$ 
	belongs to $C([0,T]; L^p(\Omega;\dot H_A^\sigma))$ 
	for all $p \in [1,\infty)$. 
\end{proposition}

\begin{proof} 
	We first note that 
	the assumption 
	$\Phi_{a,b} \in L^2(0,T;\mathscr L_2(H;\dot{H}^\sigma_A))$, 
	combined with the 
	Burkholder--Davis--Gundy inequality 
	(see \cite[Theorem~6.1.2]{LiuRockner2015})
	and the continuous embedding 
	\begin{equation}\label{eq:Lp-embedding} 
	L^2(\Omega;\Hdot{\sigma}) \hookrightarrow 
	L^p(\Omega;\Hdot{\sigma}), 
	\qquad 
	p\in[1,2), \; 
	\sigma\in[0,\infty),
	\end{equation} 
	imply that 
	$\int_0^t \Phi_{a,b}(t-s)\rd W(s)$ 
	indeed is a well-defined element 
	of $L^p(\Omega;\dot H_A^\sigma)$ 
	for all $t\in [0,T]$ 
	and every $p\in[1,\infty)$.  
	
	It remains to check the $L^p(\Omega; \Hdot{\sigma})$-continuity 
	of $t \mapsto \int_0^t \Phi_{a,b}(t-s)\rd W(s)$. 
	For fixed $t \in [0,T)$ and $h \in (0, T-t]$, we split the 
	stochastic integrals as follows:
	\begin{align*} 
	&\int_0^{t+h} \Phi_{a,b}(t+h-s)\rd W(s) - \int_0^t \Phi_{a,b}(t-s)\rd W(s) 
	\\
	&\quad 
	= 
	\int_t^{t+h} \Phi_{a,b}(t+h-s) \rd W(s) 
	+ 
	\int_0^{t} [\Phi_{a,b}(t+h-s) - \Phi_{a,b}(t-s)] \rd W(s). 
	\end{align*}
	For $p \in [2,\infty)$, the Burkholder--Davis--Gundy inequality yields 
	\begin{align*} 
	&\biggl\| 
	\int_t^{t+h} \Phi_{a,b}(t+h-s) \rd W(s) 
	+ \int_0^{t} [\Phi_{a,b}(t+h-s) - \Phi_{a,b}(t-s)] \rd W(s) 
	\biggr\|_{L^p(\Omega;\Hdot\sigma)} 
	\\
	&\lesssim_p 
	\biggl[\int_t^{t+h} \norm{\Phi_{a,b}(t+h-s)}{
		\mathscr L_2(H;\dot H_A^\sigma)}^2 \rd s 
	\biggr]^{\nicefrac{1}{2}}  
	\\
	&\qquad + 
	\biggl[ \int_0^t 
	\norm{ \Phi_{a,b}(t+h-s) - \Phi_{a,b}(t-s) }{ 
		\mathscr L_2(H;\dot H_A^\sigma)}^2 \rd s 
	\biggr]^{\nicefrac{1}{2}} 
	\\
	&= 
	\biggl[ \int_0^h 
	\norm{\Phi_{a,b}(u)}{\mathscr L_2(H;\dot H_A^\sigma)}^2 \rd u 
	\biggr]^{\nicefrac{1}{2}} 
	+ 
	\biggl[ \int_0^t 
	\norm{\Phi_{a,b}(r+h) - \Phi_{a,b}(r)}{ 
		\mathscr L_2(H;\dot H_A^\sigma)}^2 \rd r 
	\biggr]^{\nicefrac{1}{2}}, 
	\end{align*} 
	where $u := t+h-s$ and $r := t-s$.
	Since $\Phi_{a,b}\in L^2(0,T;\mathscr L_2(H;\dot H_A^\sigma))$ 
	the first integral tends to zero as $h \downarrow 0$ 
	by dominated convergence. 
	The second term tends to zero by Lemma~\ref{lem:translations}, 
	see Appendix~\ref{app:auxiliary}. 
	
	For $t\in(0,T]$ and $h\in[-t, 0)$, 
	the difference of stochastic integrals 
	can be rewritten using $\int_0^t = \int_0^{t+h} + \int_{t+h}^t$. 
	Thus, we obtain, for every $p\in[2,\infty)$, the bound 
	\begin{align*}
	&\biggl\| 
	\int_0^{t+h} \Phi_{a,b}(t+h-s) \rd W(s) 
	- 
	\int_0^t \Phi_{a,b}(t-s)\rd W(s) 
	\biggr\|_{L^p(\Omega;\Hdot\sigma)} 
	\\ 
	&\lesssim_p 
	\biggl[ \int_0^{-h} 
	\norm{\Phi_{a,b}(r)}{\mathscr L_2(H;\dot H_A^\sigma)}^2 \rd r 
	\biggr]^{\nicefrac{1}{2}} 
	+ 
	\biggl[ \int_{-h}^t 
	\norm{\Phi_{a,b}(r+h) - \Phi_{a,b}(r)}{\mathscr L_2(H;\dot H_A^\sigma)}^2 
	\rd r \biggr]^{\nicefrac{1}{2}}, 
	\end{align*}
	where we again used the change of 
	variables $r := t-s$. 
	Both terms on the last line tend to zero, 
	again by dominated convergence and 
	Lemma~\ref{lem:translations}, respectively.
	
	Finally, we note 
	that the result for $p=2$ implies that 
	for $p\in[1,2)$ by~\eqref{eq:Lp-embedding}. 
\end{proof}

Furthermore, we obtain the following result 
regarding the temporal H\"older continuity of 
the stochastic 
convolution with the 
function~$\Phi_{a,b}$ 
in~\eqref{eq:def:Phi-ab}. 

\begin{proposition}\label{prop:regularity-stoch-conv}
	Let 
	Assumptions~\textup{\ref{assumption:A}\ref{assumption:A:semigroup},\ref{assumption:A:bdd-analytic}} 
	hold, 
	let $T \in (0,\infty)$, 
	${a \in \bigl(-\frac{1}{2},\infty\bigr)}$, ${b, \sigma \in [0,\infty)}$ 
	and $\tau\in\bigl( 0,a + \frac{1}{2} \bigr] \cap (0,1)$. 
	If $\sigma \neq 0$, then suppose also that
	Assumption~\textup{\ref{assumption:A}\ref{assumption:A:bddly-inv}} holds. 
	If
	$A^{-a-\frac{1}{2} +b+\tau} Q^{\frac{1}{2}} \in \mathscr L_2(H; \dot H_A^\sigma)$ 
	and $\Phi_{a,b}$ is defined  
	by~\eqref{eq:def:Phi-ab},
	then
	$t \mapsto \int_0^t \Phi_{a,b}(t-s)\rd W(s)$
	belongs to $C^{0,\tau}([0,T];L^p(\Omega;\dot H_A^\sigma))$ 
	for all $p \in [1,\infty)$. 
\end{proposition}

\begin{proof}
	For $t \in [0,T)$ and $h \in (0,T-t]$, we 
	obtain
	\begin{align*}
	\biggl\| 
	&\int_0^{t+h} \Phi_{a,b}(t+h -s) \rd W(s) - \int_0^t \Phi_{a,b}(t-s)\rd W(s) 
	\biggr\|_{L^p(\Omega; \Hdot \sigma)} 
	\\
	&\leq  
	\biggl\| \int_0^t \bigl[ \Phi_{a,b}(t+h-s) - \Phi_{a,b}(t-s) \bigr] \rd W(s) 
	\biggr\|_{L^p(\Omega;\Hdot \sigma)} 
	\\
	&\quad + 
	\biggl\| \int_t^{t+h} \Phi_{a,b}(t+h-s) \rd W(s) \biggr\|_{L^p(\Omega; \Hdot \sigma)} 
	\lesssim_{(p,a,\tau)} 
	h^\tau 
	\bigl\| A^{-a-\frac{1}{2} +b+\tau} Q^{\frac{1}{2}} \bigr\|_{\mathscr L_2(H;\Hdot \sigma)} 
	\end{align*}
	by Lemmas~\ref{lemma:I-1}~and~\ref{lemma:I-2}, 
	see Appendix~\ref{app:auxiliary}. 
	The analogous result for the case that 
	$t\in(0,T]$ and $h\in[-t,0)$  
	follows upon splitting 
	$\int_0^t = \int_0^{t+h} + \int_{t+h}^t$ 
	and applying the lemmas with 
	$\bar t := t+h \in [0,T)$ and 
	$\bar h := -h \in (0,T-\bar t\,]$.
\end{proof}

We now investigate temporal mean-square differentiability. 
To this end, we need the following estimate which 
is implied by \eqref{eq:analytic-est-1}: 
For all $a\in\R$, $b\in[0,\infty)$, we have  
\begin{equation}\label{eq:analytic-est-Phi} 
\forall c \in [0,\infty): 
\quad 
\norm{\Phi_{a,b}(t)x}{H}
\lesssim_{c} 
t^{a-c} 
\bigl\| A^{b-c} Q^{\frac{1}{2}} x \bigr\|_{H}
\quad 
\forall x\in \mathsf D\bigl( A^{b-c} Q^{\frac{1}{2}} \bigr). 
\end{equation}
The next lemma records some information 
about the derivatives of~$\Phi_{a,b}$ in~\eqref{eq:def:Phi-ab}.

\begin{lemma}\label{lem:diffbty-of-Phi}
	Let Assumptions~\textup{\ref{assumption:A}\ref{assumption:A:semigroup},\ref{assumption:A:bdd-analytic}} 
	be satisfied, 
	and let $a\in\R$, $b,\sigma\in [0,\infty)$.
	If $\sigma\in(0,\infty)$, 
	suppose furthermore that 
	Assumption~\textup{\ref{assumption:A}\ref{assumption:A:bddly-inv}} holds.
	Then, 
	the function $\Phi_{a,b}$ 
	defined by~\eqref{eq:def:Phi-ab}
	belongs 
	to $C^\infty((0,\infty); \mathscr L(H; \dot H_A^\sigma))$  
	with $k$th derivative 
	\begin{equation}\label{eq:deriv-operator-suggestive} 
	\frac{\mathrm d^k}{\mathrm d t^k} \, 
	\Phi_{a,b} (t)
	= 
	\sum_{j=0}^k C_{a,j,k} t^{a -(k-j)} A^{b+j} S(t)Q^\frac{1}{2} 
	= 
	\sum_{j=0}^k C_{a,j,k} \Phi_{a-(k-j),b+j}(t),  
	\end{equation} 
	where 
	$C_{a,j,k} := 
	(-1)^j \binom{k}{j} \prod_{i=1}^{k-j} (a-(k-j)+i)$ 
	for $a\in\bbR$, $j,k\in\bbN_0$, $j\leq k$. 

	Moreover, if $r\in[0, 2b+\sigma ]$ is such that 
	$Q^{\frac{1}{2}} \in \mathscr L(H;\dot H_A^r)$ and 
	$n\in\N_0$ satisfies $n < a - b - \tfrac{\sigma-r}{2}$, 
	then $\Phi_{a,b}$ has 
	a continuous extension in 
	$C^n( [0,\infty); \mathscr L(H;\Hdot\sigma))$
	with all $n$ derivatives vanishing at zero.
\end{lemma}

\begin{proof}
	Since $(S(t))_{t\ge0}$ is assumed
	to be analytic, 
	$S(\,\cdot\,)$ is infinitely differentiable
	from $(0,\infty)$ 
	to $\mathscr L(H)$,
	with $j$th derivative $(-A)^j S(\,\cdot\,)$ and, 
	for $t \in (0,\infty)$, $\varepsilon := \frac{t}{2}$, 
	\begin{align*}
	\bigl[ A^{ b + \frac{\sigma}{2} } S(\,\cdot\,) \bigr]^{(j)}(t) 
	&= 
	\bigl[ S( \,\cdot\, - \varepsilon) A^{ b+ \frac{\sigma}{2} } 
	S( \varepsilon ) \bigr]^{(j)}(t) 
	\\ 
	&= 
	(-A)^j S(t-\varepsilon) A^{ b + \frac{\sigma}{2} } S(\varepsilon)  
	= 
	(-1)^j A^{ j + b + \frac{\sigma}{2} } S(t).  
	\end{align*}
	Here, the limits for the derivatives 
	are taken in the $\mathscr L(H)$ norm. 
	This is equivalent to  
	$[ A^b S(\,\cdot\,) ]^{(j)}(t) 
	= (-1)^j A^{j+b} S(t)$
	with respect to the 
	$\mathscr L(H;\dot H_A^\sigma)$ norm. 
	The expression 
	for the $k$th derivative of $\Phi_{a,b}$ 
	thus follows from the Leibniz rule. 
	
	Now let 
	$r\in[0, 2b+\sigma ]$, 
	$n\in\N_0$ be such that $n < a-b-\frac{\sigma-r}{2}$ and 
	$Q^{\frac{1}{2}} \in \mathscr L(H;\dot H_A^r)$. 
	To prove the second claim, 
	we derive that for all $k \in \{0,1,\dots,n\}$ 
	and $t \in (0,\infty)$
	\begin{align*} 
	\biggl\| 
	\frac{\mathrm d^k}{\mathrm d t^k} \, 
	\Phi_{a,b} (t) 
	\biggr\|_{\mathscr L(H;\Hdot\sigma)} 
	&= 
	\biggl\| 
	\sum_{j=0}^k C_{a,j,k} t^{a -(k-j)} 
	A^{b+j+\frac{\sigma-r}{2}} S(t) A^{\frac{r}{2}} Q^\frac{1}{2} 
	\biggr\|_{\mathscr L(H)} 
	\\
	&\lesssim_{(a,b,k,r,\sigma)} 
	t^{a-k-b-\frac{\sigma-r}{2}}  
	\bigl\| Q^{\frac{1}{2}} \bigr\|_{\mathscr L(H; \dot H_A^r)} 
	\end{align*} 
	by 
	applying \eqref{eq:analytic-est-Phi} to each summand 
	with $c:=b+j+\frac{\sigma-r}{2}\geq 0$. 
	Furthermore, since 
	$a-k-b-\frac{\sigma-r}{2}\geq a-n-b-\frac{\sigma-r}{2}>0$, 
	the above quantity tends to zero as $t \downarrow 0$. 
	Hence, extending 
	$t\mapsto \frac{\mathrm d^k}{\mathrm d t^k} \, \Phi_{a,b}(t)$ 
	by zero at $t=0$  
	gives a function in $C([0,\infty); \mathscr L(H; \Hdot\sigma))$ 
	for all $k \in \{0,1,\dots,n\}$. 
	Inductively 
	it follows then 
	that the $k$th derivative of the zero extension is 
	the zero extension of the original $k$th derivative. 
\end{proof}

\begin{proposition}\label{prop:stoch-conv-Psi-diffb-under-int}
	Let $\sigma\in[0,\infty)$, and 
	whenever ${\sigma\in(0,\infty)}$ 
	require additionally 
	Assumptions~\textup{\ref{assumption:A}\ref{assumption:A:semigroup},\ref{assumption:A:bdd-analytic},\ref{assumption:A:bddly-inv}}. 
	Suppose that 
	$\Psi \in H^1_{0,\{0\}}(0, T; \mathscr L_2(H; \Hdot{\sigma}))$ 
	and let $\Psi'$ denote its weak derivative.
	Then, for every $p \in [1,\infty)$, the stochastic convolution  
	$t \mapsto \int_0^t \Psi(t-s)\rd W(s)$ is differentiable 
	from $[0,T]$ to $L^p(\Omega; \Hdot \sigma)$, with derivative 
	\begin{align}
	\frac{\rd}{\rd t}\int_0^{t} \Psi(t-s)\rd W(s) 
	= 
	\int_0^t \Psi'(t-s) \rd W(s) 
	\quad 
	\forall \,
	t \in [0,T].
	\label{eq:stoch-conv-Psi-diffb-under-int-sign}
	\end{align}
\end{proposition}

\begin{proof}
	For $t \in [0,T)$ and $h \in (0,T-t]$, 
	we can write 
	\begin{align*}
	&\frac{1}{h} \biggl[ 
	\int_0^{t+h} \Psi(t+h-s)\rd W(s) - \int_0^{t} \Psi(t-s)\rd W(s) \biggr] 
	- \int_0^t \Psi'(t-s) \rd W(s) 
	\\
	&= 
	\int_0^t \biggl[ 
	\frac{\Psi(t+h-s) - \Psi(t-s)}{h} - \Psi'(t-s) \biggr] \rd W(s) 
	+ 
	\frac{1}{h}\int_t^{t+h} \Psi(t+h-s)\rd W(s) 
	\\
	&=: 
	I^{h^+}_{1} + I^{h^+}_{2}\!.
	\end{align*}
	For $t\in(0,T]$ and $h \in [-t,0)$, we instead have
	\begin{align*}
	\frac{1}{h}\biggl[ 
	&\int_0^{t+h} \Psi(t+h-s)\rd W(s) - \int_0^{t} \Psi(t-s)\rd W(s)\biggr] 
	- \int_0^t \Psi'(t-s) \rd W(s) 
	\\
	&= 
	\int_0^{t+h} 
	\biggl[ \frac{\Psi(t+h-s)-\Psi(t-s)}{h} - \Psi'(t-s)\biggr] \rd W(s) 
	\\ 
	&\quad 
	- \frac{1}{h}\int_{t+h}^t \Psi(t-s) \rd W(s) - \int_{t+h}^t \Psi'(t-s)\rd W(s) 
	=: I_1^{h^-} + I_2^{h^-} + I_3^{h^-}\!.
	\end{align*}
	We first deal with the terms $I_2^{h^\pm}$\!.
	Note that
	$\Psi \in H^1_{0,\{0\}}(0, T; \mathscr L_2(H; \Hdot{\sigma}))$
	implies $\Psi(u) = \int_0^u \Psi'(r)\rd r$ for 
	all $u \in (0, |h|)$, see \cite[\S5.9.2, Theorem~2]{Evans2010}.
	In conjunction with the Burkholder--Davis--Gundy inequality 
	(combined with the embedding~\eqref{eq:Lp-embedding}
	if $p\in[1,2)$) and the Cauchy--Schwarz inequality,
	this leads to 
	\begin{align*}
	\bigl\| I_2^{h^\pm} 
	&\bigr\|_{L^p(\Omega;\Hdot\sigma)} 
	\lesssim_{p} 
	\frac{1}{|h|} 
	\biggl[ \int_0^{|h|} 
	\norm{\Psi(u)}{\mathscr L_2(H;\Hdot\sigma)}^2 \rd u \biggr]^{\nicefrac{1}{2}}  
	\\
	&\leq 
	\frac{1}{|h|} 
	\biggl[ \int_0^{|h|} 
	\biggl(\int_0^u \norm{\Psi'(r)}{\mathscr L_2(H;\Hdot\sigma)} \rd r\biggr)^2 
	\rd u\biggr]^{\nicefrac{1}{2}} 
	\leq
	\norm{\Psi'}{L^2(0,|h|;\mathscr L_2(H;\Hdot\sigma))}.
	\end{align*}
	Moreover, we find that 
	\begin{align*} 
	\bigl\| I_3^{h^-} \bigr\|_{L^p(\Omega;\Hdot \sigma)} 
	&\lesssim_p 
	\biggl[\int_{t+h}^t 
	\norm{\Psi'(t-s)}{\mathscr L_2(H; \Hdot{\sigma})}^2 \rd s\biggr]^{\nicefrac{1}{2}} 
	\\
	&= 
	\biggl[\int_{0}^{|h|} 
	\norm{\Psi'(u)}{\mathscr L_2(H; \Hdot{\sigma})}^2 \rd u 
	\biggr]^{\nicefrac{1}{2}} 
	= 
	\norm{\Psi'}{L^2(0,|h|;\mathscr L_2(H; \Hdot{\sigma}))}.
	\end{align*}
	Since $\Psi' \in L^2(0,T; \mathscr L_2(H; \Hdot{\sigma}))$, 
	we have that $\norm{\Psi'}{L^2(0,|h|;\mathscr L_2(H; \Hdot{\sigma}))} \to 0$ 
	as $h \to 0$ by dominated convergence. 
	Thus, it remains to deal with the $I_1^{h^\pm}$ terms. 
	For the case of positive $h$, we find 
	using the definition of the difference quotient $D_h$ 
	(see Equation~\eqref{eq:def:Dh} in Subsection~\ref{app:subsec:orbital} 
	of Appendix~\ref{app:auxiliary}) 
	that 
	\begin{align*} 
	\bigl\| I_1^{h^+} \bigr\|_{L^p(\Omega;\Hdot \sigma)} 
	&\lesssim_p 
	\biggl[ \int_0^t 
	\biggl\| \frac{\Psi(t+h-s) - \Psi(t-s)}{h} - \Psi'(t-s) 
	\biggr\|_{\mathscr L_2(H; \Hdot{\sigma})}^2 
	\rd s\biggr]^{\nicefrac{1}{2}} 
	\\
	&= \biggl[\int_0^t 
	\biggl\| \frac{\Psi(u+h) - \Psi(u)}{h} - \Psi'(u) 
	\biggr\|_{\mathscr L_2(H; \Hdot{\sigma})}^2 \rd u \biggr]^{\nicefrac{1}{2}} 
	\\
	&= \norm{D_h \Psi - \Psi'}{L^2(0,t; \mathscr L_2(H; \Hdot{\sigma}))}.
	\end{align*} 
	For the case of negative $h$, we arrive at
	\begin{align*}
	\bigl\| I_1^{h^-}\bigr\|_{L^p(\Omega;\Hdot \sigma)} 
	&\lesssim_p 
	\biggl[\int_0^{t+h} 
	\biggl\| \frac{\Psi(t+h-s) - \Psi(t-s)}{h} - \Psi'(t-s) 
	\biggr\|_{\mathscr L_2(H; \Hdot{\sigma})}^2\rd s \biggr]^{\nicefrac{1}{2}} 
	\\
	&= \biggl[\int_{-h}^t 
	\biggl\| \frac{\Psi(u+h) - \Psi(u)}{h} - \Psi'(u) 
	\biggr\|_{\mathscr L_2(H; \Hdot{\sigma})}^2\rd u \biggr]^{\nicefrac{1}{2}} 
	\\
	&= \norm{D_h \Psi - \Psi'}{L^2(-h,t; \mathscr L_2(H; \Hdot{\sigma}))}.
	\end{align*}
	The convergence 
	$\lim_{h\to 0}\norm{I_1^{h^\pm}}{L^p(\Omega;\Hdot \sigma)}=0$ 
	follows then from Proposition~\ref{prop:Dh-converges-in-Lp}. 
\end{proof}

We are now ready to prove Theorem \ref{thm:temporal-regularity}.

\begin{proof}[Proof of Theorem \ref{thm:temporal-regularity}]
	We first claim that 
	the mild solution, 
	interpreted as a mapping 
	$Z_\gamma\from[0,T]\to L^p(\Omega;\Hdot{\sigma})$,
	is $n$ times 
	differentiable 
	and that, for every $k \in \{0,1,\dots,n\}$ 
	and all $t\in[0,T]$, 
	its $k$th derivative satisfies 
	\begin{equation}\label{eq:stoch-conv-diffb-under-int-sign}
	Z_\gamma^{(k)}(t) 
	= 
	\frac{1}{\Gamma(\gamma)} 
	\int_0^t \Phi_{\gamma-1,0}^{(k)}(t-s) \rd \widehat{W}(s), 
	\quad 
	\text{$\bbP$-a.s.}, 
	\end{equation} 
	where $\Phi_{\gamma-1,0}^{(k)}$ 
	is the $k$th derivative of $\Phi_{\gamma-1,0}$ 
	given by~\eqref{eq:deriv-operator-suggestive}, 
	and $\widehat{W}$ 
	is the cylindrical Wiener process
	$\widehat{W}(t) := Q^{-\frac{1}{2}} W^Q(t)$, $t\geq 0$. 
	We prove this by induction with respect to $k$.
	For $k=0$, the 
	identity
	\eqref{eq:stoch-conv-diffb-under-int-sign} follows from 
	Definition~\ref{def:frac-mild-sol-zeroIC} 
	and 
	\eqref{eq:def:Phi-ab}. 
	Now let 
	${k \in \{0,1,\dots,n-1\}}$ 
	and 
	suppose that  
	$Z_\gamma$ is $k$ times differentiable and  \eqref{eq:stoch-conv-diffb-under-int-sign} 
	holds. 
	Then, the induction hypothesis 
	and Lemma~\ref{lem:diffbty-of-Phi} show that, 
	for all $t\in[0,T]$,
	\begin{align*}
	\frac{\mathrm d^{k+1}}{\rd t^{k+1}} \, 
	Z_\gamma(t) 
	&=
	\frac{\mathrm d}{\rd t} \, 
	Z_\gamma^{(k)}(t) 
	= 
	\frac{\rd}{\rd t} 
	\biggl[  
	\frac{1}{\Gamma(\gamma)}
	\int_0^t \Phi^{(k)}_{\gamma-1,0}(t-s) \rd \widehat{W}(s) 
	\biggr] 
	\\
	&= 
	\frac{1}{\Gamma(\gamma)} 
	\frac{\rd}{\rd t} 
	\int_0^t \sum_{j=0}^k 
	C_{\gamma-1,j,k} \Phi_{\gamma-1-(k-j),j}(t-s) \rd \widehat{W}(s),
	\quad 
	\text{$\bbP$-a.s.} 
	\end{align*}
	Fixing an arbitrary $j \in \{0,1,\dots,k\}$, 
	it suffices to verify that $\Psi := \Phi_{\gamma-1-(k-j),j}$  
	satisfies the conditions of 
	Proposition~\ref{prop:stoch-conv-Psi-diffb-under-int}, 
	so that \eqref{eq:stoch-conv-Psi-diffb-under-int-sign} holds 
	for the cylindrical Wiener process $\widehat{W}$. 
	Indeed, having proved this for an arbitrary $j$, 
	by linearity 
	\begin{align*}
	\frac{\mathrm d^{k+1}}{\rd t^{k+1}} \, 
	Z_\gamma(t) 
	&= 
	\frac{1}{\Gamma(\gamma)}
	\int_0^t \sum_{j=0}^k C_{\gamma-1,j,k} \Phi_{\gamma-1-(k-j),j}'(t-s) \rd \widehat{W}(s) 
	\\
	&= 
	\frac{1}{\Gamma(\gamma)}
	\int_0^t\Phi_{\gamma-1,0}^{(k+1)}(t-s) \rd \widehat{W}(s), 
	\quad 
	\text{$\bbP$-a.s.}, 
	\end{align*} 
	follows, 
	where the latter identity is an equality of the operator-valued integrands. 
	
	Using
	\eqref{eq:analytic-est-1} with $c:=b$, 
	the identity 
	$A^{\frac{\sigma}{2}} \Phi_{a,b}(t) 
	= 
	2^a (t/2)^a A^b S(t/2) A^{\frac{\sigma}{2}} S(t/2) Q^{\frac{1}{2}}$
	and a change of variables $u:=t/2$, we observe that
	\begin{equation}\label{eq:moving-b-to-a}
	\begin{split} 
	\| \Phi_{a,b} 
	&\|_{L^2(0,T;\mathscr L_2(H;\dot H_A^\sigma))} 
	\lesssim_{(a,b)}  
	\biggl[ 
	\int_0^T
	(t/2)^{2(a-b)} 
	\bigl\| A^{\frac{\sigma}{2}} S(t/2) Q^{\frac{1}{2}} \bigr\|_{\mathscr L_2(H)}^2 
	\frac{\rd t}{2} 
	\biggr]^{\nicefrac{1}{2}} 
	\\
	&= 
	\biggl[ 
	\int_0^T
	\| \Phi_{a-b,0}(t/2) \|_{\mathscr L_2(H;\dot H_A^\sigma)}^2 
	\frac{\rd t}{2}
	\biggr]^{\nicefrac{1}{2}} 
	\le 
	\norm{\Phi_{a-b,0}}{L^2(0,T;\mathscr L_2(H;\dot H_A^\sigma))} 
	\end{split} 
	\end{equation}
	holds 
	for all 
	$a\in\R$ and 
	$b\in[0,\infty)$. 
	For $\Psi=\Phi_{\gamma-1-(k-j),j}$ 
	we use \eqref{eq:moving-b-to-a} to obtain
	\[
	\norm{\Psi}{L^2(0,T;\mathscr L_2(H;\dot H_A^\sigma))} 
	\lesssim _{(\gamma,k,j)} 
	\norm{\Phi_{\gamma-1-k,0}}{L^2(0,T;\mathscr L_2(H;\dot H_A^\sigma))}.
	\] 
	The norm on the right-hand side is finite 
	by~\eqref{eq:stoch-conv-diffb-assump}, since 
	$k \leq n - 1 < n$.
	Next, noting that 
	$\gamma-1-k-\frac{\sigma-r}{2} 
	\geq 
	\gamma - n - \frac{\sigma-r}{2} > 0$, 
	the second assertion of Lemma~\ref{lem:diffbty-of-Phi} 
	implies that $t \mapsto \Psi(t)$ 
	has a continuous extension in 
	$C_{0,\{0\}}([0,T]; \mathscr L(H;\dot H_A^\sigma))$.
	Furthermore, also by Lemma~\ref{lem:diffbty-of-Phi},
	$\Psi$ is differentiable from 
	$(0,T)$ to $\mathscr L(H;\Hdot\sigma)$, 
	with derivative 
	\[
	\Psi' 
	= 
	(\gamma-1-(k-j))\Phi_{\gamma-1-(k-j)-1, j} 
	- 
	\Phi_{\gamma-1-(k-j), j+1}.
	\] 
	Applying the 
	triangle inequality and \eqref{eq:moving-b-to-a} 
	then shows that 
	\[ 
	\norm{\Psi'}{L^2(0,T;\mathscr L_2(H;\dot H_A^\sigma))} 
	\lesssim _{(\gamma,k,j)}
	\norm{\Phi_{\gamma-1-(k+1),0}}{L^2(0,T;\mathscr L_2(H;\dot H_A^\sigma))}, 
	\] 
	where the norm on the right-hand side is finite 
	by~\eqref{eq:stoch-conv-diffb-assump},  
	as ${k + 1 \leq n }$. 
	Since 
	$\mathscr L_2(H;\Hdot\sigma) 
	\hookrightarrow
	\mathscr L(H;\Hdot\sigma)$, 
	Lemma~\ref{lem:embedding-weak-derivs} implies that 
	$\Psi \in H^1_{0,\{0\}}(0,T; \mathscr L_2(H;\Hdot\sigma))$.
	Thus, we may indeed use 
	Proposition~\ref{prop:stoch-conv-Psi-diffb-under-int}, 
	and the differentiability follows.  
	
	It remains to show that the $n$th derivative 
	$Z_\gamma^{(n)}$  
	is (H\"older) continuous, i.e., 
	$Z_\gamma^{(n)} \in C^{0,\tau}([0,T]; L^p(\Omega; \dot H_A^\sigma))$. 
	To this end, we use \eqref{eq:stoch-conv-diffb-under-int-sign} 
	and \eqref{eq:deriv-operator-suggestive},  and write
	\[ 
	\forall t\in [0,T]: 
	\quad 
	Z_\gamma^{(n)}(t) 
	= 
	\frac{1}{\Gamma(\gamma)}
	\sum_{j=0}^n 
	C_{\gamma-1,j,n} \int_0^t \Phi_{\gamma-1-(n-j),j}(t-s) \rd \widehat{W}(s), 
	\quad 
	\bbP\text{-a.s.} 
	\]  
	The case $\tau=0$ 
	(i.e., continuity) follows after applying, 
	for all $j \in \{0,1,\ldots,n\}$, 
	Proposition~\ref{prop:continuity-in-time} 
	with $a=\gamma-1-(n-j)$ and $b = j$.
	Note that  
	$\Phi_{\gamma-1-(n-j),j}$ indeed 
	is an element of 
	$L^2(0,T;\mathscr L_2(H;\dot H_A^\sigma))$ 
	for all $j \in \{0,\ldots,n\}$ by \eqref{eq:stoch-conv-diffb-assump} 
	and \eqref{eq:moving-b-to-a}. 
	For $\tau \in \bigl(0,\gamma-n-\frac{1}{2}\bigr] \cap (0,1)$, 
	the H\"older continuity of 
	$Z_\gamma^{(n)}$ follows from 
	Proposition~\ref{prop:regularity-stoch-conv} 
	which we may apply, for all $j \in \{0,1,\ldots,n\}$, 
	with $a=\gamma-1-(n-j)$ and $b = j$,  
	since 
	$ A^{n+\tau+\frac{1}{2}-\gamma} Q^{\frac{1}{2}} \in \mathscr L_2(H;\dot H_A^\sigma)$ 
	is assumed. 
\end{proof}

\section{Covariance structure}
\label{sec:covariance-structure}

In this section, we study the covariance 
structure of solutions 
to~\eqref{eq:fractional-parabolic-spde-zeroIC}. 
More specifically, 
we consider the mild solution process
$(Z_\gamma(t))_{t\in[0,T]}$ 
from Definition~\ref{def:frac-mild-sol-zeroIC}.
The covariance structure
of $Z_\gamma$ will be expressed in terms
of the family of 
\emph{covariance operators} 
$(Q_{Z_\gamma}(s,t))_{s,t\in [0,T]} \subseteq \mathscr L(H)$  
which satisfies, for all $s,t\in[0,T]$, that 
\[ 
( Q_{Z_\gamma}(s,t) x, y)_H   
= 
\E[ (Z_\gamma(s)-\E[Z_\gamma(s)], x)_H( Z_\gamma(t)-\E[Z_\gamma(t)], y)_H ] 
\quad 
\forall x,y\in H. 
\] 
Note that this family is well-defined  
whenever $Z_\gamma$ is square-integrable, e.g., 
under the assumptions made in Theorem~\ref{thm:existence-mild}. 
Note also that $\E[Z_\gamma(t)] = 0$ 
for all $t\in [0,T]$. 

We present three results on the covariance operators
of the mild solution $Z_\gamma$. 
The most general result is
Proposition~\ref{prop:covariance-mild-solution},
which provides an explicit 
integral representation of $Q_{Z_\gamma}(s,t)$. 
Corollary~\ref{cor:asymptotic-marginal-spatial-covariance} 
is concerned with the asymptotic behavior of the covariance 
operator $Q_{Z_\gamma}(t,t)$ as $t\to\infty$. 
Subsequently, in Corollary~\ref{cor:separable-cov} 
we consider 
a situation in which the covariance is separable 
in time and space, 
and prove that the temporal part is asymptotically
of Mat\'ern type.

\begin{proposition}\label{prop:covariance-mild-solution}
	Let 
	Assumption~\textup{\ref{assumption:A}\ref{assumption:A:semigroup}} 
	be satisfied and $\gamma \in (0,\infty)$ be such that~\eqref{eq:mildsol-cond-fractional} holds.
	The covariance operators $(Q_{Z_\gamma}(s,t))_{s,t\in[0,T]}$ 
	of $Z_\gamma$ admit 
	the representation 
	\begin{equation}\label{eq:covariance-operator-Zt-Zs}
	Q_{Z_\gamma}(s,t) 
	= 
	\frac{1}{\Gamma(\gamma)^2}
	\int_0^{s\wedge t}
	[(s-r)(t-r)]^{\gamma-1} S(t-r) Q [S(s-r)]^*
	\rd r. 
	\end{equation}
\end{proposition}
\begin{proof}
	Square-integrability of $Z_\gamma$ 
	is a consequence of Theorem~\ref{thm:existence-mild} 
	and~\eqref{eq:mildsol-cond-fractional}. 
	In order to prove 
	the integral representation~\eqref{eq:covariance-operator-Zt-Zs}, 
	for $s\in[0,T]$, $r\in(0,s)$ and $x\in H$, we 
	define $f(s,r;x) \in \mathscr L(H;\R)$ by 
	\[
	f(s,r;x) z 
	:= 
	[\Gamma(\gamma)]^{-1} 
	(z, (s-r)^{\gamma-1} [S(s-r)]^* x)_H, 
	\quad 
	z\in H. 
	\]
	We proceed 
	similarly as in~\cite[Lemma~3.10]{KirchnerLangLarsson2017} 
	and obtain~\eqref{eq:covariance-operator-Zt-Zs}  
	from the It\^o isometry combined with the 
	polarization identity: 
	\begin{align*}
	\E[ (Z_\gamma(s), x)_H &( Z_\gamma(t), y)_H ] 
	= 
	\E\left[  
	\int_0^s f(s,r;x) \rd W^Q(r)  
	\int_0^t  f(t,\tau;y) \rd W^Q(\tau) 
	\right] 
	\\
	&= 
	\int_0^{s\wedge t} 
	\bigl( 
	f(s,r;x) Q^{\frac{1}{2}}, 
	f(t,r;y) Q^{\frac{1}{2}} 
	\bigr)_{\mathscr L_2(H;\R)} \rd r 
	\\
	&= 
	\frac{1}{\Gamma(\gamma)^2} 
	\int_0^{s\wedge t} 
	[(s-r)(t-r)]^{\gamma-1} 
	( S(t-r) Q [S(s-r)]^* x, y )_H 
	\rd r . 
	\end{align*}
	Then,~\eqref{eq:covariance-operator-Zt-Zs} follows 
	from exchanging the order of integration 
	and taking the inner product, 
	which is justified since 
	$(0,s\wedge t)\ni r \mapsto [(s-r)(t-r)]^{\gamma-1} 
	S(t-r) Q [S(s-r)]^* x$ is integrable 
	by \eqref{eq:mildsol-cond-fractional}. 
\end{proof}
By imposing more assumptions  
on the operator $A$, 
one can obtain explicit representations 
of the asymptotic covariance
structure of 
$Z_\gamma$ 
as $t\to\infty$, 
as the next~two corollaries show.
Note that, 
if \eqref{eq:mildsol-cond-fractional} 
holds for $\delta=0$ and $T=\infty$, 
in Definition~\ref{def:frac-mild-sol-zeroIC} 
the stochastic convolution 
$\widetilde Z_\gamma$  
and the mild solution $Z_\gamma$ 
are well-defined on 
the infinite time interval $[0,\infty)$.  
It is thus 
meaningful to consider the asymptotic behavior.  
\begin{corollary}\label{cor:asymptotic-marginal-spatial-covariance}
	Let 
	Assumptions~\textup{\ref{assumption:A}\ref{assumption:A:semigroup},\ref{assumption:A:bdd-analytic},\ref{assumption:A:bddly-inv}} 
	be satisfied 
	and let 
	${\gamma \in (\nicefrac{1}{2},\infty)}$. 
	Suppose that \eqref{eq:mildsol-cond-fractional} 
	holds for $\delta=0$ and $T=\infty$. 
	If for every $t \in [0,\infty)$ 
	the operator $S(t)$ is self-adjoint and commutes with
	the covariance operator 
	$Q$ of $W^Q$\!, we have  
	\[
	\lim_{t\to\infty} Q_{Z_\gamma}(t,t) 
	= 
	\Gamma(\gamma-\nicefrac{1}{2})
	\bigl[ 2\sqrt{\pi} \Gamma(\gamma) \bigr]^{-1} 
	A^{1-2\gamma} Q 
	\quad 
	\text{in }
	\mathscr L(H). 
	\]
\end{corollary}

\begin{proof}
	Starting from 
	the identity~\eqref{eq:covariance-operator-Zt-Zs}
	for a fixed $t = s \in [0,\infty)$, 
	we recall self-adjointness of the operators 
	$(S(t))_{t\geq 0}$ and 
	the commutativity with $Q$ 
	to obtain~that 
	\begin{align*}
	Q_{Z_\gamma}(t,t) 
	&= 
	\frac{1}{\Gamma(\gamma)^2}
	\int_0^t 
	(t-r)^{2(\gamma-1)} S(t-r) Q S(t-r) 
	\rd r 
	\\
	&= 
	\frac{1}{\Gamma(\gamma)^2}
	\int_0^t 
	(t-r)^{2\gamma-2} Q S(2t-2r) 
	\rd r
	=
	\frac{2^{1-2\gamma}}{\Gamma(\gamma)^2}
	\int_0^{2t} u^{2\gamma-2} Q S(u) \rd u, 
	\end{align*}
	where we also used  
	the semigroup property and 
	the change of variables ${u := 2(t-r)}$. 
	Now we interchange the bounded
	linear operator $Q$ 
	with the integral, 
	and pass to the limit $t \to \infty$
	in $\mathscr L(H)$, which 
	by~\eqref{eq:fractional-power-exp-bdd-semigroup}
	with $\alpha := 2\gamma - 1 \in (0,\infty)$ gives 
	\[
	\lim_{t\to\infty} Q_{Z_\gamma}(t,t) 
	= 
	2^{1-2\gamma}\Gamma(2\gamma-1)
	[\Gamma(\gamma)]^{-2} 
	A^{1-2\gamma} Q
	= 
	\Gamma(\gamma-\nicefrac{1}{2})
	\bigl[ 2\sqrt{\pi} \Gamma(\gamma) \bigr]^{-1}
	A^{1-2\gamma} Q.
	\] 
	The last equality follows by applying the 
	Legendre duplication formula for the gamma 
	function (see e.g.~\cite[Formula~5.5.5]{Olver2010})
	to $\Gamma(2\gamma-1) = \Gamma(2[\gamma-\nicefrac{1}{2}])$.
\end{proof}

\begin{corollary}\label{cor:separable-cov}
	Suppose the setting of 
	Corollary~\textup{\ref{cor:asymptotic-marginal-spatial-covariance}} and 
	let $A := \kappa I$ for ${\kappa \in (0,\infty)}$. 
	Then the covariance function of $Z_\gamma$ is 
	separable and its temporal part 
	is asymptotically of Matérn type, 
	i.e., there is a function 
	$\varrho_{Z_\gamma}\colon [0,\infty)\times[0,\infty) \to \R$ 
	such that  
	\begin{align}
	&\forall s,t \in [0,\infty):
	&&
	Q_{Z_\gamma}(s,t) 
	=
	\varrho_{Z_\gamma}(s,t)\, Q , 
	\notag
	\\ 
	&\forall h \in \bbR\setminus\{0\}: 
	&&
	\lim_{t\to\infty}\varrho_{Z_\gamma}(t,t+h)
	= 
	\frac{2^{\frac{1}{2}-\gamma}\kappa^{1-2\gamma}}
	{\sqrt{\pi}\Gamma(\gamma)}
	(\kappa |h|)^{\gamma-\frac{1}{2}} 
	K_{\gamma-\frac{1}{2}}(\kappa |h|) . 
	\label{eq:temporal-asymp-matern-Q}  
	\end{align}
\end{corollary}

\begin{remark}
	On the right-hand side 
	of~\eqref{eq:temporal-asymp-matern-Q}, one recognizes the
	Mat\'ern covariance function~\eqref{eq:Matern-Cov} 
	with smoothness parameter $\nu = \gamma-\nicefrac{1}{2}$,
	correlation length parameter~$\kappa$ and variance
	$\sigma^2 
	= 
	\kappa^{1-2\gamma} \Gamma(\gamma-\nicefrac{1}{2}) 
	\bigl[ 2 \sqrt{\pi} \Gamma(\gamma) \bigr]^{-1}$\!. 
\end{remark} 

\begin{proof}[Proof of Corollary~\ref{cor:separable-cov}] 
	For $s,t \geq 0$,
	the integral representation~\eqref{eq:covariance-operator-Zt-Zs} 
	yields 
	\[ 
	Q_{Z_\gamma}(s,t) 
	= 
	\frac{1}{\Gamma(\gamma)^2}
	\int_0^{s\wedge t}
	[(s-r)(t-r)]^{\gamma-1} e^{-\kappa(s+t-2r)}
	\rd r
	\, Q 
	= \varrho_{Z_\gamma}(s,t) \, Q,
	\] 
	where we moved the bounded operator $Q\in \mathscr L(H)$ 
	out of the integral.	
	Next, we fix $h \in (0,\infty)$, let $t \in [0,\infty)$ 
	and perform the change of 
	variables ${u :=h + 2(t-r) }$, 
	\[ 
	\varrho_{Z_\gamma}(t,t+h)   
	=
	\varrho_{Z_\gamma}(t+h,t)  
	= 
	\frac{2^{1-2\gamma}}{\Gamma(\gamma)^2}
	\int_h^{2t+h}
	[(u+h)(u-h)]^{\gamma-1} e^{-\kappa u}
	\rd u. 
	\] 
	Thus, by passing to the limit $t \to \infty$, 
	we obtain  
	\begin{align*}
	&\lim_{t\to\infty} \varrho_{Z_\gamma}(t,t+h)
	= 
	\frac{2^{1-2\gamma}}{\Gamma(\gamma)^2}
	\int_h^{\infty}
	\bigl( u^2 - h^2 \bigr)^{\gamma-1} e^{-\kappa u}
	\rd u
	\\ 
	&\quad = 
	\frac{2^{1-2\gamma}}
	{\Gamma(\gamma)^2} 
	\cL
	\bigl[
	u
	\mapsto
	\bigl( u^2-h^2 \bigr)^{\gamma-1} 
	\mathbf 1_{(h,\infty)}(u)
	\bigr]
	(\kappa) 
	= 
	\frac{2^{1-2\gamma}}
	{\Gamma(\gamma)^2} 
	\frac{(2h)^{\gamma-\frac{1}{2}} 
		\Gamma(\gamma)}{\sqrt{\pi}\kappa^{\gamma-\frac{1}{2}}} 
	K_{\gamma-\frac{1}{2}}(\kappa h), 
	\end{align*}
	where $\cL[f](\kappa)$ denotes the Laplace transform of
	the function $f \from [0,\infty)\to\R$ evaluated at $\kappa$,
	and the last identity follows 
	from~\cite[Chapter I, Formula~3.13]{Oberhettinger1973}. 
\end{proof}

\section{Spatiotemporal Whittle--Mat\'ern fields}
\label{sec:Whittle-Matern} 

In this section, 
we demonstrate how the results of the previous  
Sections~\ref{sec:analysis-zeroIC}~and~\ref{sec:covariance-structure} 
can be related to the widely used statistical models 
involving generalized Whittle--Mat\'ern operators \eqref{eq:DO} 
on ${H = L^2(\cX)}$,
where $\cX=\cD\subsetneq\R^d$ is a bounded domain 
in the Euclidean space 
(see Subsection~\ref{subsec:example:bdd-euclidean})
or a surface $\cX=\cM$ 
(see Subsection~\ref{subsec:example:surfaces}).

\subsection{Bounded Euclidean domains}
\label{subsec:example:bdd-euclidean}

Throughout this subsection, let 
${\emptyset \neq \cD \subsetneq \R^d}$ 
be a bounded, 
connected and open 
domain.  
In order to rigorously define the 
symmetric, strongly elliptic second-order  
differential operator $L$,
formally given by~\eqref{eq:DO},
as a linear operator on $L^2(\cD)$,
we make the following assumptions 
on its coefficients  
$\kappa\from \cD \to \R$ and 
$a\from \cD\to\R^{d\times d}_{\rm sym}$, 
as well as on the spatial domain 
$\cD\subsetneq \R^d$. 

\begin{assumption}[Euclidean domain---minimal conditions]\label{assumption:minimal}
	\ 
	
	\begin{enumerate}[label=(\roman*), leftmargin=1cm] 
		\item\label{assumption:minimal-D} 
		$\cD$ has a Lipschitz continuous boundary $\partial\cD$;
		\item\label{assumption:minimal-a} 
		$a \in L^\infty\bigl( \cD; \R^{d\times d}_{\rm sym} \bigr)$ 
		is strongly elliptic, i.e., 
		\[ 
		\exists \, \theta > 0 : 
		\qquad 
		\operatornamewithlimits{ess\,inf}_{x\in \cD} \, \xi^\top a(x) \xi \geq \theta \norm{\xi}{\R^d}^2
		\qquad
		\forall \xi \in \R^d;
		\] 
		\item\label{assumption:minimal-kappa}  
		$\kappa \in L^\infty(\cD)$. 
	\end{enumerate}
\end{assumption}

Under these assumptions, 
we introduce the bilinear form
\[ 
\mathfrak{a}_L \from H_0^1(\cD)\times H_0^1(\cD) \to \R, 
\qquad
\mathfrak{a}_L(u,v) := 
(a\nabla u, \nabla v)_{L^2(\cD)}
+
(\kappa^2 u,v)_{L^2(\cD)}, 
\] 
which is symmetric, continuous and coercive. 
We say that $u \in H^1_0(\cD)$~belongs to the domain 
$\mathsf D(L)$ of the differential operator $L$ if and only if 
$|\mathfrak{a}_L(u,v) | 
\lesssim_u
\norm{v}{L^2(\cD)}$
holds for all
$v \in H^1_0(\cD)$. 
In this case, we define
$Lu$ as the unique element of $L^2(\cD)$ 
which satisfies the relation  
$\mathfrak{a}_L(u,v) = (Lu,v)_{L^2(\cD)}$ 
for all $v \in H^1_0(\cD)$.

By the Lax--Milgram theorem
the inverse $L^{-1} \in \mathscr L(L^2(\cD); H_0^1(\cD)) $ exists and can be  
extended to 
$L^{-1} \in \mathscr L(H_0^1(\cD)^*; H_0^1(\cD))$. 
Moreover, 
it is a consequence of the Rellich--Kondrachov 
theorem~(see \cite[Theorem~6.3]{AdamsFournier2003}) 
that $L^{-1}$ is compact~on~$L^2(\cD)$. 
For this reason, the spectral theorem 
for self-adjoint compact operators 
is applicable and shows that there exist
an orthonormal basis $(e_j)_{j\in \N}$ for $L^2(\cD)$ and
a non-decreasing sequence $(\lambda_j)_{j\in\N}$ 
of positive real numbers accumulating only at 
infinity such that $L e_j = \lambda_j e_j$ holds for all $j\in\N$. 
Furthermore, the eigenvalues 
of $L$ satisfy the following asymptotic behavior, 
known as Weyl's law 
\cite[Theorem~6.3.1]{Davies}:
\begin{equation}\label{eq:weyl}
\lambda_j 
\eqsim
j^{\nicefrac{2}{d}}
\quad 
\forall j \in \N.
\end{equation}
In this setting, for two differential operators 
$L$ and $\widetilde{L}$ on $L^2(\cD)$ 
with coefficients~$a,\kappa$ 
and $\widetilde{a},\widetilde{\kappa}$, respectively, 
we obtain the following corollary from the regularity results in
Section~\ref{sec:analysis-zeroIC} 
for spatiotemporal Whittle--Mat\'ern 
fields, where $A:=L^\beta$ and $Q:=\widetilde{L}^{-\alpha}$\!. 

\begin{corollary}\label{cor:example-1a-zeroIC} 
	Let $\alpha,\beta,\sigma\in[0,\infty)$, 
	set 
	$r:=\frac{\alpha}{\beta} \wedge \sigma$ if $\beta > 0$ 
	and $r:=\sigma$ if $\beta=0$, 
	and suppose that
	$n \in \N_0$, $\tau \in [0,1)$  
	and 
	$\gamma \in\bigl( n + \frac{(\sigma-r)\vee 1}{2}, \infty\bigr)$ 
	are
	such that
	\begin{equation}\label{eq:conditions-exponents-example1a}
	\gamma 
	\geq 
	n + \tfrac{1+(\sigma-r)\vee(2\tau)}{2} 
	\quad 
	\text{and} 
	\quad 
	\beta\gamma 
	> 
	\tfrac{d}{4} - \tfrac{\alpha}{2} + \beta\bigl( n+\tau+\tfrac{1 + \sigma}{2} \bigr)
	.
	\end{equation}
	
	Let $L\from\mathsf{D}(L) \subseteq H^1_0(\cD) \to L^2(\cD)$ and 
	${\widetilde{L}\from\mathsf{D}(\widetilde{L}) \subseteq H^1_0(\cD) \to L^2(\cD)}$ 
	be symmetric, strongly elliptic second-order differential
	operators  
	as defined above, cf.~\eqref{eq:DO}. 
	Suppose that 
	Assumption~\textup{\ref{assumption:minimal}\ref{assumption:minimal-D}}  
	holds for $\cD\subsetneq\R^d$\!,  
	and that 
	the coefficients 
	$a,\kappa$ of $L$ and $\widetilde a, \widetilde \kappa$ of $\widetilde L$
	satisfy  Assumptions~\textup{\ref{assumption:minimal}\ref{assumption:minimal-a},\ref{assumption:minimal-kappa}}. 
	Assume further  
	that 
	$L$ and~$\widetilde L$ 
	diagonalize 
	with respect to the same orthonormal basis $(e_j)_{j\in\bbN}$ 
	for $L^2(\cD)$, i.e., 
	there exist non-decreasing sequences 
	$(\lambda_j)_{j\in\N}$, $(\widetilde{\lambda}_j)_{j\in\N}$ of
	positive real numbers
	such that ${Le_j = \lambda_j e_j}$ 
	and $\widetilde L e_j = \widetilde{\lambda}_j e_j$
	for all $j \in \N$. 
	
	Then, setting $A := L^\beta$ and $Q := \widetilde L^{-\alpha}$\!, 
	the mild solution $Z_\gamma$ 
	to~\eqref{eq:fractional-parabolic-spde-zeroIC} in the sense of 
	Definition~\textup{\ref{def:frac-mild-sol-zeroIC}}, 
	see also \eqref{eq:spde-fractional-parabolic}, 
	belongs to 
	$C^{n,\tau}([0, T]; L^p(\Omega;\Hdot \sigma))$ 
	for all~${p \in [1,\infty)}$. 
	If the above conditions hold with $n = 0$ and 
	$\tau\in(0,1)$, 
	then for every $p\in[1,\infty)$ 
	and all $\tau'\in[0,\tau)$ the mild solution 
	$Z_\gamma$ has a modification 
	$\widehat Z_\gamma \in
	L^p\bigl(\Omega; C^{0,\tau'}([0,T];\dot H^\sigma_A)\bigr)$.
\end{corollary}

\begin{proof}
	By the spectral mapping theorem for fractional powers of operators, 
	see e.g.\ \cite[Section 5.3]{Carracedo2001}, we obtain that 
	$Ae_j = L^\beta e_j = \lambda_j^\beta e_j$ 
	and 
	${Qe_j =  \widetilde{L}^{-\alpha} e_j = \widetilde \lambda_j^{-\alpha} e_j}$. 
	In particular, it follows that $A$ inherits the self-adjointness and strict
	positive-definiteness
	from $L$. 
	This readily implies that $0\in \rho(A)$. 
	By~\cite[Proposition~10.2.23]{AnalysisInBanachSpacesII} 
	we see that $A_\C$ admits a bounded $H^\infty$-calculus
	of angle $\omega_{H^\infty}(A_\C)=0$,
	showing that
	Assumptions~\ref{assumption:A}\ref{assumption:A:semigroup}--\ref{assumption:A:bddly-inv}
	are satisfied for $A$. 
	
	Furthermore, we note that, 
	for every $\sigma,s\in[0,\infty)$, 
	we have that 
	$\dot H_A^\sigma = \dot H_L^{\sigma\beta}$ 
	and  
	the spaces $\dot H_L^s$ 
	and $\dot H_{\widetilde L}^s$ are isomorphic. 
	The latter fact follows from 
	the asymptotic behavior \eqref{eq:weyl}
	of the eigenvalues 
	$(\lambda_j)_{j\in\N}$ 
	and 
	$(\widetilde{\lambda}_j)_{j\in\N}$, 
	since $L$ 
	and $\widetilde{L}$ have the same eigenfunctions. 
	Thus, we obtain that 
	$Q^{\frac{1}{2}} = \widetilde{L}^{-\frac{\alpha}{2}} 
	\in \mathscr L(H;\dot H_L^{\alpha}) 
	\subseteq 
	\mathscr L(H; \dot H_A^r)$. 
	
	Since 
	$\gamma\in \bigl(\frac{1}{2}+n,\infty\bigr)\cap 
	\bigl[\frac{1}{2}+n+\frac{\sigma-r}{2},\infty\bigr)$ 
	is assumed, 
	by
	Proposition~\ref{prop:easier-condition-H-infty} 
	(see also Remark~\ref{rem:easier-condition-H-infty})
	the condition~\eqref{eq:stoch-conv-diffb-assump} 
	of Theorem~\ref{thm:temporal-regularity} 
	is equivalent to 
	requiring that 
	$A^{n+\frac{1}{2}-\gamma}Q^{\frac{1}{2}} \in 
	\mathscr L_2(H;\dot H_A^\sigma)$. 
	Since also 
	$\gamma\in\bigl( \frac{\sigma-r}{2} + n, \infty\bigr) 
	\cap \bigl[n + \tau +\frac{1}{2}, \infty\bigr)$, 
	we therefore conclude with Theorem~\ref{thm:temporal-regularity} 
	that it suffices to check that 
	the quantity
	\begin{equation}\label{eq:HS-norm-to-bound}  
	\begin{split} 
	\bigl\| & A^{\frac{\sigma}{2} + n + \tau + \frac{1}{2} - \gamma } 
	Q^{\frac{1}{2}} 
	\bigr\|_{\mathscr L_2(H)}^2 
	= 
	\bigl\| 
	L^{\beta\left( \frac{\sigma}{2}+n+\tau + \frac{1}{2}-\gamma \right)} 
	\widetilde L^{-\frac{\alpha}{2}} 
	\bigr\|_{\mathscr L_2(H)}^2 
	\\
	&\qquad = 
	\sum_{j=1}^\infty 
	\bigl\| 
	L^{\beta \left(\frac{\sigma}{2}+n+\tau + \frac{1}{2}-\gamma \right)}
	\widetilde L^{-\frac{\alpha}{2}} e_j 
	\bigr\|_H^2 
	= 
	\sum_{j=1}^\infty 
	\lambda_j^{2\beta\left( \frac{\sigma}{2}+n+\tau + \frac{1}{2}-\gamma \right)}
	\widetilde{\lambda}_j^{-\alpha} 
	\end{split} 
	\end{equation}
	is finite.
	Indeed, applying Weyl's law~\eqref{eq:weyl}
	to both $L$ and $\widetilde L$, it follows that 
	\[
	\sum_{j=1}^\infty 
	\lambda_j^{2\beta\left( \frac{\sigma}{2}+n+\tau + \frac{1}{2}-\gamma \right)}
	\widetilde{\lambda}_j^{-\alpha} 
	\eqsim_{\text(\alpha,\beta,\gamma,\sigma,n,\tau)}
	\sum_{j=1}^\infty 
	j^{\frac{4}{d} \left[ \beta\left( n+\tau + \frac{1+\sigma}{2} \right) 
		-\beta\gamma -\frac{\alpha}{2}\right]},
	\] 
	so that~\eqref{eq:HS-norm-to-bound} is finite if and only if 
	\eqref{eq:conditions-exponents-example1a} holds, as we assume.
	Then, 
	for all $p \in [1,\infty)$, 
	Theorem~\ref{thm:existence-mild}, 
	Theorem~\ref{thm:temporal-regularity} 
	and Proposition~\ref{prop:easier-condition-H-infty} 
	yield the existence of a mild solution 
	$Z_\gamma\in C^{n,\tau}([0,T]; L^p(\Omega;\Hdot \sigma))$, 
	which is unique up to modification. 
	The last assertion for $n = 0$ and 
	$\tau\in(0,1)$ follows from
	Corollary~\ref{cor:temporal-regularity-pathwise}.
\end{proof}

The spatial regularity obtained in Corollary~\ref{cor:example-1a-zeroIC} 
is measured using the spaces ${\Hdot\sigma = \dot H_L^{\beta\sigma}}$\!. 
It would be more practical to express this
in terms of fractional-order Sobolev spaces $H^s(\cD)$, $s\geq 0$. 
This raises the question 
of how $\dot H_L^s$ and $H^s(\cD)$ relate. 
The answer to this question depends on the 
smoothness of the coefficients $a,\kappa$ 
and of the boundary~$\partial \cD$.
We therefore introduce two additional sets of assumptions:  
Assumption~\ref{assumption:Lipschitz-convex}
is only slightly more restrictive than the 
minimal conditions of Assumption~\ref{assumption:minimal}, 
whereas Assumption~\ref{assumption:smooth}
requires a high degree of smoothness.

\begin{assumption}[Euclidean domain---$H^2(\cD)$-regular setting] 
	\label{assumption:Lipschitz-convex}
	\
	
	\begin{enumerate}[label=(\roman*), leftmargin=1cm] 
		\item 
		$\cD$ is convex.
		\item $a \from \overline\cD \to\R_\text{sym}^{d\times d} $
		is Lipschitz continuous, i.e.,
		\[ 
		|a_{ij}(x) - a_{ij}(y)| 
		\lesssim 
		\norm{x-y}{\R^d} 
		\quad \forall x,y\in\overline \cD,  
		\quad 
		\forall i,j \in \{1,\dots,d\}.
		\] 
	\end{enumerate}
\end{assumption}

\begin{assumption}[Euclidean domain---smooth setting]\label{assumption:smooth}
	\
	
	\begin{enumerate}[label=(\roman*), leftmargin=1cm] 
		\item 
		The boundary $\partial \cD$ is of class $C^\infty$;
		\item 
		$a_{ij} \in C^\infty(\overline\cD)$ 
		holds for all $i,j\in\{1,\ldots,d\}$, i.e., 
		for all entries of $a$; 
		\item 
		$\kappa \in C^\infty(\overline\cD)$. 
	\end{enumerate}
\end{assumption}

The results of the next lemma are taken 
from \cite[Lemma~2]{Cox2020regularity} 
and \cite[Lemma~3.4]{Bolin2023equivalence}.

\begin{lemma}\label{lem:dot-sobolev}
	Let $L\from\mathsf{D}(L) \subseteq H^1_0(\cD) \to L^2(\cD)$ 
	be a symmetric second-order 
	differential operator 
	as defined as above, cf.~\eqref{eq:DO}. 
	Then, the following assertions hold:
	\begin{enumerate}[label=(\alph*), leftmargin=1cm] 
		\item\label{lem:dot-sobolev-a}
		If Assumption~\textup{\ref{assumption:minimal}} is satisfied, 
		then $\dot H_L^s \hookrightarrow H^s(\cD)$ 
		for all ${s \in [0,1]}$. Moreover, 
		the norms $\| \,\cdot\, \|_{\dot H_L^s}$ 
		and 
		$\| \,\cdot\, \|_{H^s(\cD)}$ 
		are equivalent on 
		$\dot H_L^s$ for 
		$s \in [0,1]\setminus\{\nicefrac{1}{2}\}$; 
		\item\label{lem:dot-sobolev-b} 
		If Assumptions~\textup{\ref{assumption:minimal}} 
		and~\textup{\ref{assumption:Lipschitz-convex}} 
		are fulfilled, then
		\[ 
		\bigl( \dot H_L^s, \| \,\cdot\, \|_{\dot H_L^s} \bigr) 
		\cong 
		\bigl( 
		H^s(\cD) \cap H_0^1(\cD),
		\|\,\cdot\,\|_{H^s(\cD)} 
		\bigr) 
		\quad \forall s \in [1,2];
		\] 
		\item\label{lem:dot-sobolev-c} 
		If Assumptions~\textup{\ref{assumption:minimal}}  
		and~\textup{\ref{assumption:smooth}} are satisfied, 
		then we have $\dot H_L^s \hookrightarrow H^s(\cD)$ for all $s \in [0,\infty)$, 
		and the norms $\| \,\cdot\, \|_{\dot H_L^s}$, 
		$\|\,\cdot\,\|_{H^s(\cD)}$ 
		are equivalent on 
		$\dot H_L^s$ for every $s \in [0,\infty)\setminus \mathfrak{E}$, 
		where $\mathfrak{E} 
		:= 
		\{ 2k + \nicefrac{1}{2} : k \in \N_0 \}$ 
		is called the \emph{exclusion set}.
	\end{enumerate}
\end{lemma}

Combining Lemma~\ref{lem:dot-sobolev} 
with the results of 
Corollary~\ref{cor:example-1a-zeroIC} shows 
that the mild solution $Z_\gamma$ is an element of 
$C^{n,\tau}([0,T]; L^p(\Omega; H^{\beta\sigma}(\cD)))$, 
provided that $\sigma\beta\in[0,s']$, 
where $s'\in[1,\infty)$ is prescribed by the 
smoothness of the coefficients 
$a,\kappa$ 
and the boundary $\partial\cD$ 
via Lemma~\ref{lem:dot-sobolev}\ref{lem:dot-sobolev-a}, 
\ref{lem:dot-sobolev-b} or \ref{lem:dot-sobolev-c}. 
Note that we do not have to take the exclusion set
$\mathfrak{E}$ into account, 
as we only need the embedding 
$\dot H_L^s \hookrightarrow H^s(\cD)$.

Lastly, we consider the covariance structure 
of the mild solution,
as treated in the abstract setting 
in Section~\ref{sec:covariance-structure}.
The most illustrative results are the asymptotic formulas
presented in 
Corollaries~\ref{cor:asymptotic-marginal-spatial-covariance}~and~\ref{cor:separable-cov},
which we translate to the current setting in
Corollary~\ref{cor:example-1a-covariance}.
We see that
(Whittle--)Mat\'ern operators 
are recovered as marginal spatial 
or temporal covariance operators. 

\begin{corollary}\label{cor:example-1a-covariance}
	Consider the setting of Corollary~\textup{\ref{cor:example-1a-zeroIC}} 
	with $L=\widetilde{L}$, i.e., ${Q := L^{-\alpha}}$\!. 
	Let  
	$\alpha,\beta\in [0,\infty)$ and  
	$\gamma \in (\nicefrac{1}{2},\infty)$
	be such that
	$\beta\gamma > \frac{1}{2}\bigl( \frac{d}{2} - \alpha +\beta\bigr)$, 
	and let $Z_\gamma$ be the mild solution 
	in the sense of Definition~\textup{\ref{def:frac-mild-sol-zeroIC}}. 
	Then the asymptotic marginal spatial covariance
	of $Z_\gamma$ 
	satisfies
	\[ 
	\lim_{t\to\infty} Q_{Z_\gamma}(t,t) 
	= 
	\Gamma(\gamma-\nicefrac{1}{2})
	\bigl[ 2\sqrt{\pi} \Gamma(\gamma) \bigr]^{-1} 
	L^{\beta(1-2\gamma)-\alpha} 
	\quad 
	\text{in}\;\;
	\mathscr L(L^2(\cD)). 
	\] 
	For $\beta = 0$, 
	the covariance of $Z_\gamma$ is separable
	in the sense that
	there exists a function 
	$\varrho_{Z_\gamma}\colon [0,\infty)\times[0,\infty) \to \R$
	such that 
	\[ 
	Q_{Z_\gamma}(s,t) 
	=
	\varrho_{Z_\gamma}(s,t) \, L^{-\alpha}
	\quad 
	\forall s,t\in[0,\infty), 
	\] 
	and for all $h \in \R \setminus \{0\}$ we have
	\[ 
	\lim_{t\to\infty}Q_{Z_\gamma}(t,t+h)
	= 
	2^{\frac{1}{2}-\gamma}
	\bigl[ \sqrt{\pi}\Gamma(\gamma) \bigr]^{-1}
	|h|^{\gamma-\frac{1}{2}} 
	K_{\gamma-\frac{1}{2}}(|h|)\, L^{-\alpha}
	\quad 
	\text{in}\;\; 
	\mathscr L(L^2(\cD)). 
	\] 
\end{corollary}

\begin{proof}
	Existence and uniqueness of the mild
	solution $Z_\gamma$ follow from Corollary~\ref{cor:example-1a-zeroIC}
	with $L = \widetilde L$ and $n = \tau = \sigma = 0$.
	Recall from its proof that $A$ satisfies
	Assumptions~\ref{assumption:A}\ref{assumption:A:semigroup}--\ref{assumption:A:bddly-inv}.
	Note also that $A = L^\beta$ is self-adjoint
	and $Q = L^{-\alpha}\in \mathscr L(L^2(\cD))$ 
	commutes with $A$, so that 
	it also commutes with $S(t)$ 
	for all $t \in [0,\infty)$,
	cf.\ \cite[Theorem~1.3.2(a)]{Haase2006}.
	All assertions follow thus 
	from Corollaries~\ref{cor:asymptotic-marginal-spatial-covariance}
	and~\ref{cor:separable-cov}.
\end{proof}

\begin{remark} 
	The asymptotic results obtained in 
	Corollary~\ref{cor:example-1a-covariance} 
	are in accordance with the marginal  
	spatial and temporal covariance functions 
	derived 
	in \cite[Section~3, Proposition~1 and Corollary~1]{Bakka2023diffusionbased} 
	for the case of the differential operator 
	$L= \gamma_s^2 - \Delta$ 
	acting on functions defined on all of $\R^2$\!, 
	where $\gamma_s\in(0,\infty)$. 
	Note that, in order to exploit 
	Fourier techniques, 
	in \cite{Bakka2023diffusionbased} the ``time'' variable  
	$t$ is an element of the whole real axis, $t\in\R$,  
	instead of only its non-negative part. 
\end{remark} 

\begin{remark}
	Corollaries~\ref{cor:example-1a-zeroIC}~and~\ref{cor:example-1a-covariance}
	explain and justify the roles of 
	the parameters
	$\alpha$, $\beta$ and $\gamma$. 
	They control three important 
	properties of spatiotemporal 
	Whittle--Mat\'{e}rn fields. 
	Besides the temporal and spatial smoothness,
	measured respectively by the quantities
	$n+\tau$ and $\sigma$,
	we identify a third degree of freedom: The
	\emph{degree of separability}, 
	expressed by the ratio $\frac{\alpha}{\beta}\in[0,\infty]$.
	Indeed, if $\frac{\alpha}{\beta} = \infty$, 
	i.e.\ $\beta = 0$, we observe that the
	covariance of the field is separable and that its
	temporal and spatial behavior are exclusively 
	governed by the parameters 
	$\gamma$ and $\alpha$, respectively.
	In contrast, if $\frac{\alpha}{\beta} = 0$, 
	i.e.\ $\alpha = 0$, 
	the SPDE is driven by spatiotemporal Gaussian white noise 
	and the ``coloring'' of its solution is fully determined by 
	the fractional parabolic differential operator 
	$\bigl(\partial_t + L^\beta \bigr)^\gamma$\!. 
\end{remark}

\subsection{Surfaces}
\label{subsec:example:surfaces}

In this subsection, we 
provide a brief demonstration of 
how the
above results can be extended to spatiotemporal 
Whittle--Mat\'ern fields on more general 
spatial domains. 
More precisely, we consider a smooth, closed, 
connected, orientable and compact $2$-surface $\cM$
immersed in $\R^3$ and endowed with
the positive surface measure~$\nu_\cM$ on $\cB(\cM)$,
induced by the first fundamental form. 
An important example of such a surface is given 
by the $2$-sphere, $\cM=\mathbb{S}^2$\!. 

On
$H := L^2(\cM)$, we consider the 
following analog of the
symmetric, strongly elliptic second-order 
differential operator from
Subsection~\ref{subsec:example:bdd-euclidean}, formally given by
\[
Lu := -\nabla_\cM \cdot (a\nabla_\cM u)
+
\kappa^2 u,
\qquad 
u \in \mathsf D(L)\subseteq L^2(\cM),
\] 
where $\nabla_{\cM}\,\cdot\,$ and 
$\nabla_{\cM}$ denote 
the surface divergence and 
the surface gradient, respectively. 
We record the precise
conditions on the surface $\cM$ and on the coefficients $a,\kappa$ 
in Assumption~\ref{assumption:surface-smooth} below; 
with regard to smoothness, they 
are analogous to the setting of Assumption~\ref{assumption:smooth}
in the case of a bounded Euclidean domain. 

\begin{assumption}[Surface---smooth setting]
	\
	
	\label{assumption:surface-smooth}
	\begin{enumerate}[label=(\roman*), leftmargin=1cm] 
		\item 
		$a$ is a symmetric tensor field,
		i.e., $a(x)\from T_x \cM\to T_x \cM$ is linear
		and symmetric for all $x \in \cM$, where
		$T_x \cM$ denotes the tangent space of~$x$. 
		Moreover, $a$ is smooth and strongly
		elliptic in the following sense:  
		\[ 
		\exists \, \theta > 0: 
		\quad 
		\forall x \in \cM, \; \forall \xi \in T_x\cM : \quad 
		\xi^\top a(x) \xi 
		\geq  
		\theta \norm{\xi}{\R^3}^2.
		\]
		\item 
		The coefficient $\kappa \from \cM \to \R$
		is smooth and 
		bounded away from zero, i.e., 
		there exists $\kappa_0 \in (0,\infty)$ such that
		$|\kappa(x)|\ge \kappa_0$ for all $x \in \cM$.
	\end{enumerate}
\end{assumption}

The conditions in Assumption~\ref{assumption:surface-smooth} 
are sufficient to ensure that
$L \from \dot H_L^1 \to (\dot H_L^{1})^*$ is 
boundedly invertible,  
and has a compact inverse on $L^2(\cM)$. 
This allows us to find an orthonormal basis $(e_j)_{j\in\N}$
for $L^2(\cM)$ and a non-decreasing sequence of positive
real eigenvalues $(\lambda_j)_{j\in\N}$ of
$L$ accumulating only at infinity, 
as in Subsection~\ref{subsec:example:bdd-euclidean}.
Moreover, fractional powers $L^\beta$ are well-defined
for all $\beta\in\R$, the sequence of eigenvalues
still satisfies Weyl's law~\eqref{eq:weyl} (with $d = 2$),
and a spectral mapping theorem holds,
cf.\ \cite[Theorems~XII.1.3~and~XII.2.1]{Taylor1981}.
These facts are sufficient to repeat the proofs of
Corollaries~\ref{cor:example-1a-zeroIC} 
and~\ref{cor:example-1a-covariance} 
yielding the analogous results, with $d = 2$ 
and other obvious modifications to the conditions.
In particular, the analog of Corollary~\ref{cor:example-1a-zeroIC} 
on the surface $\cM$ implies regularity of the solution process in  
the space 
$C^{n,\tau}([0,T]; L^p(\Omega; H^{\beta\sigma}(\cM)))$. 

An important difference from the 
(smooth) Euclidean setting of
Assumption~\ref{assumption:smooth}
is that 
under Assumption~\ref{assumption:surface-smooth},
the Sobolev space $H^s(\cM)$ and  
$\dot H_L^s$ are isomorphic for every $s \in [0,\infty)$,
see \cite[Example~XII.2.1]{Taylor1981}.
In other words, the absence of a boundary $\partial\cM$
implies that one does not need to 
exclude the exception set $\mathfrak{E}$ 
from the admissible exponents $s$
in the analog of Lemma~\ref{lem:dot-sobolev}\ref{lem:dot-sobolev-c}. 


\appendix

\section{Auxiliary results}\label{app:auxiliary} 

Throughout this section, $H$ denotes a separable Hilbert space  
which, if not specified otherwise, is considered over the 
real scalar field~$\R$. 

\subsection{Bochner counterparts}\label{app:subsec:bochner-counterpart} 

The first auxiliary result records  
relations between 
a (possibly unbounded) linear operator 
$A\from\mathsf D(A)\subseteq H \to H$ 
and its Bochner space counterpart $\cA$ 
which is defined on a subspace 
of $L^2(0,T;H)$, where $T\in(0,\infty)$.  

\begin{lemma}\label{lem:Bochner-space-operator}
	Let $T\in(0,\infty)$ and 
	$A\from \mathsf D(A)\subseteq H\to H$ 
	be a linear operator on a real or complex 
	Hilbert space $H$.  
	Consider the associated operator 
	$\mathcal A$ on $L^2(0,T;H)$ 
	as defined in \eqref{eq:def-cA}.
	Then, the following hold: 
	\begin{enumerate}[leftmargin=1cm, label={\normalfont(\alph*)}]
		\item\label{lem:Bochner-space-operator-a} 
		$\mathcal A$ is bounded if and only if 
		$A$ is bounded, 
		and in that case we have 
		\[
		\norm{\mathcal A}{\mathscr L(L^2(0,T;H))} 
		= 
		\norm{A}{\mathscr L(H)} ;
		\]  
		\item\label{lem:Bochner-space-operator-b}
		$\mathcal A$ is closed if and only if $A$ is.
	\end{enumerate}
\end{lemma}

\begin{proof}
	If $A$ is bounded, 
	then the inequality 
	$\norm{\mathcal A}{\mathscr L(L^2(0,T;H))} 
	\leq 
	\norm{A}{\mathscr L(H)}$ 
	is easily verified.
	Now suppose that $\mathcal A$ is bounded. 
	Then for all $x \in H$ we have
	\begin{align*}
	\norm{Ax}{H} 
	&= 
	\norm{T^{-\nicefrac{1}{2}}\mathbf 1_{(0,T)} \otimes Ax}{L^2(0,T;H)} 
	= 
	\norm{\mathcal A(T^{-\nicefrac{1}{2}}\mathbf 1_{(0,T)} \otimes x)}{L^2(0,T;H)} 
	\\
	&\leq 
	\norm{\mathcal A}{\mathscr L(L^2(0,T;H))} 
	\norm{T^{-\nicefrac{1}{2}}\mathbf 1_{(0,T)} \otimes x}{L^2(0,T;H)} 
	= \norm{\mathcal A}{\mathscr L(L^2(0,T;H))} \norm{x}{H}.
	\end{align*} 
	Here, given $f \from (0,T)\to \R$ and $x \in H$,
	the function $f \otimes x \from (0,T)\to H$ is 
	defined by $[f\otimes x](t) := f(t)x$ for all $t \in (0,T)$.
	We thus find that $A$ is bounded 
	with operator norm  
	$\norm{A}{\mathscr L(H)}
	\leq \norm{\mathcal A}{\mathscr L(L^2(0,T;H))}$,
	which finishes the proof of~\ref{lem:Bochner-space-operator-a}. 
	
	To prove 
	part~\ref{lem:Bochner-space-operator-b}, 
	first let $A$ be closed and let the sequence 
	$(v_n)_{n\in\N}$ in $ \mathsf D(\mathcal A)$ 
	be such that $v_n \to v$ and $\mathcal Av_n \to y$ 
	in $L^2(0,T;H)$. We need to prove that 
	$v \in \mathsf D(\mathcal A)$ and 
	$y = \mathcal Av$. Let $(v_{n_k})_{k\in\N}$ 
	be a subsequence such that 
	$v_{n_k} \to v$ 
	and $Av_{n_k} \to y$ in~$H$, 
	a.e.\ in $(0,T)$, 
	so that by the closedness of $A$ it follows that 
	$v(\vartheta) \in \mathsf D(A)$ and $y(\vartheta) = Av(\vartheta)$ 
	for a.a.\ $\vartheta\in(0,T)$. 
	From the latter we obtain that $y = \mathcal Av$, 
	which is meaningful since $v,y\in L^2(0,T;H)$ 
	yields that $v \in \mathsf D(\mathcal A)$.
	
	Now let $\mathcal A$ be closed and 
	let $(x_n)_{n\in \N}$ in $\mathsf D(A)$ 
	be such that $x_n \to x$ and $Ax_n \to y$ in $H$. 
	This implies the following convergences in 
	$L^2(0,T;H)$: 
	\begin{align*}
	\mathbf 1_{(0,T)} \otimes x_n 
	\to \mathbf 1_{(0,T)}\otimes x, 
	\\ 
	\mathcal A(\mathbf 1_{(0,T)} \otimes x_n) 
	= 
	\mathbf 1_{(0,T)} \otimes Ax_n  
	\to \mathbf 1_{(0,T)} \otimes y. 
	\end{align*}
	Since $\mathcal A$ is closed, 
	we deduce that 
	$\mathbf 1_{(0,T)} \otimes x \in \mathsf D(\mathcal A)$ 
	and $\mathbf 1_{(0,T)} \otimes y = \mathcal A(\mathbf 1_{(0,T)} \otimes x)$, 
	from which we may conclude $x \in \mathsf D(A)$ and $y = Ax$. 
	Hence $A$ is closed.
\end{proof}

The following lemma is generally useful for determining 
the domain of a generator of a given $C_0$-semigroup, 
and it will subsequently be used 
to show that the Bochner space counterpart 
of a $C_0$-semigroup  
is again a $C_0$-semigroup, 
see Proposition~\ref{prop:curly-A-semigroup}. 

\begin{lemma}\label{lem:generators}
	Let $(S(t))_{t\ge0}$ be a $C_0$-semigroup on $H$ 
	with infinitesimal generator 
	$\widetilde A \from \mathsf D(\widetilde A)\subseteq H \to H$. 
	If $A\from \mathsf D(A)\subseteq H\to H$ is 
	a linear operator satisfying $A \subseteq \widetilde{A}$ 
	and $\mathsf D(A)$ is dense in $\mathsf D(\widetilde{A})$ 
	with respect to the graph norm 
	$\| \cdot \|_{\mathsf D(\widetilde{A})}$, 
	then $\widetilde A = \clos A$. 
\end{lemma}

\begin{proof}
	Let $(x, \widetilde{A}x) \in \mathsf G(\widetilde{A})$  
	and choose a sequence $(x_n)_{n\in\N}$ in $\mathsf D(A)$ 
	such that ${x_n \to x}$ in 
	$\mathsf D(\widetilde{A})$.
	Using $A \subseteq \widetilde{A}$, we have
	$(x_n, Ax_n) = (x_n, \widetilde{A}x_n) \to (x, \widetilde A x)$ 
	with respect to the product norm on $H \times H$,
	which shows that $(x, \widetilde{A}x) \in \clos{\mathsf G(A)}$.
	Conversely, for any $(x, y) \in \clos{\mathsf G(A)}$
	there exists a sequence $(x_n)_{n\in\N}\subseteq \mathsf D(A)$
	such that
	$(x_n, \widetilde{A}x_n) = (x_n, Ax_n) \to (x,y)$ in $H \times H$.
	Since $\widetilde A$ is closed
	as the generator of a $C_0$-semigroup, 
	see \cite[Theorem~II.1.4]{Engel1999}, 
	we find that $(x,y) \in \mathsf G(\widetilde{A})$.
	This proves $\clos{\mathsf G(A)} = \mathsf G(\widetilde{A})$.
\end{proof}

\begin{proposition}\label{prop:curly-A-semigroup}
	Let $T\in(0,\infty)$ and 
	let Assumption~\textup{\ref{assumption:A}\ref{assumption:A:semigroup}} 
	be satisfied. 
	The family $(\mathcal S(t))_{t\ge0}$ of operators on $L^2(0,T;H)$ 
	given by \eqref{eq:def-cS} is a $C_0$-semigroup 
	with infinitesimal generator $-\cA$, as defined by \eqref{eq:def-cA}.
\end{proposition}

\begin{proof}
	First note that the operators $(\cS(t))_{t\ge0}$ 
	are well-defined in the sense that they map 
	elements in $L^2(0,T;H)$ to $L^2(0,T;H)$. 
	In fact, 
	Lemma~\ref{lem:Bochner-space-operator}\ref{lem:Bochner-space-operator-a} 
	shows that 
	$\norm{\mathcal S(t)}{\mathscr L(L^2(0,T;H))} 
	= \norm{S(t)}{\mathscr L(H)}$ for all $t \geq 0$. 
	
	We now check that $(\mathcal S(t))_{t\ge0}$ 
	is a $C_0$-semigroup. Clearly,
	$\mathcal S(0)=I$ and the semigroup property holds. 
	Let $M \ge 1$ and $w \in \R$ be as 
	in \eqref{eq:exp-estimate-semigroup},
	so that
	\[ 
	\forall h \in [0,1]: 
	\quad 
	\norm{ S(h) }{\mathscr L(H)} 
	\leq 
	M e^{-wh} 
	\leq 
	Me^{(-w)\vee 0} 
	=: 
	\widetilde{M}. 
	\]
	To show strong continuity, let $x \in H$, 
	$h\in(0,1)$ and note that
	\[ 
	\norm{S(h)x - x}{H}^2 
	\leq 
	2 \norm{S(h)x}{H}^2 + 2\norm{x}{H}^2
	\leq 
	2 \bigl( \widetilde M^2 + 1 \bigr) \norm{x}{H}^2. 
	\] 
	By dominated convergence, 
	$\lim_{h \downarrow 0} \norm{\mathcal S(h)v-v}{L^2(0,T;H)}=0$ 
	for $v \in L^2(0,T;H)$. 
	
	Next we investigate the infinitesimal generator 
	of $(\cS(t))_{t\ge0}$, which we denote by 
	$-\widetilde \cA$ for the time being. 
	We wish to show that $\widetilde{\cA} = \cA$. 
	Let $x \in \mathsf D(A)$ and consider
	\[
	\biggl\| \frac{1}{h}(S(h)x-x)+ Ax \biggr\|_{H}^2 
	\leq 
	2 
	\biggl\| \frac{1}{h}( S(h)x-x) \biggr\|_{H}^2 
	+ 
	2 \norm{ Ax }{H}^2. 
	\] 
	To bound the first term, 
	we use \cite[Chapter~1, Theorem~2.4(d)]{Pazy1983} 
	and note that, for every $h\in(0,1)$, 
	we obtain  
	\[ 
	\biggl\| \frac{1}{h}(S(h)x-x) \biggr\|_{H}^2 
	= 
	\biggl\| 
	\frac{1}{h}\int_0^h S(s) A x \rd s 
	\biggr\|_H^2
	\leq 
	\frac{1}{h^2} 
	\biggl| \int_0^h \norm{S(s) A x }{H}\rd s \biggr|^2 
	\leq 
	\widetilde{M}^2 \norm{Ax}{H}^2.
	\]
	The two previous displays show that, 
	for $v \in L^2(0,T;\mathsf D(A))$ 
	and all $h\in(0,1)$, 
	\[
	\int_0^T 
	\biggl\| \frac{1}{h}(S(h)v(\vartheta)-v(\vartheta)) 
	+ Av(\vartheta) \biggr\|_{H}^2 \rd \vartheta 
	\leq 
	2 \bigl( \widetilde M^2 + 1 \bigr) \norm{\cA v}{L^2(0,T;H)}^2 
	< \infty.
	\] 
	This justifies the use of the dominated convergence theorem 
	to conclude that
	\[ 
	-\widetilde{\cA}v 
	= \lim_{h\downarrow 0} \frac{1}{h}(\mathcal S(h)v-v) 
	= -\mathcal Av 
	\qquad 
	\text{in}\;\; L^2(0,T;H), 
	\] 
	i.e., $-\cA \subseteq -\widetilde{\cA}$ 
	as $v \in \mathsf D(\cA) = L^2(0,T; \mathsf D(A))$ was arbitrary.
	Since $\mathsf D(\cA)$ is dense 
	in $L^2(0,T;H)$ (by density of $\mathsf D(A)$ in $H$), 
	and $\mathcal S(t)$ maps $\mathsf D(\cA)$ 
	to itself for each $t \geq 0$, 
	Proposition~II.1.7 of \cite{Engel1999} 
	implies that $\mathsf D(\cA)$ 
	is dense in the 
	domain $\mathsf D(\widetilde\cA)$ of the generator 
	of $(\mathcal S(t))_{t\ge0}$ with respect to 
	the graph norm $\| \cdot \|_{\mathsf D(\widetilde{\cA})}$.
	Applying Lemma~\ref{lem:generators} 
	and noting that $\mathcal A$ is closed 
	by Lemma~\ref{lem:Bochner-space-operator}\ref{lem:Bochner-space-operator-b} 
	completes the proof.
\end{proof}

\subsection{Translation operators}\label{app:subsec:translation} 

\begin{lemma}\label{lem:translations} 
	Let $U$ be a real and separable Hilbert space and
	let $J := (0,T)$ for some $T \in (0,\infty]$.
	For every $u \in L^2(J; U)$ we have that 
	\[
	\lim_{h\to 0} 
	\norm{ u(\,\cdot\, + h) - u }{ L^2(J_h ; U) } 
	= 0.
	\] 
	Here, we define for each $h \in \R$ the interval
	$J_h := ((-h) \vee 0, T\wedge (T-h)) \subseteq J$ and
	$u(\,\cdot\, + h) \from J_h \to U$ 
	denotes the function $u$ shifted to the left 
	by an increment $h$. 
\end{lemma}

\begin{proof}
	Let $v \in C^\infty_c(J; U)$ and 
	fix an arbitrary 
	$\varepsilon\in(0,\infty)$. 
	Choose a compact 
	interval~$[a,b]\subset [0,\infty)$ 
	such that 
	$\operatorname{supp}\bigl( 
	v(\,\cdot\, + h) - v|_{J_h} 
	\bigr) \subseteq [a,b]$ 
	for all  
	$h\in[-1,1]$. 
	By the uniform continuity of~$v$, 
	there exists a $\delta\in(0,1)$ such that, 
	for all $h\in(-\delta,\delta)$ 
	and every $t\in J_h$, 
	the estimate 
	$\norm{v(t+h)-v(t)}{U} < \sqrt{\varepsilon/(b-a)}$ 
	holds. Thus,
	\[ 
	\norm{ v(t+h)-v(t) }{L^2(J_h;U)}^2 
	< 
	\varepsilon 
	\quad 
	\forall h\in (-\delta,\delta). 
	\] 
	This shows the desired convergence 
	for functions in the space $C^\infty_c(J; U)$,
	which is dense in $L^2(J;U)$;
	indeed, since the set of $U$-valued measurable simple functions is 
	dense in $L^2(J; U)$ \cite[Lemma~1.2.19(1)]{AnalysisInBanachSpacesI}, 
	it suffices to note
	that the scalar-valued function space $C_c^\infty(J)$ is dense in $L^2(J)$
	\cite[Corollary~2.30]{AdamsFournier2003}. 
	Combined with the fact that 
	the translation operator is contractive  
	from $L^2(J; U)$ to $L^2(J_h; U)$ 
	(and thus bounded, uniformly in $h$), 
	the result extends 
	to $L^2(J; U)$.
\end{proof}

\begin{proposition}\label{prop:translation-semigroup}
	Let $T\in(0,\infty)$. 
	The family 
	$(\mathcal T(t))_{t\ge0} \subseteq \mathscr L(L^2(0,T;H))$ 
	defined in~\eqref{eq:def-cT}  
	is a $C_0$-semigroup whose infinitesimal generator 
	is given by $-\partial_t$, 
	where $\partial_t$ is the Bochner--Sobolev 
	vector-valued weak derivative 
	on $\mathsf D(\partial_t) = H^1_{0,\{0\}}(0,T; H)$.
\end{proposition}

\begin{proof}
	For each $t \geq 0$, it is clear that $\mathcal T(t)$ 
	is a well-defined contractive linear map on $L^2(0,T;H)$. 
	Furthermore, it follows readily from 
	the definition~\eqref{eq:def-cT} that $\mathcal T(0) = I$ 
	and that the semigroup property is satisfied, 
	since for all $s,t\geq 0$, $v \in L^2(0,T;H)$ 
	and a.a.\ $\vartheta\in[0,T]$ we have that 
	\[ 
	[\mathcal T(t)\mathcal T(s)v](\vartheta) 
	= \widetilde{[\mathcal T(s)v]}(\vartheta-t) 
	= \widetilde v(\vartheta-t-s)
	= 
	[\mathcal T(t+s)v](\vartheta).
	\] 
	The strong continuity follows 
	from 
	Lemma~\ref{lem:translations} 
	for $h \uparrow 0$. 
	
	Next, we turn to  
	the generator of $(\mathcal T(t))_{t\geq 0}$. 
	To this end, let an arbitrary ${v \in C_{c}^\infty((0,T];H)}$ 
	be given and note that its extension 
	by zero to $(-\infty,T]$, again denoted by $\widetilde v$, 
	is continuously differentiable 
	with classical (and hence weak) derivative 
	${\partial_\vartheta \widetilde v = \widetilde{\partial_\vartheta v}}$ 
	by the compact support of $v$ in $(0,T]$. 
	Fix an arbitrary $\vartheta\in[0,T]$. 
	The function $t \mapsto \widetilde v(\vartheta-t)$ 
	is continuously differentiable on $[0,\infty)$ 
	with derivative $t \mapsto -\widetilde{\partial_\vartheta v}(\vartheta-t)$ 
	by the chain rule. 
	Thus, the fundamental theorem of calculus~gives 
	\[
	\mathcal T(t)v(\vartheta) - v(\vartheta) 
	= \widetilde v(\vartheta-t) - \widetilde v(\vartheta)
	= -\int_0^t \widetilde{\partial_\vartheta v}(\vartheta-s) \rd s
	= -\int_0^t [\mathcal T(s)\partial_\vartheta v](\vartheta) \rd s
	\]
	for every $t \geq 0$.
	It follows that 
	\[ 
	\mathcal T(t)v-v = -\int_0^t \mathcal T(s)\partial_\vartheta v \rd s.
	\] 
	Furthermore, 
	we know from \cite[Chapter~1, Theorem~2.4(b)]{Pazy1983} 
	that if $R$ denotes the generator of $(\mathcal T(t))_{t\ge0}$, 
	then we have 
	\[
	\mathcal T(t)v-v=R \int_0^t \mathcal T(s)v \rd s,
	\]
	hence, combining the previous two displays yields 
	\begin{equation}\label{eq:proof-generator-T}
	R \int_0^t \mathcal T(s)v \rd s 
	= 
	-\int_0^t \mathcal T(s)\partial_\vartheta v\rd s. 
	\end{equation}
	Set $v_t := \frac{1}{t}\int_0^t \cT(s)v \rd s$ for $t \in (0,\infty)$. 
	It follows that $v_t \to \cT(0)v = v$ in $L^2(0,T;H)$ 
	as $t \downarrow 0$, 
	see e.g.\ \cite[Chapter~1, Theorem~2.4(a)]{Pazy1983}. 
	Dividing both sides of \eqref{eq:proof-generator-T} 
	by $t \in (0,\infty)$ and passing to the limit $t \downarrow 0$, 
	one obtains
	\[
	Rv_t 
	= 
	R \, \frac{1}{t}\int_0^t \mathcal T(s)v \rd s 
	= 
	- \frac{1}{t}\int_0^t \mathcal T(s)\partial_\vartheta v\rd s \to 
	- \mathcal T(0)\partial_\vartheta v = -\partial_\vartheta v.
	\] 
	Since $R$ is assumed to be the generator of a $C_0$-semigroup, 
	it is in particular closed by \cite[Proposition~II.1.4]{Engel1999}. 
	Combined with the convergence 
	$v_t \to v$ and $Rv_t \to -\partial_\vartheta v$ 
	as $t \downarrow 0$, 
	this yields $v \in \mathsf D(R)$ and ${Rv = -\partial_\vartheta v}$, 
	hence $-\partial_\vartheta|_{C_c^\infty((0,T];H)} \subseteq R$.
	
	As $C_{c}^\infty((0,T]; H)$ is dense in $L^2(0,T; H)$  
	and $\mathcal T(t)C_{c}^\infty((0,T]; H) \subseteq C_{c}^\infty((0,T]; H)$  
	for all $t \geq 0$, we have that $C_{c}^\infty((0,T]; H)$ is dense 
	in $\mathsf D(R)$ with respect to the graph norm of $R$ 
	by \cite[Proposition~II.1.7]{Engel1999}. 
	It is evident from the respective definitions that 
	$\norm{\cdot}{\mathsf D(R)} \eqsim \norm{\cdot}{H^1(0,T;H)}$.
	These observations together imply 
	\[ 
	\mathsf D(R) 
	= 
	\clos{C_c^\infty((0,T];H)}^{\mathsf D(R)} 
	= 
	\clos{C_c^\infty((0,T];H)}^{H^1(0,T;H)} 
	= 
	H^1_{0,\{0\}}(0,T; H). 
	\qedhere
	\] 
\end{proof}

\subsection{The proof of 
	\texorpdfstring{Lemma~\ref{lem:adjoint-neg-frac}}{Lemma 3.6}}
\label{app:subsec:proof-lemma-adjoint} 

\begin{proof}[Proof of Lemma~\ref{lem:adjoint-neg-frac}]
	Analogously to~\cite[Proposition~5.9]{DaPrato2014} 
	it can be shown that the operator defined
	by the right-hand side of~\eqref{eq:adjoint-neg-parabolic-formula} 
	maps functions in $L^2(0,T;H)$ to $C_{0,\{T\}}([0,T];H)$. 
	Now we prove the identity in \eqref{eq:adjoint-neg-parabolic-formula}. 
	Let ${f,g \in L^2(0,T;H)}$ be arbitrary. 
	By \eqref{eq:neg-parabolic-formula} 
	and by continuity of the inner product 
	$(\,\cdot\,,\,\cdot\,)_H$, 
	we find that
	\begin{align} 
	(\cB^{-\gamma} &f, g)_{L^2(0,T;H)} 
	= 
	\int_0^T \bigl( [\cB^{-\gamma}f](t), g(t) \bigr)_H \rd t 
	\notag 
	\\
	&= 
	\int_0^T 
	\biggl( 
	\frac{1}{\Gamma(\gamma)} 
	\int_0^t (t-s)^{\gamma-1} S(t-s) f(s) \rd s , g(t) 
	\biggr)_H \rd t 
	\notag 
	\\
	&= 
	\frac{1}{\Gamma(\gamma)} 
	\int_0^T \int_0^T 
	\bigl( \mathbf 1_{(0,t)}(s) (t-s)^{\gamma-1} S(t-s) f(s), g(t) \bigr)_H 
	\rd s \rd t . 
	\label{eq:proof:adjoint-neg-frac} 
	\end{align}
	Next, we would like to use Fubini's theorem 
	to exchange the order of integration. 
	By \eqref{eq:exp-estimate-semigroup} 
	the semigroup $(S(t))_{t\geq 0}$ is uniformly bounded 
	on the compact interval~$[0,T]$, 
	\[
	\widetilde{M}_T 
	:= 
	\sup\nolimits_{t\in[0,T]} \norm{ S(t) }{\mathscr L(H)} 
	\leq  
	M e^{(-wT) \vee 0}  
	< \infty.  
	\]
	We then use the 
	Cauchy--Schwarz inequality on $H$ and on $L^2(0,T)$ 
	as well as the fact that 
	$\gamma > \tfrac{1}{2}$ 
	to check that 
	\begin{align*}
	&\int_0^T \int_0^T 
	\bigl| 
	\bigl( \mathbf 1_{(0,t)}(s) (t-s)^{\gamma-1} S(t-s) f(s), g(t) \bigr)_H 
	\bigr| \rd s \rd t 
	\\
	&\quad\leq 
	\widetilde{M}_T 
	\int_0^T \int_0^t 
	(t-s)^{\gamma-1} \norm{ f(s) }{H}  \rd s 
	\;
	\norm{ g(t) }{H} \rd t 
	\\
	&\quad\leq 
	\widetilde{M}_T \, 
	\norm{f}{L^2(0,T;H)} 
	\int_0^T 
	\biggl(\int_0^t (t-s)^{2\gamma-2} \rd s \biggr)^{\nicefrac{1}{2}}
	\norm{ g(t) }{H} \rd t 
	\\
	&\quad = 
	\tfrac{ \widetilde{M}_T }{ \sqrt{2\gamma - 1} } 
	\, 
	\norm{f}{L^2(0,T;H)} 
	\int_0^T t^{\gamma-\frac{1}{2}} \norm{g(t)}{H} \rd t 
	\leq 
	\tfrac{ \widetilde{M}_T T^\gamma}{\sqrt{2\gamma(2\gamma - 1)}}   
	\,
	\norm{f}{L^2(0,T;H)} \norm{g}{L^2(0,T;H)} 
	\end{align*}
	is finite.  
	This justifies changing the order of integration 
	in \eqref{eq:proof:adjoint-neg-frac}, which gives 
	\begin{align*}
	(\cB^{-\gamma}f, g)_{L^2(0,T;H)} 
	&= 
	\frac{1}{\Gamma(\gamma)} 
	\int_0^T \int_0^T 
	\bigl( \mathbf 1_{(s,T)}(t) (t-s)^{\gamma-1} S(t-s) f(s), g(t) \bigr)_H 
	\rd t \rd s 
	\\
	&= 
	\frac{1}{\Gamma(\gamma)} 
	\int_0^T \int_0^T 
	\bigl( f(s), \mathbf 1_{(s,T)}(t) (t-s)^{\gamma-1} \dual{[S(t-s)]} g(t) \bigr)_H 
	\rd t \rd s 
	\\
	&= 
	\int_0^T 
	\biggl( f(s), \frac{1}{\Gamma(\gamma)} 
	\int_s^T \mathbf (t-s)^{\gamma-1} \dual{[S(t-s)]} g(t) \rd t 
	\biggr)_H \rd s, 
	\end{align*}
	where we interchanged integrals and inner products 
	as before in the last step.  
\end{proof} 

\subsection{H\"{o}lder continuity and weak derivatives}
\label{app:subsec:orbital} 

Recall from Section~\ref{sec:prelims} 
that $(W(t))_{t\geq 0}$ denotes 
an $H$-valued cylindrical Wiener process. 

\begin{lemma}\label{lemma:I-1}
	Let 
	Assumptions~\textup{\ref{assumption:A}\ref{assumption:A:semigroup},\ref{assumption:A:bdd-analytic}} 
	be satisfied, 
	let $a \in \bigl(-\frac{1}{2},\infty\bigr)$, ${b, \sigma \in [0,\infty)}$ 
	and $\tau\in\bigl( 0,a + \frac{1}{2} \bigr] \cap (0,1)$.  
	If $\sigma \neq 0$, then suppose moreover that
	Assumption~\textup{\ref{assumption:A}\ref{assumption:A:bddly-inv}} holds. 
	Let 
	$\Phi_{a,b}\from(0,\infty)\to\mathscr L(H; \Hdot{\sigma})$ 
	be defined by \eqref{eq:def:Phi-ab} 
	and let $J:=(0,T)$ 
	for some ${T \in (0,\infty]}$. 
	Then, 
	for all $p \in [1,\infty)$, 
	$t \in [0,T)$ and $h \in J$ 
	with $h\leq T-t$, 
	\begin{align*}  
	\biggl\| 
	\int_0^t \bigl[ \Phi_{a,b}(t+h-s) - \Phi_{a,b}(t-s) \bigr] \rd W(s) 
	&\biggr\|_{L^p(\Omega;\dot H_A^\sigma)} 
	\\
	&\lesssim_{(p,a,\tau)} 
	h^\tau 
	\bigl\| A^{-a-\frac{1}{2} +b +\tau} 
	Q^{\frac{1}{2}} \bigr\|_{\mathscr L_2(H;\Hdot \sigma)}. 
	\end{align*}
\end{lemma}

\begin{proof}
	We first use the Burkholder--Davis--Gundy inequality 
	(combined with nestedness of the 
	$L^p$ spaces if $p<2$) 
	to bound the quantity of interest $I_\star$, 
	\begin{align}
	I_\star 
	:= 
	\biggl\| 
	\int_0^t \bigl[ \Phi_{a,b}(t+h-s) &- \Phi_{a,b}(t-s) \bigr] 
	\rd W(s) 
	\biggr\|_{L^p(\Omega;\dot H_A^\sigma)} 
	\notag 
	\\ 
	&\lesssim_p 
	\biggl[ 
	\int_0^t 
	\norm{ \Phi_{a,b}(t+h-s) - \Phi_{a,b}(t-s) }{\mathscr L_2(H;\dot H_A^\sigma)}^2 
	\rd s \biggr]^{\nicefrac{1}{2}} 
	\notag 
	\\ 
	&= 
	\biggl[ 
	\int_0^t 
	\norm{ \Phi_{a,b}(u+h) - \Phi_{a,b}(u) }{\mathscr L_2(H;\dot H_A^\sigma)}^2 
	\rd u \biggr]^{\nicefrac{1}{2}}, 
	\label{eq:proof:emma:I-2:1} 
	\end{align}
	where we also applied  
	the change of variables $u := t-s$. 
	For every $u \in (0,t)$, 
	Lemma~\ref{lem:diffbty-of-Phi} implies that
	$\Phi_{a,b}( u + \,\cdot\, )$ is differentiable as a function from 
	$(0,h)$ to $\mathscr L(H;\Hdot\sigma)$ with derivative 
	$\Phi_{a,b}'( u + \,\cdot\, )$ 
	and, moreover,
	$r \mapsto \norm{\Phi'_{a,b}(u+r)}{\mathscr L(H;\Hdot\sigma)}$ 
	is bounded on $[0,h]$. 
	We conclude that 
	$\Phi_{a,b}( u + \,\cdot\, ) \in H^1(0,h; \mathscr L(H;\Hdot\sigma))$,
	so that by \cite[\S5.9.2, Theorem~2]{Evans2010} 
	the identity 
	\[ 
	\Phi_{a,b}(u+h) - \Phi_{a,b}(u) 
	= 
	\int_0^h \Phi'_{a,b}(u+r) \rd r 
	\] 
	holds 
	as operators in $\mathscr L(H; \Hdot\sigma)$.
	We now estimate \eqref{eq:proof:emma:I-2:1} 
	by exploiting this relation, 
	moving the norm inside the integral,  
	applying formula \eqref{eq:deriv-operator-suggestive} 
	for the derivative of~$\Phi_{a,b}$ 
	and using the triangle and Minkowski inequalities, 
	which gives 
	\begin{align}
	I_\star 
	&\lesssim_p 
	\biggl[ 
	\int_0^t 
	\biggl|\int_0^h 
	\norm{\Phi_{a,b}'(u+r)}{\mathscr L_2(H;\dot H_A^\sigma)} 
	\rd r\biggr|^2 \rd u 
	\biggr]^{\nicefrac{1}{2}}
	\notag 
	\leq 
	\biggl[ \int_0^t \bigl| |a| F(u) + G(u) \bigr|^2 \rd u 
	\biggr]^{\nicefrac{1}{2}} 
	\\
	&\leq 
	|a| 
	\biggl[ \int_0^t |F(u)|^2 \rd u \biggr]^{\nicefrac{1}{2}} 
	+ 
	\biggl[ 
	\int_0^t |G(u)|^2 \rd u 
	\biggr]^{\nicefrac{1}{2}} \!, 
	\label{eq:proof:emma:I-2:2} 
	\end{align}
	where 
	\begin{align*} 
	F(u) 
	&:= 
	\int_0^h 
	\norm{\Phi_{a-1,b}(u+r)}{\mathscr L_2(H;\dot H_A^\sigma)} \rd r , 
	\\ 
	G(u) 
	&:= 
	\int_0^h 
	\norm{\Phi_{a,b+1}(u+r)}{\mathscr L_2(H;\dot H_A^\sigma)} \rd r. 
	\end{align*} 
	Using Minkowski's integral inequality 
	(see e.g.~\cite[\S{}A.1]{Stein1970}), 
	we obtain 
	\begin{align*}
	\biggl[\int_0^t &|F(u)|^2 \rd u \biggr]^{\nicefrac{1}{2}} 
	= 
	\biggl[\int_0^t \,
	\biggl| \int_0^h 
	\| \Phi_{a-1,b}(u+r) \|_{\mathscr L_2(H;\Hdot\sigma)} \rd r \biggr|^2 
	\rd u\biggr]^{\nicefrac{1}{2}} 
	\\
	&\leq 
	\int_0^h \biggl[ 
	\int_0^t 
	\| \Phi_{a-1,b}(u+r) \|_{\mathscr L_2(H;\Hdot\sigma)}^2 \rd u
	\biggr]^{\nicefrac{1}{2}} \rd r 
	\\
	&= 
	\int_0^h \biggl[ 
	\int_0^t (u+r)^{2(a-1)} 
	\bigl\| A^{a+\frac{1}{2}-\tau} S(u+r) A^{\frac{\sigma}{2}-a-\frac{1}{2}+b+\tau} Q^{\frac{1}{2}} \bigr\|_{ 
		\mathscr L_2(H)}^2 
	\rd u\biggr]^{\nicefrac{1}{2}} \rd r 
	\end{align*} 
	Since the semigroup $(S(t))_{t\geq 0}$ is assumed to be analytic, 
	by~\eqref{eq:analytic-est-1} the estimate 
	\begin{equation}\label{eq:proof:emma:I-2:3} 
	\begin{split} 
	\bigl\| A^{a+\frac{1}{2}-\tau} S(u+r) &A^{\frac{\sigma}{2} -a-\frac{1}{2}+b+\tau} 
	Q^{\frac{1}{2}} \bigr\|_{ \mathscr L_2(H) } 
	\\
	&\lesssim_{(a,\tau)}
	(u+r)^{-a-\frac{1}{2}+\tau} 
	\bigl\| A^{\frac{\sigma}{2} -a-\frac{1}{2}+b+\tau} Q^{\frac{1}{2}} \bigr\|_{ 
		\mathscr L_2(H)} 
	\end{split} 
	\end{equation} 
	follows, where we also used the assumption that $a+\frac{1}{2}-\tau\geq 0$. 
	We conclude that 
	\begin{align*} 
	\biggl[\int_0^t &|F(u)|^2 \rd u \biggr]^{\nicefrac{1}{2}} 
	\lesssim_{(a,\tau)} 
	\bigl\| A^{-a-\frac{1}{2}+b+\tau} Q^{\frac{1}{2}} \bigr\|_{  
		\mathscr L_2(H;\Hdot\sigma)} 
	\int_0^h \biggl[ 
	\int_0^t 
	(u+r)^{2\tau-3} 
	\rd u\biggr]^{\nicefrac{1}{2}} \rd r 
	\\
	&\leq 
	\bigl\| A^{-a-\frac{1}{2}+b+\tau} Q^{\frac{1}{2}} \bigr\|_{  
		\mathscr L_2(H;\Hdot\sigma)} 
	\int_0^h \biggl[ 
	\int_r^\infty 
	\bar{u}^{2\tau-3} 
	\rd \bar{u}\biggr]^{\nicefrac{1}{2}} \rd r 
	\\
	&=  
	\frac{1}{\sqrt{2 - 2\tau}} \, 
	\bigl\| A^{-a-\frac{1}{2}+b+\tau} Q^{\frac{1}{2}} \bigr\|_{ 
		\mathscr L_2(H;\Hdot\sigma)} 
	\int_0^h r^{\tau-1}\rd r 
	\\
	&= 
	\frac{1}{\tau \sqrt{2 - 2\tau}} \, 
	h^\tau \, 
	\bigl\| A^{-a-\frac{1}{2}+b+\tau} Q^{\frac{1}{2}} \bigr\|_{ 
		\mathscr L_2(H;\Hdot\sigma)} . 
	\end{align*}

	Similarly, we can bound the integral $\int_0^t |G(u)|^2 \rd u$ 
	in \eqref{eq:proof:emma:I-2:2}. 
	Again by Minkowski's integral inequality 
	and analogously to \eqref{eq:proof:emma:I-2:3}, 
	noting that $a+\frac{3}{2}-\tau > a + \frac{1}{2} -\tau \geq 0$, we find that   
	\begin{align*}
	\biggl[\int_0^t &|G(u)|^2 \rd u \biggr]^{\nicefrac{1}{2}} 
	= 
	\biggl[\int_0^t \,
	\biggl| \int_0^h 
	\| \Phi_{a,b+1}(u+r) \|_{ \mathscr L_2(H;\Hdot\sigma) } \rd r \biggr|^2 
	\rd u\biggr]^{\nicefrac{1}{2}} 
	\\
	&\leq 
	\int_0^h \biggl[ 
	\int_0^t 
	\| \Phi_{a,b+1}(u+r) \|_{\mathscr L_2(H;\Hdot\sigma)}^2 \rd u
	\biggr]^{\nicefrac{1}{2}} \rd r 
	\\
	&= 
	\int_0^h \biggl[ 
	\int_0^t (u+r)^{2a} 
	\bigl\| A^{a+\frac{3}{2}-\tau} S(u+r) 
	A^{ \frac{\sigma}{2}-a-\frac{1}{2}+b+\tau} Q^{\frac{1}{2}} \bigr\|_{ 
		\mathscr L_2(H)}^2 
	\rd u\biggr]^{\nicefrac{1}{2}} \rd r 
	\\
	&\lesssim_{(a,\tau)}
	\bigl\| A^{-a-\frac{1}{2}+b+\tau} Q^{\frac{1}{2}} \bigr\|_{ 
		\mathscr L_2(H;\Hdot\sigma)} 
	\int_0^h \biggl[ 
	\int_0^t 
	(u+r)^{2\tau-3} 
	\rd u\biggr]^{\nicefrac{1}{2}} \rd r 
	\\
	&\leq 
	\frac{1}{\tau \sqrt{2 - 2\tau}} \, 
	h^\tau \, 
	\bigl\| A^{-a-\frac{1}{2}+b+\tau} Q^{\frac{1}{2}} \bigr\|_{ 
		\mathscr L_2(H;\Hdot\sigma)} ,
	\end{align*}
	which completes the proof. 
\end{proof}

\begin{lemma}\label{lemma:I-2}
	Let 
	Assumptions~\textup{\ref{assumption:A}\ref{assumption:A:semigroup},\ref{assumption:A:bdd-analytic}} 
	be satisfied, let $a \in \bigl(-\frac{1}{2},\infty\bigr)$, 
	${b, \sigma \in [0,\infty)}$ 
	and $\tau\in \bigl( 0, 1 \wedge \bigl( a+\frac{1}{2} \bigr) \bigr]$. 
	If $\sigma \neq 0$, then suppose 
	furthermore that
	Assumption~\textup{\ref{assumption:A}\ref{assumption:A:bddly-inv}} 
	holds. 
	Let $J:=(0,T)$ 
	for some $T \in (0,\infty]$. 
	Then, 
	for all $p \in [1,\infty)$, 
	$t \in [0,T)$ and $h \in J$ 
	with $h\leq T-t$, 
	the function 
	$\Phi_{a,b}\from(0,\infty)\to\mathscr L(H; \Hdot{\sigma})$ 
	in \eqref{eq:def:Phi-ab} 
	satisfies 
	\[ 
	\biggl\| 
	\int_t^{t+h} \Phi_{a,b}(t+h-s) \rd W(s) 
	\biggr\|_{L^p(\Omega;\dot H_A^\sigma)} 
	\lesssim_{(p,a,\tau)}  
	h^\tau  
	\bigl\| A^{-a-\frac{1}{2}+b+\tau } 
	Q^{\frac{1}{2}}\bigr\|_{ \mathscr L_2(H;\Hdot \sigma) }. 
	\] 
\end{lemma}

\begin{proof} 
	We apply the Burkholder--Davis--Gundy inequality 
	(combined with nestedness of the 
	$L^p$ spaces if $p<2$), 
	the change of variables $u := t+h-s$,  
	and obtain  
	\begin{align*}
	\biggl\| 
	&\int_t^{t+h} \Phi_{a,b}(t+h-s) \rd W(s) 
	\biggr\|_{L^p(\Omega;\dot H_A^\sigma)}^2 
	\\
	&\lesssim_p 
	\int_t^{t+h} 
	\norm{\Phi_{a,b}(t+h-s)}{\mathscr L_2(H;\Hdot{\sigma})}^2 
	\rd s 
	= 
	\int_0^h \norm{\Phi_{a,b}(u)}{\mathscr L_2(H;\Hdot{\sigma})}^2 
	\rd u 
	\\
	&= 
	\int_0^h 
	u^{ 2a }
	\bigl\| A^{a+\frac{1}{2}-\tau} S(u) A^{\frac{\sigma}{2} -a -\frac{1}{2} + b +\tau }
	Q^{\frac{1}{2}} \bigr\|_{ \mathscr L_2(H)}^2 
	\rd u 
	\\
	&\lesssim_{(a,\tau)} 
	\bigl\| A^{ -a -\frac{1}{2} + b +\tau }Q^{\frac{1}{2}} \bigr\|_{ 
		\mathscr L_2(H; \dot H_A^\sigma)}^2 
	\int_0^h 
	u^{ 2\tau - 1}
	\rd u 
	=
	\frac{h^{ 2\tau }}{2\tau} \, 
	\bigl\| A^{-a -\frac{1}{2} + b +\tau }Q^{\frac{1}{2}} \bigr\|_{ 
		\mathscr L_2(H; \dot H_A^\sigma)}^2 ,
	\end{align*} 
	where we could proceed as in \eqref{eq:proof:emma:I-2:3}, 
	since $a+\frac{1}{2}-\tau \geq 0$ is assumed. 
	This completes the proof of the assertion. 
\end{proof}

Proposition~\ref{prop:Dh-converges-in-Lp} provides 
a useful relation between the weak derivative and the difference quotient.
\begin{proposition}\label{prop:Dh-converges-in-Lp} 
	Let $U$ be a real and separable Hilbert space
	and 
	let $J:=(0,T)$ 
	for some $T \in (0,\infty]$. 
	Suppose that $\Psi \in H^1(J; U)$ 
	and let $\Psi'\in L^2(J;U)$ denote the 
	weak derivative of $\Psi$.  
	For $h\in \R\setminus\{0\}$, let $J_h\subseteq J$ be as in
	Proposition~\textup{\ref{lem:translations}} and
	define the \emph{difference quotient} 
	$D_h \Psi \from J_h \to U$ 
	of $\Psi$ by 
	\begin{equation}\label{eq:def:Dh} 
	[D_h \Psi](t) 
	:= 
	\frac{ \Psi(t+h) - \Psi(t) }{h} 
	\quad 
	\text{for a.a.~$t \in J_h$}.
	\end{equation} 
	Then, we have 
	$\lim_{h\to 0} 
	\norm{D_h \Psi - \Psi'}{L^2(J_h;U)} 
	=0$. 
\end{proposition}

\begin{proof} 
	Suppose that 
	$\Psi \in E$, 
	where  the space $E$ is given by 
	$E:=C^\infty([0,T];U)$ if $T<\infty$ 
	and 
	$E:=C_c^\infty([0,\infty);U)$ if $T=\infty$, 
	and fix $h \in \R\setminus \{0\}$. 
	Then, 
	\begin{equation}\label{eq:FTC-Dh} 
	[D_h \Psi](t) 
	= 
	\frac{1}{h} 
	\int_{0}^h \Psi'(t+s) \rd s 
	\quad 
	\forall 
	t \in J_h
	\end{equation} 
	holds by the fundamental theorem of calculus, 
	where we use the convention that 
	$\int_0^h = -\int_h^0$ 
	whenever $h \in (-t, 0)$. 
	Applying the Cauchy--Schwarz inequality gives 
	\begin{align*}
	\norm{ [D_h \Psi](t) }{U}^2 
	\leq 
	\biggl| 
	\frac{1}{h} 
	\int_0^h \norm{ \Psi'(t+s) }{U} \rd s 
	\biggr|^2 
	&\leq 
	\biggl| 
	\frac{1}{h} 
	\int_0^h \norm{ \Psi'(t+s) }{U}^2 \rd s 
	\biggr| 
	\\ 
	&=  
	\frac{1}{h} 
	\int_0^h \norm{ \Psi'(t+s) }{U}^2 \rd s 
	\qquad 
	\forall t \in J_h. 
	\end{align*} 
	The absolute value can be removed 
	in the last step by the integral sign convention. 
	Integrating this expression 
	over $t \in J_h$ 
	and using Fubini's theorem, 
	we obtain that 
	\begin{equation}\label{eq:proof-Dh-integral-bound} 
	\begin{split} 
	\| D_h &\Psi \|_{L^2(J_h;U)}^2  
	=
	\int_{J_h} \norm{ [D_h \Psi](t) }{U}^2 \rd t 
	\\
	&\leq 
	\frac{1}{h} 
	\int_{J_h} \int_0^h \norm{ \Psi'(t+s) }{U}^2 \rd s \rd t 
	= 
	\frac{1}{h} 
	\int_0^h\int_{J_h} \norm{\Psi'(t+s)}{U}^2 \rd t \rd s.
	\end{split} 
	\end{equation}
	For all $s\in(0,h)$ (resp., $s\in (h,0)$ if $h<0$), 
	the change of variables $r := t+s$ gives  
	\[ 
	\int_{J_h} \norm{\Psi'(t+s)}{U}^2 \rd t 
	= 
	\int_{J_h + s} \norm{\Psi'(r)}{U}^2 \rd r 
	\leq 
	\int_J \norm{\Psi'(r)}{U}^2 \rd r 
	= 
	\norm{\Psi'}{L^2(J;U)}^2.
	\]
	Hence, we can bound the inner integral 
	in~\eqref{eq:proof-Dh-integral-bound} independently of $s$, 
	which implies 
	\begin{equation}\label{eq:DH-proof-h-uniform-bound} 
	\norm{D_h \Psi}{L^2(J_h;U)}^2 
	\leq 
	\norm{\Psi'}{L^2(J;U)}^2 
	\leq 
	\norm{\Psi}{H^1(J;U)}^2. 
	\end{equation}
	This estimate shows that the linear operator $D_h$ 
	is bounded from $\bigl( E, \norm{\,\cdot\,}{H^1(J;U)} \bigr)$ 
	to ${L^2(J_h;U)}$ for all $h \in \R\setminus\{0\}$. 
	By density of $E$ in $H^1(J;U)$ 
	(see \cite[XVIII.\S{1}.2, Lemma~1]{DautrayLionsVol5}), 
	the above estimate holds for all $\Psi \in H^1(J;U)$.
	
	Suppose again that 
	$\Psi\in E$. 
	We recall \eqref{eq:FTC-Dh} and find  
	\begin{equation}\label{eq:DH-proof-h-diff-equality}  
	[D_h \Psi](t) - \Psi'(t) 
	= 
	\frac{1}{h}\int_0^h 
	\bigl( \Psi'(t+s) - \Psi'(t) \bigr)\rd s 
	\quad 
	\forall \,
	t \in J_h.
	\end{equation}
	By the compact support  
	of $D_h \Psi$ and $\Psi'$, 
	there exists a bounded interval 
	$K \subset [0,\infty)$ 
	such that 
	$\operatorname{supp}(D_h \Psi - \Psi'|_{J_h}) \subseteq K$ 
	for all $h\in[-1,1]$. 
	Furthermore, 
	by uniform continuity of 
	$\Psi' \in C^\infty([0,T];U)$ 
	(resp., $\Psi' \in C_c^\infty([0,\infty);U)$), 
	for every $\varepsilon\in(0,\infty)$, 
	there exists 
	some $\delta\in(0,1)$ such that 
	$\norm{ \Psi'(\xi) - \Psi'(\eta)}{U} < \varepsilon$ 
	if $|\xi-\eta|<\delta$. 
	Thus, 
	\[ 
	\norm{[D_h \Psi](t) - \Psi'(t)}{U} 
	< 
	\varepsilon 
	\quad 
	\forall t \in J_h
	\] 
	follows 
	for all $h\in(-\delta,\delta)$ 
	by~\eqref{eq:DH-proof-h-diff-equality}
	and, consequently,
	\[ 
	\norm{D_h \Psi - \Psi'}{L^2(J_h;U)} 
	\lesssim_{K} 
	\norm{D_h \Psi - \Psi'}{L^\infty(J_h;U)} 
	\to 0 
	\quad \text{as $h \to 0$}.
	\] 
	This proves the assertion  
	for functions $\Psi\in E$. 
	The general case for $\Psi\in H^1(J;U)$ 
	follows then from density of $E$ and the $h$-uniform 
	bound~\eqref{eq:DH-proof-h-uniform-bound}:  
	Given $\varepsilon\in(0,\infty)$, 
	we may choose $v \in E$ 
	such that ${\norm{\Psi-v}{H^1(J;U)} < \frac{\varepsilon}{3}}$, 
	and $h_0 \in(0,\infty)$ such that 
	$\norm{D_h v - v'}{L^2(J_h;U)}<\frac{\varepsilon}{3}$ 
	for all $h\in(-h_0, h_0)$, 
	Thus, we obtain for all $h\in(-h_0, h_0)$ 
	\begin{align*}
	\| D_h  \Psi - \Psi' &\|_{L^2(J_h;U)} 
	\\
	&\leq 
	\norm{D_h (\Psi - v)}{L^2(J_h;U)}
	+ 
	\norm{D_h v - v'}{L^2(J_h;U)}
	+ 
	\norm{ v' - \Psi'}{L^2(J_h;U)} 
	\\
	&\leq 
	2\norm{\Psi - v}{H^1(J;U)} 
	+ 
	\norm{D_h v - v}{L^2(J_h;U)} 
	< \varepsilon . 
	\qedhere 
	\end{align*}
\end{proof}

\begin{lemma}\label{lem:embedding-weak-derivs}
	Let $J := (0,T)$ for some $T \in (0,\infty]$. 
	Let $E$ and $F$ be real separable Banach spaces 
	such that
	$E \hookrightarrow F$. 
	If $u \in H^1_{ 0,\{0\} }(J; F)$ and
	$u, u'\in L^2(J; E)$, where $u'$ denotes the
	$F$-valued weak derivative of $u$,
	then $u \in H^1_{ 0,\{0\} }(J; E)$ 
	and its $E$-valued weak derivative
	coincides with $u'$ 
	almost everywhere in $J$. 
\end{lemma}

\begin{proof}
	Let $\mathscr I_E \from L^1(J;E)\to E$ and 
	$\mathscr I_F \from L^1(J;F) \to F$ denote, respectively,
	the $E$-valued and $F$-valued Bochner integrals over 
	the interval $J$. 
	Given an arbitrary $\phi \in C_c^\infty(J)$, 
	the assumption
	$u \in H^1(J;F$)
	implies 
	$
	\mathscr I_F(\phi u') 
	= 
	-\mathscr I_F(\phi' u)
	$, and
	we wish to show
	$
	\mathscr I_E(\phi u') 
	= 
	-\mathscr I_E(\phi' u).
	$
	To this end, we claim that
	$\mathscr I_E$ and $\mathscr I_F$ coincide on
	$L^1(J;E) \hookrightarrow L^1(J;F)$
	and we apply this fact to 
	$\phi u'$ and $\phi' u$.
	To verify the claim, fix $f \in L^1(J;E)$.
	By definition
	of $\mathscr I_E$, there exist $E$-valued 
	measurable simple functions $(f_n)_{n\in\N}$
	satisfying $f_n \to f$ in $L^1(J;E)$ and 
	$\mathscr I_E(f_n) \to \mathscr I_E(f)$ in~$E$. For all $n \in \N$, 
	it readily follows from the
	respective definitions and the inclusion $E \subseteq F$ that
	$f_n$ is an $F$-valued measurable simple function and 
	$\mathscr I_F(f_n) = \mathscr I_E(f_n)$. 
	Since $E \hookrightarrow F$, we observe that
	$f_n \to f$ in $L^1(J;F)$ and 
	$\mathscr I_F(f_n) = \mathscr I_E(f_n) \to \mathscr I_E(f)$ in $F$, 
	hence $\mathscr I_F(f) = \mathscr I_E(f)$.
	We conclude that $u\in H^1(J;E)$ 
	and the $E$-valued weak derivative 
	coincides with $u'$ a.e.\ in $J$. 
	Now it remains to prove that $u \in H^1_{0,\{0\}}(J; E)$. 
	Note that
	$u \in H^1_{0,\{0\}}(J;F)$ is equivalent to 
	the statement that the unique continuous representative
	$\widetilde u \in C(\overline J; F)$ of $u$, 
	which exists 
	by virtue of \cite[\S5.9.2, Theorem~2]{Evans2010}, 
	vanishes at zero,
	cf.~\cite[\S5.5, Theorem~2]{Evans2010}.
	Similarly, from $u \in H^1(J; E)$ we obtain a function 
	$\widehat u \in C(\overline J; E) \hookrightarrow C(\overline J; F)$ 
	such that $u = \widehat u$ a.e.,
	hence $\widehat u = \widetilde u$ by uniqueness.
	In particular, $\widehat u(0) = 0$ and thus $u \in H_{0,\{0\}}(J;E)$. 
\end{proof}

\section{Sectorial linear operators and functional calculus}
\label{app:functional-calc} 

In this appendix, several definitions and results 
regarding sectorial linear operators, 
semigroups and functional calculus 
are recorded. 
We refer the reader to \cite{Engel1999,Haase2006,AnalysisInBanachSpacesII,Pazy1983} 
for more details on these topics. 

Throughout this section,
$A \from \mathsf D(A)\subseteq H \to H$ denotes a
linear operator whose negative $-A$ generates
a $C_0$-semigroup $(S(t))_{t\geq 0}$ on a 
separable Hilbert space $H$. 
The corresponding scalar field is given by the 
complex numbers $\C$ 
in Subsection~\ref{subsec:complex-func-calc} 
and the real numbers $\R$ 
in Subsection~\ref{subsec:functional-calc:real}. 

\subsection{Sectoriality and 
	\texorpdfstring{$H^\infty$-calculus}{H-infinity-calculus}}
\label{subsec:complex-func-calc}

Let $H$ be a Hilbert space 
over the \emph{complex} scalar field $\C$. 
\begin{definition}\label{def:sectorial}
	We say that $\lambda\in \C$ belongs to the 
	\emph{resolvent set} $\rho(A)$ if and only if
	$R(\lambda,A) := (\lambda I - A)^{-1}$ exists in $\mathscr L(H)$.
	The set $\sigma(A) := \C\setminus\rho(A)$ is called the \emph{spectrum}.
	$A$ is said to be \emph{sectorial} if there 
	exists an $\omega \in (0,\pi)$ such that
	\begin{equation}\label{eq:sectorial} 
	\sigma(A) 
	\subseteq 
	\clos \Sigma_\omega
	\quad 
	\text{ and } 
	\quad 
	\sup \{ 
	\norm{\lambda R(\lambda, A)}{\mathscr L(H)} 
	:
	\lambda \in \C\setminus\clos\Sigma_\omega
	\}
	< \infty, 
	\end{equation}
	where 
	$\Sigma_\omega:= \{\lambda\in\C\setminus \{0\}: \arg\lambda \in (-\omega,\omega)\}$. 
	The \emph{angle of sectoriality} $\omega(A)$ is defined as 
	the infimum of all $\omega$ for which \eqref{eq:sectorial} holds.
\end{definition}

\begin{theorem}\label{thm:sectoriality-analytic-semigroups}
	The operator $A$ is sectorial with $\omega(A)\in [0,\nicefrac{\pi}{2})$  
	if and only if
	$(S(t))_{t\ge0}$ is (uniformly) bounded analytic.
\end{theorem}

\begin{proof}
	The claim 
	follows from \cite[Theorem~G.5.2]{AnalysisInBanachSpacesII},
	noting that $\mathsf D(A)$ is dense in~$H$ 
	since $-A$ generates a $C_0$-semigroup,
	see~\cite[Theorem~II.1.4]{Engel1999}.
\end{proof}

Given $\varphi\in(0,\pi)$, 
we say that a holomorphic function
$f\from \Sigma_{\varphi}\to \C$
belongs to $H^\infty_0(\Sigma_\varphi)$ 
if and only if there exist constants $\alpha\in(0,\infty)$
and $M \in [0,\infty)$
such that
$|f(z)| \leq M \bigl(|z|^\alpha \wedge |z|^{-\alpha} \bigr)$ for
all $z \in \Sigma_\varphi$.
If $A$ is sectorial
and $\varphi \in (\omega(A), \pi)$, then 
\[
f(A)
:= 
\frac{1}{2\pi i}
\int_{\partial \Sigma_{\nu}} f(\zeta) R(\zeta,A) \rd \zeta, 
\qquad 
f \in H^\infty_0(\Sigma_\varphi),
\]
is well-defined, i.e.,
the $\mathscr L(H)$-valued Bochner integral 
is convergent and independent of $\nu \in (\omega(A),\varphi)$. 
We call the mapping $f \mapsto f(A)$ the \emph{Dunford calculus}
for $A$;
it is an algebra homomorphism 
from $H^\infty_0(\Sigma_\varphi)$ to $\mathscr L(H)$, 
see~\cite[Lemma~2.3.1(a)]{Haase2006}.

\begin{definition}\label{def:bdd-Hinfty-calc}
	Let $A$ be a sectorial operator and $\varphi \in (\omega(A),\pi)$. 
	Then $A$ is said to have a 
	\emph{bounded $H^\infty(\Sigma_\varphi)$-calculus} if
	there exists a constant $C \in (0,\infty)$ such that
	$\norm{f(A)}{\mathscr L(H)} 
	\leq  
	C
	\sup_{z\in \Sigma_\varphi}|f(z)|$
	for all $f \in H^\infty_0(\Sigma_\varphi)$.
	The angle $\omega_{H^\infty}(A)$ is defined as the infimum over 
	all admissable $\varphi \in (\omega(A),\pi)$ in the above definition.
\end{definition}

For operators acting on a (complex) Hilbert space, the admissibility of 
a bounded $H^\infty$-calculus
can be characterized by the following theorem. 
It is taken from~\cite[Theorem~7.3.1]{Haase2006}; 
see~\cite[Theorem~10.4.21]{AnalysisInBanachSpacesII} for a generalization
to non-injective $A$.

\begin{theorem}\label{thm:mcintosh}
	Let $A\from \mathsf D(A)\subseteq H \to H$ be injective and sectorial.
	Then
	\[
	\norm{x}{H}^2 
	\eqsim_f
	\int_0^\infty \norm{f(tA) x}{H}^2 \, \frac{\rd t}{t}
	\quad \forall x \in H
	\]
	holds for all non-zero
	$f \in \bigcup_{\varphi\in(\omega(A),\pi)} H_0^\infty(\Sigma_{\varphi})$
	if and only if
	$A$ admits a bounded $H^\infty(\Sigma_\varphi)$-calculus
	for some (or, equivalently, for all)
	$\varphi\in(\omega(A),\pi)$.
\end{theorem}

\begin{remark}\label{rem:Hinfty-sectoriality-angles}
	Since $\omega_{H^\infty}(A)$ is defined as an infimum
	over angles contained in the interval $(\omega(A),\pi)$, any
	operator admitting a bounded $H^\infty$-calculus satisfies
	$\omega_{H^\infty}(A) \ge \omega(A)$. This inequality is also true
	for operators on a Banach space, defined analogously to Definition~\ref{def:bdd-Hinfty-calc}.
	Theorem~\ref{thm:mcintosh} implies that reverse inequality holds
	for operators on a Hilbert space with a bounded $H^\infty(\Sigma_\varphi)$-calculus
	for some $\varphi\in(\omega(A),\pi)$. Indeed, in this case, the same holds
	for all $\varphi \in (\omega(A),\pi)$, hence 
	$\omega_{H^\infty}(A) \le \omega(A)$ upon taking the infimum.
	We thus have $\omega_{H^\infty}(A)=\omega(A)$.
\end{remark}

\subsection{Complexifications, semigroups and fractional powers}
\label{subsec:functional-calc:real}

In this subsection, 
$H$ denotes a \emph{real} Hilbert space. 

\subsubsection{Complexifications}\label{subsubsec:functional-calc:complexification} 

The complexified Hilbert space
$H_\C$ is defined by equipping the set $H\times H$
with
component-wise addition and the respective scalar and 
inner products
\begin{align}
(a+bi)(x,y) &:= (a x- by, bx + ay), 
&&x,y\in H;\; a,b\in \R, 
\notag \\
((x,y),(u,v))_{H_\C}
&:= 
(x,u)_H
+
(y,v)_H
+
i
[
(y,u)_H
-
(x,v)_H]
, &&x,y,u,v\in H.
\label{eq:inner-product-HC}
\end{align}
In the sequel, we will write $x + iy := (x,y) \in H_\C$. 

A linear operator $A$ on $H$ similarly gives
rise to a complexified counterpart $A_\C$ on $H_\C$
by defining 
$A_\C(x+iy) := Ax + i Ay$ on
$\mathsf D(A_\C) = \{x + iy : x, y \in \mathsf D(A)\}$. 
It
follows readily from the above definitions that
$T \mapsto T_\C \in \mathscr L(\mathscr L(H);\mathscr L(H_\C))$
is an inverse-preserving and isometric algebra homomorphism.
Analogous results hold for unbounded operators, 
taking natural domains into account. 
We have the following relation between semigroups and complexifications.

\begin{lemma}\label{lem:complexifications-semigroups}
	The family
	$(S(t))_{t\ge0} \subseteq \mathscr L(H)$ 
	is a $C_0$-semigroup on $H$ if and only if
	$(S_\C(t))_{t\ge0} \subseteq \mathscr L(H_\C)$ 
	is a $C_0$-semigroup on $H_\C$.
	In this case, their respective generators
	$-A \from \mathsf D(A) \subseteq H \to H$ and 
	$-\widehat A \from \mathsf D(\widehat A) \subseteq H_\C \to H_\C$ 
	satisfy $A_\C = \widehat A$.
\end{lemma}

\begin{proof}
	If $(S(t))_{t\ge0}$ is a $C_0$-semigroup, then clearly $S_\C(0) = I$ and $S_\C(t)S_\C(s) = [S(t)S(s)]_\C = S_\C(t+s)$ for $s,t\ge0$. Moreover,
	$\norm{S_\C(t)\widehat x - \widehat x}{H_\C}^2 
	=
	\norm{S(t)x-x}{H}^2 + \norm{S(t)y-y}{H}^2 \to 0
	$
	as $t \downarrow 0$
	for $\widehat x = x+iy\in H_\C$.
	The reverse implication is 
	readily established by identifying 
	every $x \in H$ 
	with $x + i0 \in H_\C$. 
	
	Suppose that
	$(S(t))_{t\ge0}$ 
	and 
	$(S_\C(t))_{t\ge0}$
	are $C_0$-semigroups with respective generators $-A$ and $-\widehat A$.
	Then 
	$\widehat x = x + iy \in \mathsf D(A_\C)$
	is equivalent to the existence
	of the limits 
	$-Ax = \lim_{t\downarrow 0} \frac{1}{t}(S(t)x - x)$
	and
	$-Ay = \lim_{t\downarrow 0} \frac{1}{t}(S(t)y - y)$
	in $H$. Thus,
	\[ 
	A_\C \, \widehat x
	= Ax + iAy
	= \lim_{t\downarrow 0} 
	\biggl[
	\frac{1}{t} \bigl(x - S(t)x \bigr)
	+ \frac{i}{t} \bigl( y - S(t)y \bigr) 
	\biggr]
	= \lim_{t\downarrow 0} \frac{1}{t} 
	\bigl( \widehat x - S_\C(t)\widehat x \bigr)
	= \widehat A \widehat x,
	\] 
	where the limits in the previous display are taken with respect to 
	$\|\cdot\|_{H_\C}$.
\end{proof}

\subsubsection{Fractional powers} 
\label{subsec:appendix-functional-calc:frac-powers}

Let $\alpha\in(0,\infty)$ be given. If
$(S(t))_{t\ge0}$ is exponentially stable, 
that is, \eqref{eq:exp-estimate-semigroup}  
holds for some $w\in(0,\infty)$, 
then we define negative fractional powers
of $A$ by 
\begin{equation}\label{eq:fractional-power-exp-bdd-semigroup}  
A^{-\alpha} := \frac{1}{\Gamma(\alpha)}\int_0^\infty t^{\alpha-1}S(t)\rd t,
\end{equation} 
non-negative powers by $A^\alpha := (A^{-\alpha})^{-1}$  
and $A^0:=I$.
By using Lemma~\ref{lem:complexifications-semigroups}
and interchanging the bounded operator $[\,\cdot\,]_\C$
with the Bochner integral, 
we find 
\begin{equation*}
[A^{-\alpha}]_\C 
= 
\frac{1}{\Gamma(\alpha)}\int_0^\infty t^{\alpha-1}S_\C(t)\rd t 
= 
A_\C^{-\alpha} \quad \forall \alpha\in(0,\infty),
\end{equation*}
and the same relation can be derived for arbitrary powers $\alpha\in\R$.
This definition of a fractional-order operator $A_\C^{\alpha}$ is adopted in
\cite[Chapter~2, Section~6]{Pazy1983} and is 
equivalent to the Dunford-type definition 
used in \cite{Haase2006}, 
see Corollary~3.3.6 therein.

\subsubsection{A square function estimate}\label{subsubsec:square-estimate}  

The following square function estimate is central to the proof of
Proposition~\ref{prop:easier-condition-H-infty}. 

\begin{lemma}\label{lem:square-func-est}
	Let $A$ satisfy 
	Assumptions~\textup{\ref{assumption:A}\ref{assumption:A:semigroup},\ref{assumption:A:bdd-Hinfty},\ref{assumption:A:bddly-inv}}.
	Then, for $a \in (0,\infty)$,
	\[
	\int_0^\infty 
	\bigl\| t^{a-\frac{1}{2}}A^a S(t) x \bigr\|_{H}^2 \rd t 
	\eqsim_{a}
	\norm{x}{H}^2 
	\quad \forall x \in H.
	\]
\end{lemma} 

\begin{proof}
	Given $a \in(0,\infty)$ 
	and $\varphi\in (\omega(A),\nicefrac{\pi}{2})$, 
	the function
	$f(z) := z^a e^{-z}$ belongs to 
	$H^\infty_0(\Sigma_{\varphi})$ 
	and we have the identity 
	$f(t A_\C)
	=
	t^a A_\C^a S_\C(t)
	=
	[t^a A^a S(t)]_\C$; 
	see the proof of \cite[Proposition~3.4.3]{Haase2006},
	which is applicable to our definition of fractional powers 
	as remarked
	in Subsection~\ref{subsec:appendix-functional-calc:frac-powers}. 
	By invoking Theorem~\ref{thm:mcintosh}, we thus find 
	\[
	\int_0^\infty 
	\bigl\| [t^a A^a S(t)]_\C \, x \bigr\|_{H_\C}^2 \, \frac{\rd t}{t} 
	= \int_0^\infty 
	\bigl\| f(tA_\C) x \bigr\|_{H_\C}^2 \, \frac{\rd t}{t} 
	\eqsim_{a}
	\| x \|_{H_\C}^2 
	\quad 
	\forall x \in H_\C.
	\]
	Applying this equivalence to $x + i0$ for all $x \in H$ finishes the proof.
\end{proof}


\section*{Acknowledgments}

K.K.\ thanks David Bolin and 
Mih\'aly Kov\'acs for several valuable discussions 
which led to the choice and the interpretation of 
the SPDE models \eqref{eq:spde-fractional-parabolic} and 
\eqref{eq:fractional-parabolic-spde-zeroIC}. 
Furthermore, the authors thank 
Mark Veraar for pointing out the 
square function estimate 
of Lemma~\ref{lem:square-func-est} 
needed for the simplified conditions 
in Proposition~\ref{prop:easier-condition-H-infty}, 
and an anonymous reviewer
for valuable comments. 

K.K.\ acknowledges support 
of the research project 
\emph{Efficient spatiotemporal statistical modelling 
	with stochastic PDEs} 
(with project number  
VI.Veni.212.021) 
by the talent programme \emph{Veni}  
which is financed by 
the Dutch Research Council (NWO). 

\bibliographystyle{siam}
\bibliography{kw-bib.bib}

\vspace{-0.15\baselineskip}

\end{document}